\documentclass[12pt,reqno]{amsart}

\usepackage[english]{babel}

\usepackage{caption}

\usepackage[usenames, dvipsnames, svgnames]{xcolor}
\usepackage{amsmath, amssymb,graphicx,amsthm,latexsym, amsfonts, enumitem, mathtools, tensor, MnSymbol}
\usepackage{hyperref,subfigure}
\usepackage[all, color]{xy}
\usepackage{color}
\usepackage{amssymb, wasysym}
\usepackage{float}
\usepackage{tikz}
\usepackage{tikz-cd}
\usepackage{adjustbox}
\usetikzlibrary{arrows,decorations.pathmorphing,decorations.pathreplacing,positioning,shapes.geometric,shapes.misc,decorations.markings,decorations.fractals,calc,patterns,}
\usepackage{circuitikz}

\usepackage{soul}
\usepackage[normalem]{ulem}
\DeclareMathOperator{\cross }{cross \,}
\DeclareMathOperator{\wt }{wt \,}
\DeclareMathOperator{\module}{mod}

\newcommand{\sg}{\mathcal{G}_{\gamma}}

\newcommand{\La}{\Lambda}
\newcommand{\Ltil}{\widetilde{\Lambda}}

\newcommand{\rectanglepath}{-- +(1cm,0cm)  -- +(1cm,1cm)  -- +(0cm,1cm) -- cycle}

\newcommand{\Loop}{\rotatebox[origin=c]{180}{$\circlearrowright$}}

\newcommand{\ou}[2]{\overset{\text{\large ${#1}$}}{#2}}

\theoremstyle{plain}
\newtheorem{theorem}{Theorem}[section]
\newtheorem*{theorem*}{Theorem}

\theoremstyle{definition}
\newtheorem{definition}[theorem]{Definition}

\newtheorem{example}[theorem]{Example} 
\newtheorem{remark}[theorem]{Remark}
\newtheorem{remarks}[theorem]{Remarks}

\newtheorem{lemma}[theorem]{Lemma}
\newtheorem{notation}[theorem]{Notation}
\newtheorem{corollary}[theorem]{Corollary}

\newtheorem{proposition}[theorem]{Proposition}

\usepackage{color}

\date{}


\AtEndDocument{\bigskip{\footnotesize%

  \textsc{VU Amsterdam, Department of Mathematics, De Boelelaan 1111, 1081 HV Amsterdam, The Netherlands} \par
  \textit{E-mail address: i.canakci@vu.nl} 
  \vspace{8pt}
  
  \textsc{School of Mathematics, University of Leeds, Leeds, LS2 9JT, United Kingdom} \par
  \textit{E-mail address: f.fedele@leeds.ac.uk} 
 \vspace{8pt}
 
 \textsc{CEMIM Fac. de Cs. Exactas y Naturales UNMdP and CONICET, Dean Funes 3350, Mar del Plata, 7600, Argentina} \par
  \textit{E-mail address:} \textit{elsener@mdp.edu.ar} \par
  \vspace{8pt}
  
\textsc{University of Kentucky, Lexington, Department of Mathematics, 951 Patterson Office Tower, Lexington, KY 40506-0027, USA} \par
  \textit{E-mail address: khrystyna.serhiyenko@uky.edu} 
}}

\begin{document}
\setlength{\parindent}{0pt}
\setlength{\parskip}{7pt}

\title{Super Caldero--Chapoton map for type $A$}
\author{\.{I}lke \c{C}anak\c{c}\i, Francesca Fedele, Ana Garcia Elsener, Khrystyna Serhiyenko}

\maketitle

\begin{abstract}
One can explicitly compute the generators of a surface cluster algebra either combinatorially, through dimer covers of snake graphs, or homologically, through the CC-map applied to indecomposable modules over the appropriate algebra.
Recent work by Musiker, Ovenhouse and Zhang used Penner and Zeitlin's  decorated super Teichm{\"u}ller theory to define a super version of the cluster algebra of type $A$ and gave a combinatorial formula to compute the even generators. We extend this theory by giving a homological way of explicitly computing these generators by defining a super CC-map for type $A$.
\end{abstract}

\section{Introduction}

Cluster algebras were defined by Fomin and Zelevinsky in the context of Lie theory \cite{fz}, and their connection with representation theory was established soon after their definition, provoking an intense collaboration on both sides of the respective theories \cite{bmrrt,BMR,CCS,kel13}. The link between representation theory and cluster algebras is given by certain maps called cluster character maps. They associate to each module over a certain associative algebra (or object over a triangulated category) an element in the cluster algebra. The first appearance of these maps occurred in \cite{Cch}, hence they are also widely known as Caldero-Chapoton maps.

Cluster algebras from marked surfaces were defined by Fomin, Shapiro and Thurston in \cite{FST08}. In this setting, elements in the cluster algebra are defined by lambda lengths of certain arcs. The concept of lambda length arises from decorated Teichm{\"u}ller theory, where a marked surface can be endowed with a hyperbolic metric having a cusp at each marked point. After choosing a horocycle at every marked point, each
arc can be assigned a number called a lambda length, see \cite{penner1987}. 
Given an initial triangulation,  iteratively applying the ``generalised Ptolemy relation'', the lambda length of an arc can be expressed as a rational function of the lambda lengths of the arcs in said triangulation. Snake graphs appeared in \cite{MS10},\cite{MSW11} as the key element in combinatorial formulas used to obtain cluster algebra elements, i.e. lambda lengths, associated to arcs in a triangulated surface. 

Hence, in the classical setting for cluster algebras arising from surfaces, cluster algebra elements can either be defined  exploiting snake graph combinatorics or can be obtained from modules using Caldero-Chapoton maps.

Recently Musiker, Ovenhouse and Zhang \cite{musiker2021expansion,musiker21} defined  a super algebra arising from decorated super Teichm{\"u}ller theory \cite{bous13,penner19}. These super algebras are generated by even variables, associated to super lambda lengths, and odd variables which anticommute with each other and commute with the even ones. The geometric model for these algebras consists of an oriented triangulation of a disk. The initial even variables are in bijection with the arcs of the triangulation and the remaining super variables are in bijection with the remaining arcs. Moreover, the initial odd variables are associated to each triangle of the triangulation.
The authors show how, as in the classic case, these super lambda lengths, that occur as rational functions on the even variables and their square roots and odd variables, can be computed combinatorially using double dimer covers of snake graphs.

In this article we give a representation theoretic interpretation for the super algebras of type $A$ studied by Musiker, Ovenhouse and Zhang.

We first define an algebra $\Ltil$ obtained by tensoring a gentle algebra $\Lambda$ with the dual numbers, that is $\Ltil \colon = \La \otimes_K K[\epsilon]/(\epsilon^2)$,  where $K$ is the underlying field. In particular we are interested in  \textit{induced modules}, that is modules in $\module\Ltil$ of the form $\widetilde{M}:= M\otimes_K K[\epsilon]/(\epsilon^2)$ such that $M$ is in $\module\La$. 

 We consider 
 a string module $M_\mathcal{G}$ in $\module\La$  corresponding to a snake graph $\mathcal{G}$, and a double dimer cover of $\mathcal{G}$, i.e. a multiset of edges obtained by superimposing two perfect matchings of $\mathcal{G}$. Using this notation, we prove the following.

\textbf{Theorem (Theorem~\ref{thm_lattice_bijection_double}).} \textit{The lattice of the double dimer covers of $\mathcal{G}$ is in bijection with the submodule lattice of $\widetilde M_\mathcal{G}$.}

We then specialise to a gentle algebra type $A$ to study the super algebra from Musiker, Ovenhouse and Zhang.
We construct a super Caldero-Chapoton map from the induced modules to the set of super lambda lengths.

\textbf{Theorem (Theorem~\ref{thm_superCC} and Definition~\ref{defn_superCC}).}
\textit{Let $\Ltil=\Lambda\otimes_K K[\epsilon]/(\epsilon^2)$ where $\Lambda$ is a Jacobian algebra coming from a triangulation (with no internal triangles) of an $(n+3)$-gon. For an arc $\gamma$ in the polygon, let $M_\gamma$ be the corresponding indecomposable in $\mathrm{mod} \,\Lambda$. Then, the corresponding super lambda length is
      \begin{align*}
         CC(\widetilde{M_\gamma})= x_\gamma  
&= X^{\mathrm{ind}_{\widetilde{\Lambda}} (\widetilde{M}_\gamma)  } \sum\limits_{\mathbf{e}\, \in \mathbb{Z}^n }\chi (\mathrm{Gr}_{\mathbf{e}}( \widetilde{M}_\gamma ) )  \prod\limits_{i=1}^n \sqrt{x_i}^{\langle S_i , \oplus_j S_j^{m_j} \rangle_{\widetilde{\Lambda}}}\mu_{\mathbf{e}},
      \end{align*}
where $\mathbf{e}  = \underline{\mathrm{dim}} ( \bigoplus_j  S_j^{m_j})$, $\langle - , - \rangle_{\Ltil}$ is the antisymmetrized bilinear form from Definition \ref{def:bilinear form}  and $\mu_\mathbf{e}$ is as in Notation~\ref{notation_mu}. Moreover, for $E=\oplus_{i=1}^r E_i$, where each $E_i$ is either an indecomposable induced module in $\module\Ltil$ or a shifted projective of the form $P_j[1]$, we define $CC(E)=\prod_{i=1}^r CC(E_i)$, where $CC(P_j[1]):=x_j$.}

In the above, almost all of the terms resemble the ones appearing in the classic Caldero-Chapoton map. Apart from the appearance of square roots, the only surprising term is $\mu_{\mathbf{e}}$, which is the term associating the correct product of odd variables to a given vector $\mathbf{e}$. Moreover, we show in Remark~\ref{remark_simplified_form} that the above formula can be rewritten to reduce most of the calculations to calculations over the algebra $\La$.
Since the super CC-map recovers the combinatorial formula for the super lambda lengths, we see in Corollary~\ref{corollary_super_CC_ptolemy} that the super CC-map respects the super Ptolemy relations.

The paper is organised as follows. In Section~\ref{sec-cluster algebras from surfaces}, we first recall some theory on cluster algebras from surfaces and the combinatorial approach for computing cluster variables and then give an overview of decorated super Teichm{\"u}ller theory and Musiker, Ovenhouse and Zhang super algebra. In Section~\ref{Sec:RepT-SG}, we recall the homological approach to cluster algebras, the classic CC-map and the Frobenius category. In Section~\ref{Sec:tensoring}, we set the basis for the homological approach to the super algebras, by introducing the algebra $\widetilde{\Lambda}$, induced modules over it and establishing further their properties. In Section~\ref{Sec:latticebij}, we prove the lattice bijection. Finally, in Section~\ref{Sec:superCC}, we prove our main result, constructing a super version of the CC-map for type $A$.

\emph{Acknowledgments.} The authors would like to thank Gregg Musiker, Nick Ovenhouse and Sylvester Zhang, as well as Pierre-guy Plamondon for helpful discussions. 
 This work originates as part of the WINART3 Workshop at the
Banff International Research Station. The authors thank BIRS and the organizers of the workshop.
The authors would like to thank the Isaac Newton Institute for Mathematical Sciences, Cambridge, for support and hospitality during the
programme “Cluster algebras and representation theory” (supported by EPSRC grant no. EP/R014604/1) where the work on this paper
was started. 
F.F. was supported by project REDCOM funded by Fondazione Cariverona - program ``Ricerca Scientifica di Eccellenza 2018'' and the ESPRC Programme Grant EP/W007509/1.
 A.G.E. was partially supported by PICT 2021-I–A-01154 Agencia Nacional de Promoción Científica y Tecnológica and by the EPSRC grant EP/R009325/1.
K.S. was supported by the NSF grant DMS-2054255.


\section{Cluster algebras from surfaces}\label{sec-cluster algebras from surfaces}

In this section we will introduce  cluster algebras from unpunctured surfaces. In this setting, a triangulation corresponds to an initial cluster and each arc to a cluster variable. Moreover, there exists an explicit combinatorial formula which computes the cluster variable associated to each arc that is not in the initial triangulation. The exposition follows  \cite{FST08,MSW11,c13}.

\subsection{Cluster algebras from marked surfaces}\label{subsec-cluster algebras marked surfaces}

Cluster algebras from marked surfaces were introduced in~\cite{FST08,FT18}.  In these articles, the authors establish a one-to-one correspondence between lambda-lengths of arcs and cluster variables. A triangulation represents an initial cluster and mutations can be interpreted as flips of arcs~\cite[Proposition 7.6]{FT18}.  
Fixing a triangulation of a surface, the collection of lambda-lengths corresponding to the
arcs in the triangulation (including boundary segments) forms a system of coordinates
for the decorated Teichm{\"u}ller space \cite{penner1987} such that all boundary segments are set to 1. Choosing another triangulation gives rise to a different coordinate chart, but
all the triangulations for a fixed surface are related by sequences of flips, and the cluster variables in the adjacent clusters are related by Ptolemy relations (see Figure \ref{fig:ptolemy}).

\begin{figure}[h]
\centering
 \begin{tikzpicture}
 \begin{scope}
   \node[draw,minimum size=3cm,regular polygon,regular polygon sides=3] (a) at (0,0) {};
  \draw (a.side 3) node[right]{$b$};
  \draw (a.side 2) node[below]{$e$};
  \draw (a.side 1) node[left]{$a$};
  \node[rotate=180, draw,minimum size=3cm,regular polygon,regular polygon sides=3,anchor= side 2] (b) at (a.side 2) {};
  \draw (b.side 3) node[left]{$d$};
  \draw (b.side 1) node[right]{$c$};
 \end{scope}
 \begin{scope}[xshift=6cm]
 \node[draw=none,minimum size=3cm,regular polygon,regular polygon sides=3] (c) {};
  \draw (c.side 3) node[right]{$b$};
  \draw (c.side 1) node[left]{$a$};
  \node[rotate=180,draw=none,minimum size=3cm,regular polygon,regular polygon sides=3,anchor= side 2] (d) at (c.side 2) {};
  \draw (d.side 3) node[left]{$d$};
  \draw (d.side 1) node[right]{$c$};
  \draw (c.corner 1)--(c.corner 2);
  \draw (c.corner 1)--(c.corner 3);
  \draw (d.corner 1)--(d.corner 2);
  \draw (d.corner 1)--(d.corner 3);
  \draw (c.corner 1)--node[right] {$f$}(d.corner 1);
  \end{scope}
  \draw[thick, ->] (2.3,-0.8) -- (3.8,-0.8);
 \end{tikzpicture}
 \caption{Ptolemy transformation: $ef=ac+bd$.}
 \label{fig:ptolemy}
\end{figure}
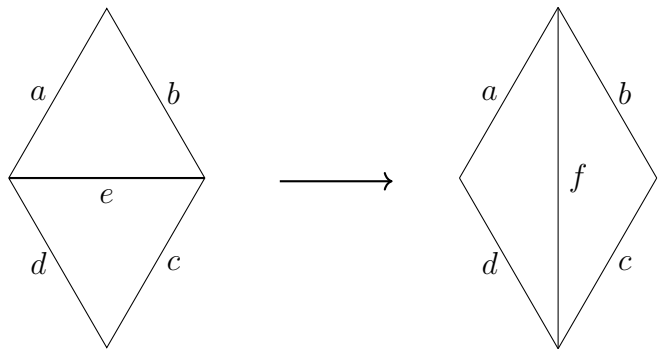

\begin{definition}\label{def:marked surface-arcs-triangulation}  Let $S$ denote an orientable compact surface, with a non-empty boundary denoted by $\partial S$. Let $M  \subset \partial S$ be a finite subset of points such that each boundary component contains at least one point in $M$. The elements of $M$ are called \emph{marked points}, and the pair $(S,M)$ is a \emph{bordered marked surface}.  If $M\subset \partial S$, then $(S,M)$ is called \emph{unpunctured} and it is called \emph{punctured} otherwise. An \emph{arc} in $(S,M)$ is a curve $\gamma$ in $S$ such that its endpoints are marked points, and it is disjoint from $M$ and $\partial S$ otherwise. An arc is considered up to
isotopy relative to its endpoints. We require that arcs do not self-cross, except possibly at the endpoints, and that they are not contractible. Given a marked surface $(S,M)$, a \emph{triangulation} $T$ of $(S,M)$ is a maximal set of non-crossing arcs. 
\end{definition}

In this paper we will  consider cluster algebras (with trivial coefficients) associated to unpunctured marked surfaces. An initial seed $(\mathbf{x}_T, Q_T)$ for the cluster algebra is given by a triangulation $T$ of $(S,M)$ as in Definition~\ref{def:initial seed-cluster algebra}

\begin{definition}\label{def:initial seed-cluster algebra}  Let $T = \{ \gamma_1 , \ldots, \gamma_n \}$ be a triangulation of  a marked surface $(S,M)$. The \emph{adjacency} quiver $Q_T$  associated to the triangulation $T$ is given as follows:
\begin{enumerate}
    \item for each $\gamma_i\in T$, we associate a vertex $i$ in $Q_T$,
    \item for all $\gamma_i, \gamma_j \in T$ that are adjacent in a  triangle, we associate an arrow from $i$ to $j$ if the angle between $\gamma_i$ and $\gamma_j$ in $S$ is clockwise from $\gamma_i$ to $\gamma_j$. 
\end{enumerate}
Furthermore, for each $\gamma_i \in T$, we associate a variable $x_i$ and we let $\mathbf{x}_T=\{x_1,\ldots,x_n\}$.
Then $(\mathbf{x}_T, Q_T)$ is an initial seed for the cluster algebra $\mathcal{A}(S,M)$ associated to the marked surface $(S,M)$.  Moreover, the generating set of $\mathcal{A}(S,M)$  consisting of cluster variables is in one-to-one correspondence with the arcs in $(S,M)$.

Starting from the initial seed $(\mathbf{x}_T, Q_T)$, the expression for the \emph{cluster variable} $x_{\gamma}$ is obtained by recursively applying the Ptolemy relations $x_ex_f=x_ax_c+x_bx_d$ (see Figure~\ref{fig:ptolemy}) and setting all boundary segments to $1$.  Here we will abuse notation by referring to an arc and its corresponding lambda length (see \cite{FT18}) and simply write $ef=ac+bd$ for a Ptolemy relation.

The Ptolemy relations encode the Fomin-Zelevinsky mutations in the surface \cite{FST08} and cluster variables turn out to be Laurent polynomials in the initial seed \cite{fz}. 

The \emph{cluster algebra} $\mathcal{A}(S,M)$ is the subalgebra of $\mathbb{Q}(x_1, \ldots, x_n)$ generated by the cluster variables $x_{\gamma}$ (associated to each arc $\gamma$).   
\end{definition}

\begin{example}\label{example:small example}  Let $T$ be the triangulation of the octagon in Figure~\ref{fig:example-type-A}. Then $Q_T$ is the adjacency quiver of $T$. Each arc $\gamma_i \in T$ corresponds to an initial cluster variable $x_i$. So $T$  gives rise to the initial seed $(\mathbf{x}_T, Q_T)=(\{ x_1, \ldots, x_5\}, \xymatrix{ 1& 2\ar[l]&3\ar[r]\ar[l]& 4&5\ar[l]} )$.

\begin{figure}[h!]
    \centering
\begin{tikzpicture}
\begin{scope}
   \node[draw,minimum size=4.5cm,regular polygon,regular polygon sides=8] (a) at (0,0) {};
  \draw [blue] (a.corner 1)--node(m)[left] {$4$} (a.corner 6);
  \draw [blue] (a.corner 1)--node(n)[right] {$5$} (a.corner 7);
  \draw [blue] (a.corner 2)--node(o)[left] {$3$} (a.corner 6);
  \draw [blue] (a.corner 2)--node(p)[below left] {$2$} (a.corner 5);
  \draw [blue] (a.corner 2)--node(q)[left] {$1$} (a.corner 4);
  \draw (a.side 1) node[above]{$b_7$};
  \draw (a.side 2) node[left=1pt]{$b_8$};
  \draw (a.side 3) node[left]{$b_1$};
  \draw (a.side 4) node[left=1pt]{$b_2$};
  \draw (a.side 5) node[below]{$b_3$};
  \draw (a.side 6) node[right=1pt]{$b_4$};
  \draw (a.side 7) node[right]{$b_5$};
  \draw (a.side 8) node[right=1pt]{$b_6$};
  \draw (a.side 2) node[left=40pt, blue]{$T$};
  \draw [-stealth, red, thick] (n)--(m);
  \draw [-stealth, red, thick] (o)--(m);
  \draw [-stealth, red, thick] (o)--(p);
  \draw [-stealth, red, thick] (p)--(q);
 \end{scope}
 \begin{scope}[xshift=7cm]
\draw (0,0)  node{\xymatrix{Q_T:  1& 2\ar[l]&3\ar[r]\ar[l]& 4&5\ar[l]}};
 \end{scope}
 \end{tikzpicture}
 \caption{A triangulation $T$ of the octagon and its associated quiver $Q_T$, highlighted in red in the figure. The cluster algebra $\mathcal{A}(S,M)$ has initial seed $(\mathbf{x}_T, Q_T)=(\{ x_1, \ldots, x_5 \}, 1 \leftarrow 2\leftarrow 3 \to 4 \leftarrow 5 )$.}
\label{fig:example-type-A}
\end{figure}
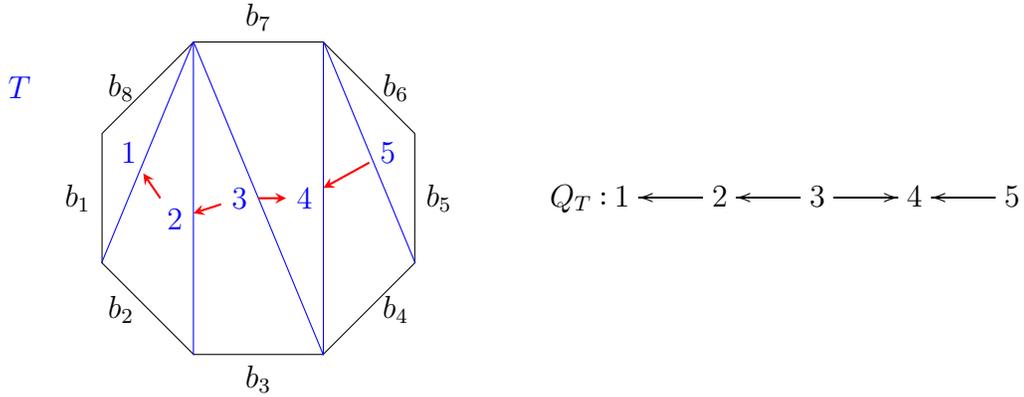
\end{example}

 Since the generating set for the cluster algebra is given by an iterative process, finding explicit formulae for the cluster variables in terms of an initial seed is a difficult question in general. In the surface case, this can be described combinatorially via snake graphs or homologically by the CC-map. We will review these two approaches, the former in this section and the  latter in Section~\ref{Sec:RepT-SG}.

\subsubsection{Snake graph formula}\label{subsec:snake-graph-single-dimer}
Snake graph formulae were introduced by Musiker, Schiffler \cite{MS10} and used by them in \cite{MSW11} to prove  structural properties for cluster algebras from surfaces.

We give an overview of the definition for unpunctured surfaces.

Let $T$ be a triangulation of an unpunctured marked surface $(S,M)$ and let $\gamma$ be an (oriented) arc  not in $T$, with starting point $s$ and ending point $t$. Then $\gamma$ crosses (not necessarily distinct) arcs $\gamma_{1}, \ldots, \gamma_{d} \in T$, in this order. The snake graph $\mathcal{G}_\gamma$ associated to $\gamma$ consists of one tile $G_i$ for each arc $\gamma_i$ and it is constructed as follows. For each $i$, the diagonal $\gamma_i$ is the common side of two triangles of the triangulation, and is thus the diagonal of a quadrilateral. The square tile $G_i$ has four sides labelled the same as the sides of this quadrilateral.
If $i$ is even, then the orientation of $G_i$ matches the one in the triangulation, and if $i$ is odd, then the orientation of $G_i$ is reversed. The edges shared by two adjacent tiles are called \textit{interior edges} and the remaining edges are called \textit{boundary edges}.

A snake graph $\mathcal{G}$ is a planar graph consisting of a finite sequence of square tiles $G_1, \ldots, G_d$ such that each tile is attached to the East or North edge of the previous one.

The reader may find the complete definition of snake graph for cluster algebras associated to punctured surfaces with principal coefficients in \cite[Section 4.3]{MSW11} and may see  the definition and examples in the introductory notes \cite{S18-notes}. We illustrate this construction for the triangulated polygon from Example~\ref{example:small example}.
 
\begin{example}  Consider the arc $\gamma$ in the
the triangulation of the octagon given in Figure \ref{fig:snake-1}. The snake graph $\mathcal{G}_\gamma$ associated to the arc $\gamma$ is shown in Figure \ref{fig:snake-1}, where the face weight $i$ of the tile corresponds to the quadrilateral with diagonal $i$.
\begin{figure}[h!]
\centering
\begin{subfigure}
    \centering
    \begin{tikzpicture}[scale=1]
   \node[draw,minimum size=4.5cm,regular polygon,regular polygon sides=8] (a) at (0,0) {};
  \draw  (a.corner 1)--node(m)[left] {$4$} (a.corner 6);
  \draw  (a.corner 1)--node(n)[right] {$5$} (a.corner 7);
  \draw  (a.corner 2)--node(o)[left] {$3$} (a.corner 6);
  \draw  (a.corner 2)--node(p)[below left] {$2$} (a.corner 5);
  \draw  (a.corner 2)--node(q)[left] {$1$} (a.corner 4);
  \draw [color=red, thick] (a.corner 3)-- node[right=2pt] {$\gamma$} (a.corner 7);
  \draw (a.corner 3) node[left, red]{$s$};
  \draw (a.corner 7) node[right, red]{$t$};
  \draw (a.side 1) node[above]{$b_7$};
  \draw (a.side 2) node[left=1pt]{$b_8$};
  \draw (a.side 3) node[left]{$b_1$};
  \draw (a.side 4) node[left=1pt]{$b_2$};
  \draw (a.side 5) node[below]{$b_3$};
  \draw (a.side 6) node[right=1pt]{$b_4$};
  \draw (a.side 7) node[right]{$b_5$};
  \draw (a.side 8) node[right=1pt]{$b_6$};
    \end{tikzpicture}
\end{subfigure}
\qquad
    \begin{subfigure}
    \centering
\begin{tikzpicture}[scale=1,yshift=4cm]
  \draw (0,0) \rectanglepath;
  \draw (0,1) \rectanglepath;
  \draw (1,1) \rectanglepath;
  \draw (2,1) \rectanglepath;
  \draw (-0.15,0.5) node{$\scriptstyle b_1$};
  \draw (0.5,-0.15) node{$\scriptstyle b_8$};
  \draw (0.5,2.15) node{$\scriptstyle 3$};
  \draw (-0.15,1.5) node{$\scriptstyle 1$};
  \draw (1.15,0.5) node{$\scriptstyle 2$};
  \draw (0.85,1.5) node{$\scriptstyle b_3$};
  \draw (0.5,1.15) node{$\scriptstyle b_2$};
  \draw (1.5,0.85) node{$\scriptstyle 2$};
  \draw (1.5,2.15) node{$\scriptstyle 4$};
  \draw (2.5,2.15) node{$\scriptstyle 5$};
  \draw (2.2,1.5) node{$\scriptstyle b_7$};
  \draw (2.5,0.85) node{$\scriptstyle 3$};
  \draw (3.2,1.5) node{$\scriptstyle b_4$};
  \draw (0.5,0.5) node{$\huge 1$};
  \draw (0.5,1.5) node{$\huge 2$};
  \draw (1.5,1.5) node{$\huge 3$};
  \draw (2.5,1.5) node{$\huge 4$};
\end{tikzpicture}
\end{subfigure}
\begin{subfigure}
\centering
    \begin{tikzpicture}[scale=1]
  \draw (0,0) \rectanglepath;
  \draw (0,1) \rectanglepath;
  \draw (1,1) \rectanglepath;
  \draw (2,1) \rectanglepath;
  \draw (-0.15,0.5) node{$\scriptstyle b_1$};
  \draw (0.5,-0.15) node{$\scriptstyle b_8$};
  \draw (0.5,2.15) node{$\scriptstyle 3$};
  \draw (-0.15,1.5) node{$\scriptstyle 1$};
  \draw (1.15,0.5) node{$\scriptstyle 2$};
  \draw (0.85,1.5) node{$\scriptstyle b_3$};
  \draw (0.5,1.15) node{$\scriptstyle b_2$};
  \draw (1.5,0.85) node{$\scriptstyle 2$};
  \draw (1.5,2.15) node{$\scriptstyle 4$};
  \draw (2.5,2.15) node{$\scriptstyle 5$};
  \draw (2.2,1.5) node{$\scriptstyle b_7$};
  \draw (2.5,0.85) node{$\scriptstyle 3$};
  \draw (3.2,1.5) node{$\scriptstyle b_4$};
  \draw (0.5,0.5) node{$\huge 1$};
  \draw (0.5,1.5) node{$\huge 2$};
  \draw (1.5,1.5) node{$\huge 3$};
  \draw (2.5,1.5) node{$\huge 4$};
  \draw[ultra thick, color=blue] (0,0)--(0,1);
  \draw[ultra thick, color=blue] (0,2)--(1,2);
  \draw[ultra thick, color=blue] (1,1)--(1,0);
  \draw[ultra thick, color=blue] (2,1)--(2,2);
  \draw[ultra thick, color=blue] (3,1)--(3,2);
    \end{tikzpicture}
\end{subfigure}
\caption{On the left a diagonal $\gamma$ on a triangualted octagon, in the middle is its snake graph $\mathcal{G}_\gamma$ and on the right is one of its dimer covers $P$ highlighted in thick blue.}
\label{fig:snake-1}
\end{figure}
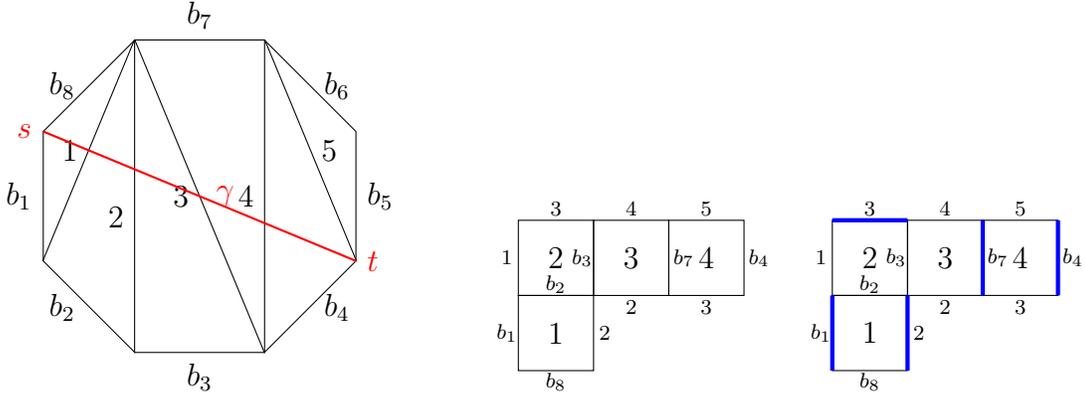
\end{example}

\begin{definition}\label{definition:perfect_matching}
A \textit{perfect matching} or \textit{dimer cover} of a graph $\mathcal{G}$ is a subset $P$ of the edges of $\mathcal{G}$ such that each vertex of $\mathcal{G}$ is incident to exactly one edge in $P$.
\end{definition}

Each snake graph has exactly two special  dimer covers consisting  only of boundary edges called \textit{boundary dimer covers}.

\begin{definition}\label{defn_minimal_matching}
Let $\mathcal{G}$ be a snake graph. Its \textit{minimal dimer cover} $P_{\textup{min}}$ is the boundary dimer cover containing the West edge of the initial tile $G_1$ and its \textit{maximal dimer cover} $P_{\textup{max}}$ is the complementary boundary dimer  cover to $P_{\textup{min}}$. 
\end{definition}

Observe that if a dimer cover $P$ on a snake graph $\mathcal{G}$ has a tile $G_i$ which has two of its edges in $P$ then these edges are either South and North or West and East edges of $G_i$. If we replace South and North edges with West and East edges or vice versa we obtain a dimer cover $P'$ which agrees with $P$ everywhere except for the tile $G_i$. This allows us to define a partial order in the set of  dimer covers $\mathcal{L}(\mathcal{G})$ of a snake graph. Starting with the minimal dimer cover (or maximal dimer cover) we twist either South and North or West and East edges of a tile $G_i$ when permitted iteratively. This gives a lattice structure on $\mathcal{L}(\mathcal{G})$, where cover relations correspond to pairs of dimer covers related by a single twist, see \cite{MSW11}.

\begin{definition}\label{def:cross weight}
Let $T=\{\gamma_1,\ldots, \gamma_n\}$ be a triangulation of a marked surface $(S,M)$. Fix an arc $\gamma$ that is not in $T$ and let $\sg$ be its snake graph.

 \begin{enumerate}
    \item Define $\cross (\gamma)$ to be the product $\displaystyle\prod_{f} x_f$ where the index set is over all face weights of $\mathcal{G}_\gamma$ (considered with multiplicities).
    \item For a dimer cover $P$, define the weight $\wt(P)$  to be the product $\displaystyle\prod_{e\in P} x_e$ where the index set is taken over 
    all edge weights in $P$ (considered with multiplicities).
\end{enumerate} 
\end{definition}

For instance in Figure~\ref{fig:snake-1}, $\cross(\gamma)= x_1 x_2 x_3 x_4$ and the dimer cover $P$ on the right has weight $\wt(P) =x_2 x_3$, where we omitted the boundary  edge  contributions, that is the $b_i$'s, as they are set equal to $1$.

\begin{definition}\label{def:expansion formula}
Let $T=\{\gamma_1,\ldots,\gamma_n\}$ be a triangulation of a marked surface $(S,M)$ and let $\gamma$ be an arc in $(S,M)$. If $\gamma=\gamma_i \in T$, we have $x_\gamma = x_i$ and if  $\gamma \notin T$, we set 
\[ x_\mathcal{G_\gamma} = \dfrac{1}{\cross(\gamma)}  \sum_{P\in\mathcal{D}(\mathcal{G})} \wt(P),\]
 where $\mathcal{D}(\mathcal{G})$ is the set of dimer covers of  the snake graph $\mathcal{G}_\gamma$ associated to $\gamma$.
\end{definition}

\begin{theorem}{\cite[Theorem 10.1]{MSW11}} Suppose $(S,M)$ is a marked surface and $T$ is a triangulation on $(S,M)$. Let $\gamma$ be an arc that is not in $T$, $x_\gamma$ be the corresponding cluster variable in the cluster algebra $\mathcal{A}(S,M)$ and $\mathcal{G}_\gamma$ be its snake graph with respect to $T$.  Then
    $x_\gamma=x_{\mathcal{G_\gamma}}.$
\end{theorem}

\begin{figure}[h!]\scalebox{0.5}{
    \centering
\begin{tikzpicture}[scale=1]
\begin{scope}
\draw (0,0) \rectanglepath;
  \draw (0,1) \rectanglepath;
  \draw (1,1) \rectanglepath;
  \draw (2,1) \rectanglepath;
  \draw (0.5,0.5) node{$\huge 1$};
  \draw (0.5,1.5) node{$\huge 2$};
  \draw (1.5,1.5) node{$\huge 3$};
  \draw (2.5,1.5) node{$\huge 4$};
  \draw[ultra thick, color=blue] (0,0)--(1,0);
  \draw[ultra thick, color=blue] (0,2)--(0,1);
  \draw[ultra thick, color=blue] (1,1)--(2,1);
  \draw[ultra thick, color=blue] (1,2)--(2,2);
  \draw[ultra thick, color=blue] (3,1)--(3,2);
  \draw (4,1.5) node{};
  \draw (1.5,2.5) node[font = {\Large\bfseries\sffamily}]{$x_1 x_2 x_4$};
  \draw [ultra thick] (2,0)--(2,-0.8);
  \end{scope}
  \begin{scope}[yshift=-3cm]
\draw (0,0) \rectanglepath;
  \draw (0,1) \rectanglepath;
  \draw (1,1) \rectanglepath;
  \draw (2,1) \rectanglepath;
  \draw (0.5,0.5) node{$\huge 1$};
  \draw (0.5,1.5) node{$\huge 2$};
  \draw (1.5,1.5) node{$\huge 3$};
  \draw (2.5,1.5) node{$\huge 4$};
  \draw[ultra thick, color=blue] (0,0)--(1,0);
  \draw[ultra thick, color=blue] (0,2)--(0,1);
  \draw[ultra thick, color=blue] (1,1)--(1,2);
  \draw[ultra thick, color=blue] (2,1)--(2,2);
  \draw[ultra thick, color=blue] (3,1)--(3,2);
  \draw (3.8,1.5) node[font = {\Large\bfseries\sffamily}]{$x_1$};
  \draw [ultra thick] (-0.2,-0.2)--(-1.8,-0.8);
  \draw [ultra thick] (3.2,-0.2)--(4.8,-0.8);
  \end{scope}
  \begin{scope}[xshift=-5cm, yshift=-6cm]
  \draw (0,0) \rectanglepath;
  \draw (0,1) \rectanglepath;
  \draw (1,1) \rectanglepath;
  \draw (2,1) \rectanglepath;
  \draw (0.5,0.5) node{$\huge 1$};
  \draw (0.5,1.5) node{$\huge 2$};
  \draw (1.5,1.5) node{$\huge 3$};
  \draw (2.5,1.5) node{$\huge 4$};
  \draw[ultra thick, color=blue] (0,0)--(1,0);
  \draw[ultra thick, color=blue] (0,2)--(0,1);
  \draw[ultra thick, color=blue] (1,1)--(1,2);
 \draw[ultra thick, color=blue] (2,1)--(3,1);
  \draw[ultra thick, color=blue] (3,2)--(2,2);
  \draw (1.5,2.5) node[font = {\Large\bfseries\sffamily}]{$x_1 x_3 x_5$};
    \draw [ultra thick] (3.2,-0.2)--(4.8,-0.8);
  \end{scope}
   \begin{scope}[xshift=5cm, yshift=-6cm]
\draw (0,0) \rectanglepath;
  \draw (0,1) \rectanglepath;
  \draw (1,1) \rectanglepath;
  \draw (2,1) \rectanglepath;
  \draw (0.5,0.5) node{$\huge 1$};
  \draw (0.5,1.5) node{$\huge 2$};
  \draw (1.5,1.5) node{$\huge 3$};
  \draw (2.5,1.5) node{$\huge 4$};
  \draw[ultra thick, color=blue] (0,0)--(1,0);
  \draw[ultra thick, color=blue] (0,2)--(1,2);
  \draw[ultra thick, color=blue] (1,1)--(0,1);
  \draw[ultra thick, color=blue] (2,1)--(2,2);
  \draw[ultra thick, color=blue] (3,1)--(3,2);
  \draw (1.5,2.5) node[font = {\Large\bfseries\sffamily}]{$x_3$};
  \draw [ultra thick] (-0.2,-0.2)--(-1.8,-0.8);
  \draw [ultra thick] (3.2,-0.2)--(4.8,-0.8);
  \end{scope}
  \begin{scope}[yshift=-9cm]
  \draw (0,0) \rectanglepath;
  \draw (0,1) \rectanglepath;
  \draw (1,1) \rectanglepath;
  \draw (2,1) \rectanglepath;
  \draw (0.5,0.5) node{$\huge 1$};
  \draw (0.5,1.5) node{$\huge 2$};
  \draw (1.5,1.5) node{$\huge 3$};
  \draw (2.5,1.5) node{$\huge 4$};
  \draw[ultra thick, color=blue] (0,0)--(1,0);
  \draw[ultra thick, color=blue] (0,2)--(1,2);
  \draw[ultra thick, color=blue] (1,1)--(0,1);
  \draw[ultra thick, color=blue] (2,1)--(3,1);
  \draw[ultra thick, color=blue] (3,2)--(2,2);
  \draw (1.5,2.5) node[font = {\Large\bfseries\sffamily}]{$x_3^2 x_5$};
    \draw [ultra thick] (3.2,-0.2)--(4.8,-0.8);
  \end{scope}
  \begin{scope}[xshift=10cm, yshift=-9cm]
       \draw (0,0) \rectanglepath;
  \draw (0,1) \rectanglepath;
  \draw (1,1) \rectanglepath;
  \draw (2,1) \rectanglepath;
  \draw (0.5,0.5) node{$\huge 1$};
  \draw (0.5,1.5) node{$\huge 2$};
  \draw (1.5,1.5) node{$\huge 3$};
  \draw (2.5,1.5) node{$\huge 4$};
  \draw[ultra thick, color=blue] (0,0)--(0,1);
  \draw[ultra thick, color=blue] (0,2)--(1,2);
  \draw[ultra thick, color=blue] (1,1)--(1,0);
  \draw[ultra thick, color=blue] (2,1)--(2,2);
  \draw[ultra thick, color=blue] (3,1)--(3,2);
  \draw (3.8,1.5) node[font = {\Large\bfseries\sffamily}]{$x_2 x_3$};
  \draw [ultra thick] (-0.2,-0.2)--(-1.8,-0.8);
  \end{scope}
  \begin{scope}[xshift=5cm, yshift=-12cm]
       \draw (0,0) \rectanglepath;
  \draw (0,1) \rectanglepath;
  \draw (1,1) \rectanglepath;
  \draw (2,1) \rectanglepath;
  \draw (0.5,0.5) node{$\huge 1$};
  \draw (0.5,1.5) node{$\huge 2$};
  \draw (1.5,1.5) node{$\huge 3$};
  \draw (2.5,1.5) node{$\huge 4$};
  \draw[ultra thick, color=blue] (0,0)--(0,1);
  \draw[ultra thick, color=blue] (0,2)--(1,2);
  \draw[ultra thick, color=blue] (1,1)--(1,0);
  \draw[ultra thick, color=blue] (2,1)--(3,1);
  \draw[ultra thick, color=blue] (3,2)--(2,2);
  \draw (1.5,2.6) node[font = {\Large\bfseries\sffamily}]{$x_2 x_3^2 x_5$};
  \end{scope}
    \end{tikzpicture}
    }
    \caption{The lattice $\mathcal{L(\mathcal{G})}$ of the snake graph $\mathcal{G}_{\gamma}$ for the previous example. In the figure, $\mathrm{wt}(P)$ is indicated for each dimer cover $P$.}
\label{fig:lattice-1}
\end{figure}
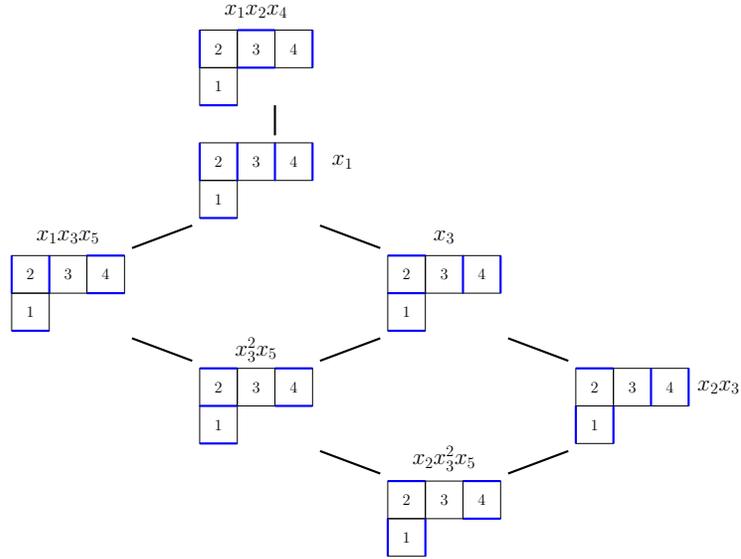

\begin{example}\label{ex: example expansion formula} Let $x_\gamma$ be the cluster variable associated to the arc $\gamma$ in Figure~\ref{fig:snake-1}. Then the terms in the sum $\displaystyle\sum_{P\in\mathcal{D}(\mathcal{G})} \wt(P)$ are given in  Figure~\ref{fig:lattice-1}.
Hence
\[ x_\gamma  = x_\mathcal{G{\gamma}} = \dfrac{1}{x_1 x_2 x_3 x_4} (x_1 x_2 x_4 + x_2 x_3 + x_1 + x_3 + x_1 x_3 x_5 + x_3^2 x_5 + x_2 x_3^2 x_5).\]
\end{example}

\subsection{Super lambda lengths}\label{subsec-super-L-len}

In \cite{penner19}, Penner and Zeitlin studied a supersymmetric analogue of decorated Teichm\"uller spaces by introducing the \textit{decorated super Teichm\"uller space} associated to a bordered marked surface (with punctures) $S$. They define a system of coordinates that splits into
two classes: the even coordinates, called the \emph{super lambda lengths}, and the odd coordinates called \emph{$\mu$-invariants}.

Similarly to the classic case described in the previous section, the super lambda lengths correspond to the arcs (including the boundary segments which we set to $1$). Moreover, the odd coordinates correspond to the triangles.

As in the classic case, the choice of a different oriented triangulation gives a different coordinate chart. In analogy to the expansion formula in Definition \ref{def:expansion formula}, Musiker, Ovenhouse and Zhang \cite{musiker2021expansion,musiker21}
 established a method to compute any super  lambda length $x_\gamma$ in terms of an initial system of coordinates defined by a triangulation when $S$ is a disk with marked points on the boundary.

Consider a triangulation $T$ (with no internal triangles) of the disk with $n+3$ marked points on the boundary. Set $x_1, \ldots, x_n$ as the super lambda lengths in $T$ (each one with extra data given by an orientation for each arc), and set $\theta_1, \ldots , \theta_{n+1}$ as the $\mu$-invariants corresponding to each triangle. Then the super lambda length  associated to an arc will be an element in (a quotient of) the $\mathbb{Z}_2$-graded algebra 
\[\mathcal{A} = A_0 \oplus A_1 = \mathbb{R}[x_1^{\pm 1/2}, \ldots, x_n^{\pm 1/2}, \theta_1, \ldots , \theta_{n+1}]\]
where any algebraic combination of $x_1^{\pm 1/2},\dots, x_n^{\pm 1/2}$ is in $A_0$, each one of $\theta_1,\dots, \theta_{n+1}$ is in $A_1$, and the $\theta$-variables are subject to relations: $\theta_i\theta_j=-\theta_j\theta_i$ for all $i,j$. The even part $A_0$ is spanned by monomials with an even number of $\theta$'s and the odd part $A_1$ is spanned by monomials with an odd number of $\theta$'s. Hence we refer to the $\theta$'s as the \emph{odd variables} and the $x_i$'s as the \emph{even variables}.


In the general surface case, the orientation of the arcs plays a role. Not all the oriented triangulations are related by sequences of flips due to the definition of flip for a general $S$. Because of this, Penner and Zeitlin introduce spin structures and an equivalence relation between them. The set of \textit{spin structures on $S$} is defined to be the set of equivalence classes of orientations on triangulations of $S$ with respect to the equivalence relation shown in Figure \ref{fig:spin_structure}, where $\epsilon_i$'s are orientations of the edges and $\theta$ is the $\mu$-invariant associated to the triangle. For a given $i$, $-\epsilon_i$ indicates the reverse of the edge $\epsilon_i$.
\begin{figure}[h]
    \centering
 \begin{tikzpicture}
   \node[draw,minimum size=3cm,regular polygon,regular polygon sides=3] (a) at (0,0) {};
  \draw (a.side 3) node[right]{$\epsilon_b$};
  \draw (a.side 2) node[below]{$\epsilon_a$};
  \draw (a.side 1) node[left]{$\epsilon_c$};
  \draw (a.center) node{$\theta$};
  \node[draw,minimum size=3cm,regular polygon,regular polygon sides=3] (b) at (5,0) {};
  \draw (b.side 3) node[right]{$-\epsilon_b$};
  \draw (b.side 2) node[below]{$-\epsilon_a$};
  \draw (b.side 1) node[left]{$-\epsilon_c$};
  \draw (b.center) node{$-\theta$};
  \draw (2.5,0.5) node{\huge $\sim$};
 \end{tikzpicture}
 \caption{The equivalence on orientations determining the spin structures on $S$.}
 \label{fig:spin_structure}
\end{figure}
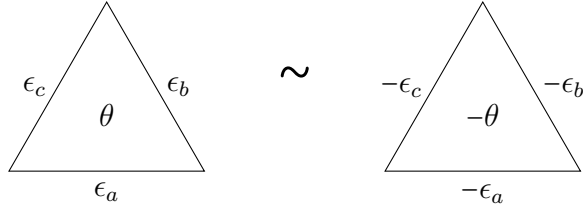

Because of the existence of these spin structures, now the Ptolemy transformation is an operation described as follows. A \textit{flip} of an oriented diagonal of $T$ is shown in Figure \ref{fig:super_ptolemy}: the flip of $e$ into $f$ is obtained by an anticlockwise rotation of $\pi/2$. In the figure, $\epsilon_i$'s are orientations of the edges and for a given $i$, $-\epsilon_i$ indicates the reverse of the edge $\epsilon_i$.

Then, the \emph{super Ptolemy relations} corresponding to the flip are given as follows:
\begin{align*}
    ef &=ac+bd + \sqrt{acbd}\, \sigma\theta\\
    \sigma'&=\frac{\sigma\sqrt{bd}-\theta \sqrt{ac}}{\sqrt{ac+bd}} \\
    \theta'&=\frac{\theta\sqrt{bd}+\sigma \sqrt{ac}}{\sqrt{ac+bd}}.
\end{align*}
 We write $\sigma > \theta$ to mean that $\sigma\theta=-\theta\sigma$ is the positive product between the two $\mu$-invariants, and similarly $\sigma' > \theta'$ (see Subsection~\ref{subsection-default-orientation} for the ordering of $\mu$-invariants in the special case of the disk).

Note that, abusing notation, we denote both an arc and its super lambda length by the same letter.

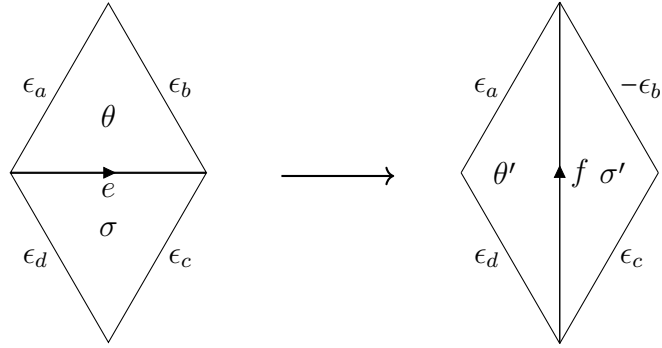
\begin{figure}[h]
    \centering
 \begin{tikzpicture}
 \begin{scope}
   \node[draw,minimum size=3cm,regular polygon,regular polygon sides=3] (a) at (0,0) {};
  \draw (a.side 3) node[right]{$\epsilon_b$};
  \draw (a.side 2) node[below]{$e$};
  \draw (a.side 1) node[left]{$\epsilon_a$};
  \draw (a.center) node{$\theta$};
  \node[rotate=180, draw,minimum size=3cm,regular polygon,regular polygon sides=3,anchor= side 2] (b) at (a.side 2) {};
  \draw (b.side 3) node[left]{$\epsilon_d$};
  \draw (b.side 1) node[right]{$\epsilon_c$};
  \draw (b.center) node{$\sigma$};
  \draw (a.corner 3)--(a.corner 2) node[currarrow,
    pos=0.5, 
    xscale=1,
    sloped]{};
 \end{scope}
 \begin{scope}[xshift=6cm]
 \node[draw=none,minimum size=3cm,regular polygon,regular polygon sides=3] (c) {};
  \draw (c.side 3) node[right]{$-\epsilon_b$};
  \draw (c.side 1) node[left]{$\epsilon_a$};
  \node[rotate=180,draw=none,minimum size=3cm,regular polygon,regular polygon sides=3,anchor= side 2] (d) at (c.side 2) {};
  \draw (d.side 3) node[left]{$\epsilon_d$};
  \draw (d.side 1) node[right]{$\epsilon_c$};
  \draw (c.corner 1)--(c.corner 2);
  \draw (c.corner 1)--(c.corner 3);
  \draw (d.corner 1)--(d.corner 2);
  \draw (d.corner 1)--(d.corner 3);
  \draw (c.corner 1)--node[right] {$f$}(d.corner 1);
  \draw (c.corner 2) node[right=8pt]{$\theta'$};
  \draw (c.corner 3) node[left=8pt]{$\sigma'$};
  \draw (c.corner 1)--(d.corner 1) node[currarrow,
    pos=0.5, 
    xscale=-1,
    sloped]{};
  \end{scope}
  \draw[thick, ->] (2.3,-0.8) -- (3.8,-0.8);
 \end{tikzpicture}
 \caption{A flip of an oriented diagonal of the triangulation $T$.}\label{fig:super_ptolemy}
\end{figure}

\begin{remarks}
\begin{enumerate}
\item Since multiplication of two  $\mu$-invariants is anticommutative, for any $\mu$-invariant $\theta$, we have that $\theta^2=0$ and in the situation of Figure \ref{fig:super_ptolemy}, $ \sigma\theta=\sigma'\theta'$.

\item The super Ptolemy transformation also affects the orientations  as illustrated in Figure \ref{fig:super_ptolemy}. The orientation of the edges $a,\, c,\, d$ are unchanged, while the one of $b$ is reversed. 
\item Note that we can ignore the orientations of the boundary segments since they do not contribute to the calculation of super lambda lengths.
\item Because of the orientation, the super Ptolemy relation is not an involution, and it needs to be applied $8$ times to go back to the initial situation. 
\item Unlike most surface cases, when $S$ is a disk with marked points and $T$ a triangulation on $S$ which doesn't contain any internal triangles, there is a unique spin structure \cite[Proposition 4.1]{musiker2021expansion}. 
\end{enumerate}
\end{remarks}

\begin{remark}
Note that performing a flip twice, we obtain the initial triangulation but the orientations of all the diagonals bounding one of the triangles are reversed, see Figure \ref{fig:super_ptolemy_twice}. Moreover, using the super Ptolemy relations, it is easy to see that the $\mu$-invariants are as indicated in the figure.
In particular, note that even if the two orientations give the same spin structure, the $\mu$-invariants are the same only up to sign. So the specific $\mu$-invariants do not only depend on a chosen triangulation and spin structure, but also the choice of an orientation. This is why we restrict to triangulations without internal triangles, because there we are able to determine a default orientation.
\end{remark}

\begin{figure}[h]
    \centering
 \begin{tikzpicture}[scale=0.8]
 \begin{scope}
   \node[draw,minimum size=3cm,regular polygon,regular polygon sides=3] (a) at (0,0) {};
  \draw (a.side 3) node[right]{$\epsilon_b$};
  \draw (a.side 2) node[below]{$e$};
  \draw (a.side 1) node[left]{$\epsilon_a$};
  \draw (a.center) node{$\theta$};
  \node[rotate=180, draw,minimum size=3cm,regular polygon,regular polygon sides=3,anchor= side 2] (b) at (a.side 2) {};
  \draw (b.side 3) node[left]{$\epsilon_d$};
  \draw (b.side 1) node[right]{$\epsilon_c$};
  \draw (b.center) node{$\sigma$};
  \draw (a.corner 3)--(a.corner 2) node[currarrow,
    pos=0.5, 
    xscale=1,
    sloped]{};
 \end{scope}
 \begin{scope}[xshift=6cm]
 \node[draw=none,minimum size=3cm,regular polygon,regular polygon sides=3] (c) {};
  \draw (c.side 3) node[right]{$-\epsilon_b$};
  \draw (c.side 1) node[left]{$\epsilon_a$};
  \node[rotate=180,draw=none,minimum size=3cm,regular polygon,regular polygon sides=3,anchor= side 2] (d) at (c.side 2) {};
  \draw (d.side 3) node[left]{$\epsilon_d$};
  \draw (d.side 1) node[right]{$\epsilon_c$};
  \draw (c.corner 1)--(c.corner 2);
  \draw (c.corner 1)--(c.corner 3);
  \draw (d.corner 1)--(d.corner 2);
  \draw (d.corner 1)--(d.corner 3);
  \draw (c.corner 1)--node[right] {$f$}(d.corner 1);
  \draw (c.corner 2) node[right=8pt]{$\theta'$};
  \draw (c.corner 3) node[left=8pt]{$\sigma'$};
  \draw (c.corner 1)--(d.corner 1) node[currarrow,
    pos=0.5, 
    xscale=-1,
    sloped]{};
  \end{scope}
  \begin{scope}[xshift=12cm]
   \node[draw,minimum size=3cm,regular polygon,regular polygon sides=3] (a) at (0,0) {};
  \draw (a.side 3) node[right]{$-\epsilon_b$};
  \draw (a.side 2) node[below]{$e$};
  \draw (a.side 1) node[left]{$-\epsilon_a$};
  \draw (a.center) node{$-\theta$};
  \node[rotate=180, draw,minimum size=3cm,regular polygon,regular polygon sides=3,anchor= side 2] (b) at (a.side 2) {};
  \draw (b.side 3) node[left]{$\epsilon_d$};
  \draw (b.side 1) node[right]{$\epsilon_c$};
  \draw (b.center) node{$\sigma$};
  \draw (a.corner 3)--(a.corner 2) node[currarrow,
    pos=0.5, 
    xscale=-1,
    sloped]{};
 \end{scope}
  \draw[thick, ->] (2.3,-0.8) -- (3.8,-0.8);
  \draw[thick, ->] (8.3,-0.8) -- (9.8,-0.8);
 \end{tikzpicture}
 \caption{Performing a flip twice.}\label{fig:super_ptolemy_twice}
\end{figure}
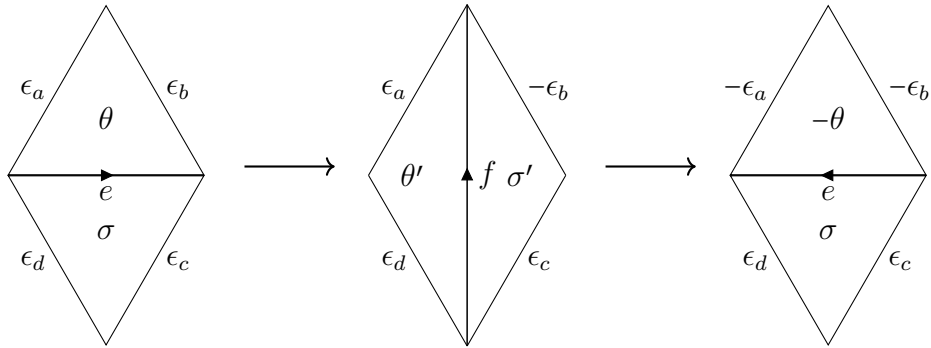

\subsubsection{The default orientation and the positive order}\label{subsection-default-orientation}
Let us fix $S$ to be a marked disk and assume $T$ is a triangulation on $S$ such that every triangle in $T$ has a boundary edge, that is there are no internal triangles in the triangulation or equivalently there is a longest arc $\gamma=(s,t)$ crossing all the arcs in $T$.

We will introduce an orientation of  $T$, which we call the \textit{default orientation} and this will determine a total order on the set of the $\mu$-invariants. We call this order the \textit{positive order}.

\begin{definition}\label{defn_fan_centre}
A triangulation $T$ is called a \textit{fan} if all the internal diagonals share a common vertex. If $T$ is not a fan, we define a canonical \textit{fan decomposition of $T$} as follows.
Fix an orientation on the longest arc $\gamma=(s,t)$.
The intersections of $\gamma$ with the internal diagonals of $T$ create smaller triangles. The vertices on $\partial S$ of these triangles are called \textit{fan centres} and are denoted by $c_1,\dots, c_N$ with the intersection of the arc $(c_i,c_{i+1})$ and $\gamma$ closer to the source of $\gamma$ than the one of $(c_{i+1},c_{i+2})$ and $\gamma$. Moreover, we set the source of $\gamma$ to be $c_0$ and its target to be $c_{N+1}$, see Figure~\ref{fig:fan_decomposition}.
\end{definition}

\begin{definition}{\label{defn_default_orientation}}
The \textit{default orientation of $T$} is defined as follows. If $T$ is a single fan, then all the interior edges are directed away from the only fan centre. Otherwise, consider the $N>1$ fan centres $c_1,\dots,c_N$ labelled as in Definition \ref{defn_fan_centre}. The interior edges inside each fan segment are directed away from its centre. Moreover, the edges where two fans meet are directed
\begin{align*}
    c_1\rightarrow c_2\rightarrow\cdots\rightarrow c_{N-1}\rightarrow c_N.
\end{align*}
\end{definition}

See Figure \ref{fig:fan_decomposition} for an example of a fan decomposition with default orientation.
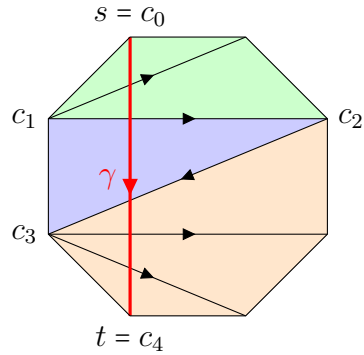
\begin{figure}
    \centering
 \begin{tikzpicture}
   \node[draw,minimum size=4cm,regular polygon,regular polygon sides=8] (a) {};
   \fill[green!20!white] (a.corner 3) -- (a.corner 8) --(a.corner 1) --(a.corner 2) --(a.corner 3);
   \fill[blue!20!white] (a.corner 8) -- (a.corner 4) --(a.corner 3) --(a.corner 8);
   \fill[orange!20!white] (a.corner 4) -- (a.corner 8) --(a.corner 7)--(a.corner 6) --(a.corner 5) --(a.corner 4);
  \draw  (a.corner 3)-- (a.corner 1) node[currarrow,pos=0.5, xscale=1, sloped]{};
  \draw  (a.corner 3)-- (a.corner 8) node[currarrow,pos=0.5, xscale=1, sloped]{};
  \draw  (a.corner 4)-- (a.corner 8) node[currarrow,pos=0.5, xscale=-1, sloped]{};
  \draw  (a.corner 4)-- (a.corner 7) node[currarrow,pos=0.5, xscale=1, sloped]{};
1  \draw  (a.corner 4)-- (a.corner 6) node[currarrow,pos=0.5, xscale=1, sloped]{};
   \draw [color=red, very thick] (a.corner 2)-- node[right=2pt, currarrow,pos=0.5, xscale=1, sloped] {} (a.corner 5);
   \node [below right=15pt and 15pt of a.corner 3, red]{$\gamma$};
  \draw (a.corner 2) node[above]{$s=c_0$};
  \draw (a.corner 5) node[below]{$t=c_4$};
  \draw (a.corner 3) node[left]{$c_1$};
  \draw (a.corner 8) node[right]{$c_2$};
  \draw (a.corner 4) node[left]{$c_3$};
  \draw (a.corner 1) -- (a.corner 2) -- (a.corner 3) -- (a.corner 4) -- (a.corner 5) -- (a.corner 6) -- (a.corner 7) -- (a.corner 8) -- (a.corner 1);
  \end{tikzpicture}
 \caption{A fan decomposition with default orientation. The different fans, with fan centres $c_1,\, c_2$ and $c_3$ respectively, are highlighted in different colours.}\label{fig:fan_decomposition}
\end{figure}

 The positive ordering can be described in different ways, the following was described in \cite[Remark 5.7]{musiker2021expansion}. Moreover, in Remark \ref{remark_positive_ordering_thetas}, we will give a different description, useful for some of our arguments.

\begin{definition}
    \label{defn:positive_ordering}
     We define the \emph{positive ordering} on the $\mu$-invariants inductively as follows. Denote the $\mu$-invariants $\theta_1,\dots,\theta_{n+1}$ ordered by proximity to $s$, that is $\theta_1$ has $s$ as a vertex and the triangle $\theta_i$ is closer to $s$ than $\theta_{i+1}$. 
    Then, if the edge between $\theta_i$ and $\theta_{i+1}$ is oriented so that  $\theta_i$ is to the right, we declare $\theta_i >\theta_j$ for all $j>i$.
    Otherwise, we declare $\theta_i<\theta_j$ for all $j>i$.
\end{definition}

Alternatively, the positive ordering on the $\mu$-invariants is induced by the ordering on fans and the positive ordering within each fan.

Recall that we denote by $\gamma = (s,t)$ the longest arc. By abuse of notation we denote by $\theta_i$ the $\mu$-invariant corresponding to a triangle and the triangle itself.

\begin{example}
Let the triangles in Figure \ref{fig:fan_decomposition} be labelled $\theta_1, \theta_2,\dots,\theta_6$ in order from $s$ to $t$. Then, the positive ordering is
\begin{align*}
    \theta_3 >\theta_6 > \theta_5 >\theta_4 > \theta_2 > \theta_1.
\end{align*}
\end{example}

\subsubsection{Snake graphs and double dimer covers}

Throughout this section, we assume $S$ is a disk with marked points on the boundary $\partial S$ and $T$ is a triangulation on $S$ that does not contain any internal triangles. We fix an orientation of the longest arc and consider the positive ordering on $(S,T)$.

Let $\gamma$ be a diagonal on $(S,T)$ and $\mathcal{G}_\gamma$ its snake graph.
Recall that each tile of a snake graph $\mathcal{G}_\gamma$ is obtained by gluing together two adjacent triangles along one arc of the triangulation.  In the presence of a spin structure on a disk, the two adjacent triangles correspond to two $\mu$-invariants and we write these in the associated tile of the snake graph at the bottom-left and the top-right corners.

\begin{example}\label{exmp_1} An example of snake graph is indicated on the right of Figure~\ref{fig:triang_polygon_eg} for the longest arc $\gamma$ depicted on the left of the figure.
\begin{figure}[h]
    \centering
 \begin{tikzpicture}
 \begin{scope}
   \node[draw,minimum size=4cm,regular polygon,regular polygon sides=5] (a) {};
  \draw  (a.corner 2)-- node[above] {$1$}  (a.corner 5)
  node[currarrow,pos=0.5, xscale=1, sloped]{};
  \draw  (a.corner 2)-- node[below=2pt] {$2$} (a.corner 4) node[currarrow,pos=0.5, xscale=1, sloped]{};
   \draw [color=red, very thick] (a.corner 1)-- node[right=.1pt,scale=1.2] {$\gamma$} (a.corner 3);
  \draw (a.corner 2) node[left=4pt]{$1$};
  \draw (a.corner 4) node[below]{$3$};
  \draw (a.corner 5) node[right]{$4$};
  \draw (a.corner 1) node[above,red]{$s=0$};
  \draw (a.corner 3) node[below,red]{$t=2$};
  \draw (a.side 3) node[below]{$e$};
  \draw (a.side 2) node[left]{$a$};
  \draw (a.side 1) node[above]{$b$};
  \draw (a.side 4) node[right]{$d$};
  \draw (a.side 5) node[above]{$c$};
  \draw (a.side 3) node[above=5pt]{\tiny $\theta_3$};
  \draw (a.side 5) node[left=8pt]{\tiny $\theta_1$};
  \draw (a.side 4) node[left=8pt]{\tiny $\theta_2$};
  \end{scope}
 \begin{scope}[scale=1.3, xshift=3.8cm, yshift=-1cm]
  \draw (0,0) \rectanglepath;
  \draw (1,0) \rectanglepath;
  \draw (-0.2,0.5) node{$b$};
  \draw (0.5,-0.2) node{$c$};
  \draw (1.5,-0.2) node{$1$};
  \draw (2.2,0.5) node{$a$};
  \draw (1.2,0.5) node{$d$};
  \draw (0.5,1.2) node{$2$};
  \draw (1.5,1.2) node{$e$};
  \draw (0.2,0.2) node{\tiny $\theta_1$};
  \draw (0.8,0.8) node{\tiny $\theta_2$};
  \draw (1.2,0.2) node{\tiny $\theta_2$};
  \draw (1.8,0.8) node{\tiny $\theta_3$};
   \draw (0.5,0.5) node{\Large $1$};
  \draw (1.5,0.5) node{\Large $2$};
 \end{scope}
 \end{tikzpicture}
 \caption{A triangulation of the pentagon with default orientation, and the snake graph of the longest arc $\gamma$ with source $s$ and target $t$ on the right.}\label{fig:triang_polygon_eg}
\end{figure}
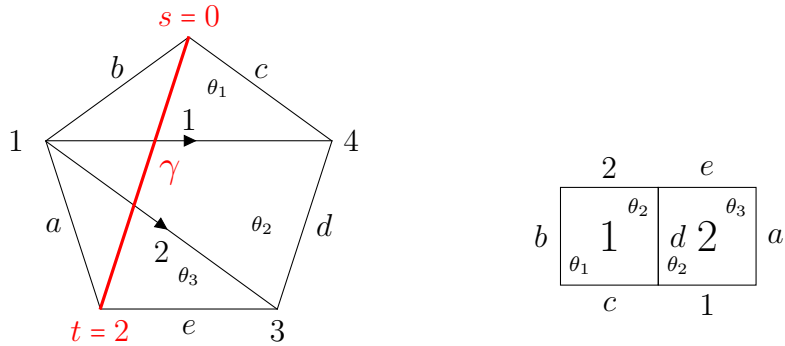
\end{example}

\begin{definition}{\cite[Definition 4.1]{musiker21}}
A \textit{double dimer cover of} a planar bipartite graph $G$ is a multiset $D$ 
of edges of $G$ such that each vertex of $G$ is incident to exactly two edges from $D$. Each element of $D$ is called a \textit{dimer}. If $D$ contains two copies of the same edge, these are called a \textit{double dimer}. The set of all double dimer covers of $G$ is denoted by $\mathcal{DD}(G)$.
\end{definition}

\begin{remark}\label{remark_superimpose_single_dimers}
    Note that any double dimer cover $D$ of a snake graph $\mathcal{G}$ can be obtained (non uniquely) by superimposing two dimer covers $P$ and $P'$ of $\mathcal{G}$, that can be constructed as follows. Whenever $D$ has two copies of the same edge $e$, include a copy of $e$ in both $P$ and $P'$. The remaining edges of $D$ are single, and by construction they lie on a cycle enclosing some tiles of $\mathcal{G}$. Such a cycle is the union of the maximal and minimal dimer covers of the smaller snake graph consisting of the enclosed tiles. Add the minimal one to $P$ and the maximal one to $P'$. It is easy to check that doing this for each cycle in $D$ creates two dimer covers $P$ and $P'$ that superimpose to give $D$ back.
\end{remark}

Each snake graph has two double dimer covers consisting only of boundary edges: the \textit{minimal} and the \textit{maximal} one.

\begin{definition}
 Following the same convention as in Definition \ref{defn_minimal_matching}, the \emph{minimal double dimer cover} is the one containing two copies of the West edge of the initial tile $G_1$ and only boundary edges. The \emph{maximal double dimer cover} is given complementary to the minimal one and contains only boundary edges. Notice that both maximal and minimal double dimer covers consist only of double edges.
\end{definition}

\begin{notation}
We illustrate double dimers as double blue lines and dimers as single blue lines.
\end{notation}

\begin{definition}{\cite[Defn 4.4]{musiker21}}\label{defn_weight_MOZ}
Let $\mathcal{G}$ be a snake graph and $D$ a double dimer cover of $\mathcal{G}$.
\begin{itemize}
\item The \textit{weight of a dimer} in $D$ is the square root of the weight of the corresponding edge of $\mathcal{G}$.
\item Let $c(D)$ be the set of cycles formed by edges of $D$. For $C\in c(D)$, let $\theta_i$ and $\theta_j$ be the odd variables corresponding to the triangles in the bottom-left and top-right corner of $C$ respectively. Then the \textit{weight of the cycle $C$}, denoted by $\wt(C)$, is the product of $\theta_i$ and $\theta_j$ multiplied according to the positive order.
\item The weight of $D$ is defined to be
\begin{align*}
    \textup{wt}_2(D):= \prod_{e\in D} \textup{wt}_2(e) \prod_{C\in c(D)} \wt(C).
\end{align*}
\end{itemize}
\end{definition}

\begin{example}
Consider the following double dimer cover $D$ of a snake graph.
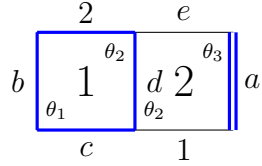
\begin{figure}[h]
\centering
\begin{tikzpicture}[scale=1.3]
  \draw (0,0) \rectanglepath;
  \draw (1,0) \rectanglepath;
  \draw (-0.2,0.5) node{$b$};
  \draw (0.5,-0.2) node{$c$};
  \draw (1.5,-0.2) node{$1$};
  \draw (2.2,0.5) node{$a$};
  \draw (1.2,0.5) node{$d$};
  \draw (0.5,1.2) node{$2$};
  \draw (1.5,1.2) node{$e$};
  \draw (0.2,0.2) node{\tiny $\theta_1$};
  \draw (0.8,0.8) node{\tiny $\theta_2$};
  \draw (1.2,0.2) node{\tiny $\theta_2$};
  \draw (1.8,0.8) node{\tiny $\theta_3$};
  \draw (0.5,0.5) node{\Large $1$};
  \draw (1.5,0.5) node{\Large $2$};
  \draw[very thick, double, double distance=1.3pt, blue](2,1)--(2,0);
  \draw[very thick, blue]  (0,1)--(1,1);
 \draw[very thick, blue]  (0,0)--(1,0);
  \draw[very thick, blue] (1,0)--(1,1);
  \draw[very thick, blue] (0,0)--(0,1);
\end{tikzpicture}
\end{figure}

The weight of $D$ is
\begin{align*}
    \textup{wt}_2(D)= \sqrt{ x_2}\, \theta_2 \theta_1.
\end{align*}
Note that in the formula for the weight of $D$, the $\mu$-invariants appear following the positive ordering $\theta_3>\theta_2>\theta_1$ which, like in this example, can be different from the order they appear in the snake graph.
\end{example}

In \cite{musiker21}, a super analogue of the snake graph formula is shown to  give expansions for super lambda lengths.
\begin{theorem}{\cite[Theorem 6.2]{musiker21}}\label{thm_MOZ_formula}
Consider a triangulated polygon with no internal triangles. Let $\mathcal{G}$ be the snake graph corresponding to an arc $\gamma \notin T$. Then  the super lambda length $x_\gamma$ is given by
\begin{align*}
    x_\gamma =\frac{1}{\cross(\gamma)} \sum_{D\in \mathcal{DD}(\mathcal{G})} \textup{wt}_2(D).
\end{align*}
\end{theorem}

\begin{example}\label{eg_double_dim}
Continuing Example \ref{exmp_1}, the set of double dimer covers for the snake graph of $\gamma$, consists of the following six elements.
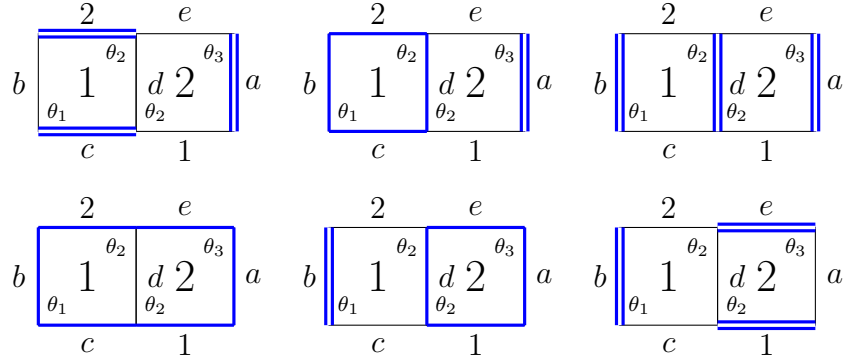
\begin{figure}[H]
    \centering
    \begin{subfigure}
    \centering
    \begin{tikzpicture}[scale=1.3]
  \draw (0,0) \rectanglepath;
  \draw (1,0) \rectanglepath;
  \draw (-0.2,0.5) node{$b$};
  \draw (0.5,-0.2) node{$c$};
  \draw (1.5,-0.2) node{$1$};
  \draw (2.2,0.5) node{$a$};
  \draw (1.2,0.5) node{$d$};
  \draw (0.5,1.2) node{$2$};
  \draw (1.5,1.2) node{$e$};
  \draw (0.2,0.2) node{\tiny $\theta_1$};
  \draw (0.8,0.8) node{\tiny $\theta_2$};
  \draw (1.2,0.2) node{\tiny $\theta_2$};
  \draw (1.8,0.8) node{\tiny $\theta_3$};
  \draw (0.5,0.5) node{\Large $1$};
  \draw (1.5,0.5) node{\Large $2$};
  \draw[very thick, double, double distance=1.3pt, blue] (2,1)--(2,0);
  \draw[very thick, double, double distance=1.3pt, blue] (1,1)--(0,1);
  \draw[very thick, double, double distance=1.3pt, blue] (0,0)--(1,0);
\end{tikzpicture}
    \end{subfigure}
    \begin{subfigure}
    \centering
    \begin{tikzpicture}[scale=1.3]
  \draw (0,0) \rectanglepath;
  \draw (1,0) \rectanglepath;
  \draw (-0.2,0.5) node{$b$};
  \draw (0.5,-0.2) node{$c$};
  \draw (1.5,-0.2) node{$1$};
  \draw (2.2,0.5) node{$a$};
  \draw (1.2,0.5) node{$d$};
  \draw (0.5,1.2) node{$2$};
  \draw (1.5,1.2) node{$e$};
  \draw (0.2,0.2) node{\tiny $\theta_1$};
  \draw (0.8,0.8) node{\tiny $\theta_2$};
  \draw (1.2,0.2) node{\tiny $\theta_2$};
  \draw (1.8,0.8) node{\tiny $\theta_3$};
 \draw (0.5,0.5) node{\Large $1$};
  \draw (1.5,0.5) node{\Large $2$};
  \draw[very thick, double, double distance=1.3pt, blue] (2,1)--(2,0);
  \draw[very thick, blue]  (0,1)--(1,1);
 \draw[very thick, blue]  (0,0)--(1,0);
  \draw[very thick, blue] (1,0)--(1,1);
  \draw[very thick, blue] (0,0)--(0,1);
\end{tikzpicture}
    \end{subfigure}
    \begin{subfigure}
    \centering
    \begin{tikzpicture}[scale=1.3]
  \draw (0,0) \rectanglepath;
  \draw (1,0) \rectanglepath;
  \draw (-0.2,0.5) node{$b$};
  \draw (0.5,-0.2) node{$c$};
  \draw (1.5,-0.2) node{$1$};
  \draw (2.2,0.5) node{$a$};
  \draw (1.2,0.5) node{$d$};
  \draw (0.5,1.2) node{$2$};
  \draw (1.5,1.2) node{$e$};
  \draw (0.2,0.2) node{\tiny $\theta_1$};
  \draw (0.8,0.8) node{\tiny $\theta_2$};
  \draw (1.2,0.2) node{\tiny $\theta_2$};
  \draw (1.8,0.8) node{\tiny $\theta_3$};
   \draw (0.5,0.5) node{\Large $1$};
  \draw (1.5,0.5) node{\Large $2$};
  \draw[very thick, double, double distance=1.3pt, blue] (2,1)--(2,0);
  \draw[very thick, double, double distance=1.3pt, blue] (0,0)--(0,1);
  \draw[very thick, double, double distance=1.3pt, blue] (1,0)--(1,1);
\end{tikzpicture}
\end{subfigure}

\begin{subfigure}
    \centering
    \begin{tikzpicture}[scale=1.3]
  \draw (0,0) \rectanglepath;
  \draw (1,0) \rectanglepath;
  \draw (-0.2,0.5) node{$b$};
  \draw (0.5,-0.2) node{$c$};
  \draw (1.5,-0.2) node{$1$};
  \draw (2.2,0.5) node{$a$};
  \draw (1.2,0.5) node{$d$};
  \draw (0.5,1.2) node{$2$};
  \draw (1.5,1.2) node{$e$};
  \draw (0.2,0.2) node{\tiny $\theta_1$};
  \draw (0.8,0.8) node{\tiny $\theta_2$};
  \draw (1.2,0.2) node{\tiny $\theta_2$};
  \draw (1.8,0.8) node{\tiny $\theta_3$};
   \draw (0.5,0.5) node{\Large $1$};
  \draw (1.5,0.5) node{\Large $2$};
  \draw[very thick, blue]  (0,1)--(1,1);
 \draw[very thick, blue] (0,0)--(1,0);
  \draw[very thick, blue] (0,0)--(0,1);
    \draw[very thick, blue] (1,0)--(2,0);
  \draw[very thick, blue] (2,0)--(2,1);
    \draw[very thick, blue] (2,1)--(1,1);
\end{tikzpicture}
    \end{subfigure}
\begin{subfigure}
    \centering
    \begin{tikzpicture}[scale=1.3]
  \draw (0,0) \rectanglepath;
  \draw (1,0) \rectanglepath;
 \draw (-0.2,0.5) node{$b$};
  \draw (0.5,-0.2) node{$c$};
  \draw (1.5,-0.2) node{$1$};
  \draw (2.2,0.5) node{$a$};
  \draw (1.2,0.5) node{$d$};
  \draw (0.5,1.2) node{$2$};
  \draw (1.5,1.2) node{$e$};
  \draw (0.2,0.2) node{\tiny $\theta_1$};
  \draw (0.8,0.8) node{\tiny $\theta_2$};
  \draw (1.2,0.2) node{\tiny $\theta_2$};
  \draw (1.8,0.8) node{\tiny $\theta_3$};
  \draw (0.5,0.5) node{\Large $1$};
  \draw (1.5,0.5) node{\Large $2$};
  \draw[very thick, double, double distance=1.3pt, blue] (0,0)--(0,1);
   \draw[very thick, blue] (1,0)--(1,1);
    \draw[very thick, blue] (1,0)--(2,0);
  \draw[very thick, blue] (2,0)--(2,1);
    \draw[very thick, blue] (2,1)--(1,1);
\end{tikzpicture}
    \end{subfigure}
\begin{subfigure}
    \centering
    \begin{tikzpicture}[scale=1.3]
  \draw (0,0) \rectanglepath;
  \draw (1,0) \rectanglepath;
  \draw (-0.2,0.5) node{$b$};
  \draw (0.5,-0.2) node{$c$};
  \draw (1.5,-0.2) node{$1$};
  \draw (2.2,0.5) node{$a$};
  \draw (1.2,0.5) node{$d$};
  \draw (0.5,1.2) node{$2$};
  \draw (1.5,1.2) node{$e$};
  \draw (0.2,0.2) node{\tiny $\theta_1$};
  \draw (0.8,0.8) node{\tiny $\theta_2$};
  \draw (1.2,0.2) node{\tiny $\theta_2$};
  \draw (1.8,0.8) node{\tiny $\theta_3$};
   \draw (0.5,0.5) node{\Large $1$};
  \draw (1.5,0.5) node{\Large $2$};
  \draw[very thick, double, double distance=1.3pt, blue] (0,0)--(0,1);
  \draw[very thick, double, double distance=1.3pt, blue] (1,0)--(2,0);
  \draw[very thick, double, double distance=1.3pt, blue] (1,1)--(2,1);
\end{tikzpicture}
    \end{subfigure}
\end{figure}
The corresponding weights are respectively
\begin{align*}
     x_2,\,\, \sqrt{ x_2}\, \theta_2 \theta_1, \,\, 1,\\
    \sqrt{ x_1 x_2}\, \theta_3 \theta_1,\,\, \sqrt{ x_1}\, \theta_3 \theta_2, \,\,  x_1.
\end{align*}

Since cross$(\gamma)=x_1 x_2$, we have that
\begin{align*}
    x_\gamma =\frac{1}{x_1x_2} ( x_2 + \sqrt{ x_2}\, \theta_2 \theta_1 + \,\, 1 +
    \sqrt{x_1 x_2}\, \theta_3 \theta_1 + \sqrt{ x_1}\, \theta_3 \theta_2+ x_1).
\end{align*}
\end{example}

The second part of \cite[Theorem 6.2]{musiker21} expresses the $\mu$-invariant associated to the triangle, having $\gamma$ and a boundary segment as sides, in terms of the initial triangulation. 

\subsubsection{Lattice structure in double dimer covers}

Similar to the lattice structure of dimer covers of a snake graph, we may introduce a lattice structure on double dimer covers of a snake graph by twisting either South and North or West and East edges of a tile $G_i$, when permitted, iteratively.
 Note that when twisting a tile containing double dimers, i.e. a tile with double blue edges, one only rotates one copy of these edges. Starting with the minimal double dimer cover, we proceed by twisting tiles until the maximal double dimer cover is reached.

See Figure \ref{fig:lattice_eg} for the lattice of double dimer covers corresponding to the running example of this section. The labels on the lattice edges indicate the face weight of the tile being rotated at each step.

\begin{figure}[h!]\scalebox{0.8}{
    \centering
\xymatrix{
&{\begin{tikzpicture}
  \draw (0,0) \rectanglepath;
  \draw (1,0) \rectanglepath;
  \draw[very thick, double, double distance=1.3pt, blue] (2,1)--(2,0);
  \draw[very thick, double, double distance=1.3pt, blue] (1,1)--(0,1);
  \draw[very thick, double, double distance=1.3pt, blue] (0,0)--(1,0);
 \draw (0.5,0.5) node{\Large $1$};
  \draw (1.5,0.5) node{\Large $2$};
\end{tikzpicture}}\ar@[red]@{-}[d]^-{\color{red}1}
\\
&{\begin{tikzpicture}
  \draw (0,0) \rectanglepath;
  \draw (1,0) \rectanglepath;
  \draw[very thick, double, double distance=1.3pt, blue] (2,1)--(2,0);
  \draw[very thick, blue]  (0,1)--(1,1);
 \draw[very thick, blue]   (0,0)--(1,0);
  \draw[very thick, blue]  (1,0)--(1,1);
  \draw[very thick, blue]  (0,0)--(0,1);
 \draw (0.5,0.5) node{\Large $1$};
  \draw (1.5,0.5) node{\Large $2$};
\end{tikzpicture}}
\ar@[red]@{-}[ld]_-{\color{red}1}\ar@[red]@{-}[rd]^-{\color{red}2}
\\
{\begin{tikzpicture}
  \draw (0,0) \rectanglepath;
  \draw (1,0) \rectanglepath;
  \draw[very thick, double, double distance=1.3pt, blue] (2,1)--(2,0);
  \draw[very thick, double, double distance=1.3pt, blue] (0,0)--(0,1);
  \draw[very thick, double, double distance=1.3pt, blue] (1,0)--(1,1);
 \draw (0.5,0.5) node{\Large $1$};
  \draw (1.5,0.5) node{\Large $2$};
\end{tikzpicture}}
\ar@[red]@{-}[rd]_-{\color{red}2}
&&
{\begin{tikzpicture}
  \draw (0,0) \rectanglepath;
  \draw (1,0) \rectanglepath;
  \draw[very thick, blue]   (0,1)--(1,1);
 \draw[very thick, blue]  (0,0)--(1,0);
  \draw[very thick, blue]  (0,0)--(0,1);
    \draw[very thick, blue]  (1,0)--(2,0);
  \draw[very thick, blue]  (2,0)--(2,1);
    \draw[very thick, blue]  (2,1)--(1,1); 
 \draw (0.5,0.5) node{\Large $1$};
  \draw (1.5,0.5) node{\Large $2$};
\end{tikzpicture}}
\ar@[red]@{-}[ld]^-{\color{red}1}
\\
&{\begin{tikzpicture}
  \draw (0,0) \rectanglepath;
  \draw (1,0) \rectanglepath;
  \draw[very thick, double, double distance=1.3pt, blue] (0,0)--(0,1);
   \draw[very thick, blue]  (1,0)--(1,1);
    \draw[very thick, blue]  (1,0)--(2,0);
  \draw[very thick, blue]  (2,0)--(2,1);
    \draw[very thick, blue] (2,1)--(1,1);
 \draw (0.5,0.5) node{\Large $1$};
  \draw (1.5,0.5) node{\Large $2$};
\end{tikzpicture}}
\ar@[red]@{-}[d]^-{\color{red}2}
\\
&{\begin{tikzpicture}
  \draw (0,0) \rectanglepath;
  \draw (1,0) \rectanglepath;
 \draw[very thick, double, double distance=1.3pt, blue] (0,0)--(0,1);
  \draw[very thick, double, double distance=1.3pt, blue] (1,0)--(2,0);
  \draw[very thick, double, double distance=1.3pt, blue] (1,1)--(2,1);
 \draw (0.5,0.5) node{\Large $1$};
  \draw (1.5,0.5) node{\Large $2$};
\end{tikzpicture}}
}}
    \caption{The lattice of double dimer covers of Example~\ref{eg_double_dim}.}
    \label{fig:lattice_eg}
\end{figure}
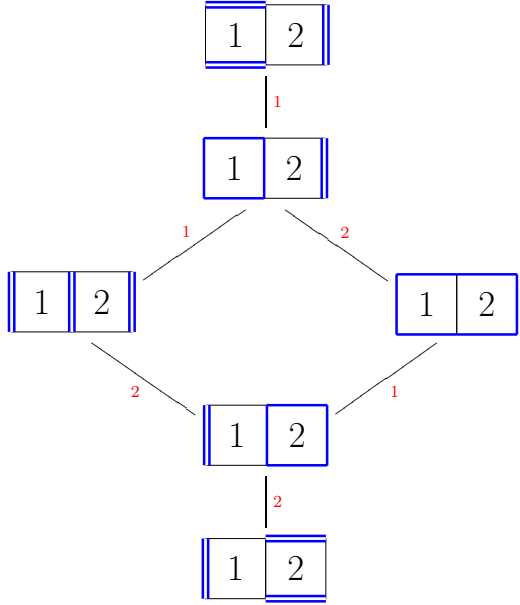

\begin{remark}
    The lattice structure on the double dimer covers of a snake graph does not depend on the face weights of the snake graph. Furthermore, one may introduce the lattice edge weights by simply indicating the position of a tile for which the two vertical edges are replaced with the two horizontal edges. Thus, we may also consider snake graphs coming from (unpunctured) marked surfaces or even abstract snake graphs with formal face (and edge) weights or with no weights at all.
\end{remark}

\section{Representation theoretic interpretation of the snake graph formula}
\label{Sec:RepT-SG}

Throughout we assume $K$ to be an algebraically closed field.  Moreover, whenever we utilize the Euler Poincar\'e characteristic $\chi(V)$ of an algebraic variety $V$, we assume $K=\mathbb{C}$. In this section, we will consider basic finite dimensional associative algebras over $K$. Such algebras can be defined by the path algebra $KQ$ of a quiver $Q$ modulo an admissible ideal $I$. For an algebra $A =KQ/I$ we denote by $P_j, I_j, S_j$ the indecomposable projective, injective, and simple $A$-modules at the vertex $j$, respectively. Note that we identify $A$-modules with $(Q,I)$-representations. We denote by $\{ e_i \}$ the usual canonical basis for the abelian group $\mathbb{Z}^{ |Q_0| }$. 
 Let $\mathrm{mod}\,A$ denote the category of finitely generated right $A$-modules. Given $M \in \mathrm{mod} \, A$, we denote by $\underline{\mathrm{dim}} (M)$ its dimension vector  and by $\left |M\right|$ the number of nonisomorphic indecomposable direct summands of $M$. As an introduction on finite dimensional algebras, we recommend \cite{assem06,s14}.

 \begin{definition}\label{def:bilinear form}
     Let $M$ and $N$ be $A$-modules where  $A = K Q/I$ is a basic finite dimensional $K$-algebra. We define an \emph{antisymmetrized bilinear form}  by
\[ \langle M, N\rangle_{A} = \mathrm{dim}_K \mathrm{Hom}_{A} (M,N) - \mathrm{dim}_K \mathrm{Hom}_{A} (N,M) - \mathrm{dim}_K  \mathrm{Ext}_{A}^1 (M,N) + \mathrm{dim}_K  \mathrm{Ext}_{A}^1 (N,M).\]
 \end{definition}

\begin{definition}\label{def:index-module} Let $A = KQ/I$ be a finite dimensional algebra and let $M$ be an $A$-module. We define the \emph{index} of $M$ as the $\mathbb{Z}^{ |Q_0| }$ vector
\[ \mathrm{ind}_A (M) =  [N^1] - [N^0], \] 
where $N^0$ and $N^1$ arise from a minimal injective resolution $ 0 \to M \to N^0 \to N^1$ of $M$ in $\mathrm{mod} \, A$, and $[N] = \sum_i  m_i e_i$ if we have $N \simeq \bigoplus_i I_i^{m_i}$.
\end{definition}

Given a triangulation $T$ of a marked surface, we can define the associated Jacobian algebra $A_T = KQ_T / I_T $. In the case of unpunctured surfaces this definition was given in \cite{ass10}. Notice that when the triangulation $T$ has no internal triangles, that is there are no triangles where the three edges belong to $T$, we have $I_T=0$. The algebra $A_T$ is a gentle algebra, hence each string defines a quiver representation. The reader can find the details in \cite{BuR87}. Each (generalized) arc $\gamma$ defines an indecomposable string $A_T$-module $M_\gamma$. 

 \begin{remark}\label{rem: g vector - notation}
    The \emph{$g$-vector} of a module $M\in\text{mod}\,A$ is defined as $[P^0]-[P^1]$ where $P^1\to P^0\to M\to 0$ is the minimal projective presentatiton of $M$.   The index and the $g$-vector are closely related. By definition of the Auslander-Reiten translation $\tau$, if $M$ is nonprojective then its $g$-vector equals $\mathrm{ind}_{A} (\tau M)$. 

 \end{remark}

\subsection{The CC-map} \label{subsect: CC map}

In addition to the combinatorial formula given in Definition~\ref{def:expansion formula}, cluster variables may also be expressed homologically by the CC-map.  This is due to Caldero and Chapoton \cite{Cch} in type $A$ and to Palu \cite{palu08} in the general setting of 2-CY categories admitting cluster-tilting objects. 

Following \cite{pla18}, we will first recall the submodule Grassmannians. 
Let $Q$ be a finite quiver and $I$ an admissible ideal. Let $V= (V_i,V_a)$, where $i \in Q_0$ and $a \in Q_1$ is an arrow $a:s(a)\to t(a)$, be a representation of $(Q,I)$. A \emph{subrepresentation} $W$ of $V$ is a tuple $(W_i)_{i \in Q_0}$ such that
\begin{enumerate}
    \item $W_i$ is a $K$-subspace of $V_i$,
    \item for each $a\in Q_1$ we have that $V_a (W_{s(a)})$ is a $K$-subspace of $W_{t(s)}$.
\end{enumerate}
Let $\mathbf{e} \in \mathbb{N}^{|Q_0|}$ be a dimension vector. The \emph{submodule Grassmannian of $V$ of dimension $\mathbf{e}$} is the subset  $\mathrm{Gr}_{\mathbf{e}} (V)$ of $\prod_{i \in Q_0} \mathrm{Gr}_{e_i} (V_i)$ given by all points defining a subrepresentation of $V$.

For $(S,M)$ a marked surface, $T$ a triangulation and $\gamma$ an arc that is not in $T$ and $M_\gamma$  the indecomposable $A_T$-module associated to $\gamma$ in the Jacobian algebra $A_T$, set
\[ \textit{CC }(M_\gamma) 
= X^{\mathrm{ind}_{A_T} (M_\gamma)  } \sum\limits_{\mathbf{e}\, \in \mathbb{Z}^n }\chi (\mathrm{Gr}_{\mathbf{e}}( M_\gamma ) )  \prod\limits_{i=1}^n x_i^{\langle S_i , \mathbf{e} \rangle} 
\]
where $\mathbf{e}  = \underline{\mathrm{dim}} ( \bigoplus_j  S_j^{m_j})$, $\vert (Q_T)_0 \vert =n $, $\langle - , - \rangle = \langle - , - \rangle_{A_T}$ is the antisymmetrized bilinear form, and  we denote $\langle S_i , \mathbf{e} \rangle = \langle S_i , \oplus_j S_j^{m_j} \rangle$.

\begin{theorem}\label{thm:cc}
Let  (S,M) be an unpunctured surface, $T$ a triangulation and $\gamma$ an arc that is not in $T$. Let $x_\gamma$ be the cluster variable associated to $\gamma$ in the cluster algebra $\mathcal{A}(S,M)$ and $M_\gamma$ be the indecomposable module associated to $\gamma$ over the Jacobian algebra $A_T$. Then $\textit{CC }(M_\gamma)=x_\gamma$.
\end{theorem}

In \cite[Section 5.3]{bz13}, the authors compare the snake graph formula and the CC-map for cluster algebras coming from unpunctured surfaces. Combining this with \cite[Theorem 3.18]{CS21}, we obtain the following result.

\begin{theorem}\label{lemma-brustle-zhang} In the setting of Theorem~\ref{thm:cc}, let $\mathcal{G_\gamma}$ be the snake graph associated to $\gamma$. Then  there is a 1-1 correspondence between the terms in the two formulae for $CC(M_\gamma)$ and $x_\mathcal{G_\gamma}$.
\end{theorem}

\begin{remark}\label{remark_BZ_hidden}
    The correspondence mentioned in the above theorem works as follows.  Recall there is a $1$-$1$ correspondence between dimer covers $P_N$ of $\mathcal{G}_\gamma$ and elements $N$ of the canonical submodule lattice of $M_\gamma$. In general the two terms $1/\textup{cross}(\gamma)$ and $X^{\textup{ind}_{A_T}(M_\gamma)}$ are not equal. However, we have the following equality
    \begin{align*}
    \dfrac{1}{\textup{cross}(\gamma)}\dfrac{\textup{wt}(P_{\textup{min}})}{\textup{wt}(P_{\textup{min}})}\textup{wt}(P_N)=X^{\textup{ind}_{A_T}(M_\gamma)}\prod\limits_{i=1}^n x_i^{\langle S_i , \oplus_j S_j^{m_j}\rangle},
    \end{align*}
    where $\underline{\textup{dim}} (N) = \underline{\textup{dim}}(\oplus_j S_j^{m_j})$.
    Note that, since $P_{\textup{min}}$ corresponds to $N=0$, we have that
    \begin{align*}
        \dfrac{\textup{wt}(P_{\textup{min}})}{\textup{cross}(\gamma)}=X^{\textup{ind}_{A_T}(M_\gamma)},
    \end{align*}
and hence  \begin{align*}
        \dfrac{\textup{wt}(P_{N})}{\textup{wt}(P_{\textup{min}})}= \prod\limits_{i=1}^n x_i^{\langle S_i , \oplus_j S_j^{m_j}\rangle}.
    \end{align*}
For example, consider the left-most dimer cover $P_N$ in the lattice in Figure \ref{fig:lattice-1}. Then the corresponding left hand side term is
    \begin{align*}
        \dfrac{1}{x_1x_2x_3x_4} x_1 x_3 x_5.
    \end{align*}
    Then $M_\gamma=\begin{smallmatrix}&3\\2&&4\\1\end{smallmatrix}$ and the submodule corresponding to $P_N$ is $N=\begin{smallmatrix}2\\1\end{smallmatrix}$. The right hand side is then
    \begin{align*}
        \dfrac{x_3 x_5}{x_1 x_4} \dfrac{x_1}{x_2x_3}.
    \end{align*}
\end{remark}

\subsection{Cluster category of type $A$}\label{subsec-cluster category}

The cluster category was defined by \cite{bmrrt}. Simultaneously the geometric cluster category $\mathcal{C}$ of Dynkin type $A_n$ was introduced in \cite{CCS}, where the realization is given in terms of the internal diagonals of the $(n+3)$-gon. Moreover, the categorical definition can be extended and include the edges of the polygon. This realization would correspond to the Frobenius cluster category $\mathcal{C}_{\mathcal{F}}$ of type $A_n$, associated to the Grassmannian of type $G_{2,n+3}$ from \cite{JKS16}. The stable category $\underline{\mathcal{C}}_{\mathcal{F}}$ is triangle equivalent to the usual cluster category $\mathcal{C}$.

We briefly describe the geometric realization for the \emph{Auslander-Reiten quiver} of $\mathcal{C}_{\mathcal{F}}$, the cluster category $\mathcal{C}$, and for the module category $\mathrm{mod} \, A_T$.  

\subsubsection{The Frobenius category $\mathcal{C}_{\mathcal{F}}$}\label{sub-frobenuis-category} Take the regular $(n+3)$-gon and label the vertices from 0 to $n+2$ counter-clockwise. Then, there is an indecomposable object $X_{i,j}$ for each diagonal $(i,j)$, where $(i,j)$ and $(j,i)$ represent the same diagonal since they are not considered with orientation, and the indices $i,j$ are considered modulo $n+3$. We include the objects that are represented by edges of the polygon, labeled $(i,i+1)$, these are the \emph{projective-injective objects} for $\mathcal{C}_{\mathcal{F}}$.

 There is an arrow $(i,j) \to (i+1,j)$ and  an arrow $(i,j) \to (i,j+1)$ for each $i,j$ representing an irreducible morphism in the category, as depicted in Figure~\ref{fig:ARquiver}. The mesh category is completed by the following data: the morphims $f$ and $g$ given by the compositions
\[ f = X_{i,j} \to X_{i,j+1} \to X_{i+1,j+1} \ \ \ \mathrm{and} \ \ \ g = X_{i,j} \to X_{i+1,j} \to X_{i+1,j+1}, \]
are linearly dependent. Hence, the dimension of $\mathrm{Hom}(X_{i,j},X_{k,l}) $ as a $K$-vector space can be computed easily observing the possible linearly independent paths on the mesh. In particular, whenever there is a straight diagonal path between two objects, there is  a one-dimensional space of morphisms between them.

\begin{figure}
  \centering
    \begin{tikzpicture}[scale=1]
\draw (0,0)  node{$\scriptstyle (0,2)$};
\draw (-1,-1)  node{$\scriptstyle (0,1)$};
\draw (1,-1)  node{$\scriptstyle (1,2)$};
\draw (1,1)  node{$\scriptstyle (0,3)$};
\draw (2,2)  node{$\scriptstyle \bullet$};
\draw (5,-1)  node{$\scriptstyle \bullet$};
\draw (7,-1)  node{$\scriptstyle (n,n+1)$};
\draw (9,-1)  node{$\scriptstyle (n+1,n+2)$};
\draw (3,3)  node{$\scriptstyle (0,n+1)$};
\draw (11,-1)  node{$\scriptstyle (n+2,0)$};
\draw (4,4)  node{$\scriptstyle (0,n+2)$};

\draw (2,0)  node{$\scriptstyle (1,3)$};
\draw (3,1)  node{$\scriptstyle (1,4)$};
\draw (5,1)  node{$\scriptstyle \bullet$};
\draw (6,0)  node{$\scriptstyle \bullet$};

\draw (5,3)  node{$\scriptstyle (1,n+2)$};
\draw (6,2)  node{$\scriptstyle \bullet$};
\draw (7,1)  node{$\scriptstyle (n-1,n+2)$};
\draw (8,0)  node{$\scriptstyle (n,n+2)$};

\draw [dashed] (0.4,0)--(1.6,0);
\draw [dashed] (2.4,0)--(3.4,0);
\draw [dashed] (4.6,0)--(5.7,0);
\draw [dashed] (6.4,0)--(7.4,0);
\draw [dashed] (8.6,0)--(9.4,0);
\draw [dashed] (10.6,0)--(11.4,0);

\draw [dashed] (1.4,1)--(2.6,1);
\draw [dashed] (5.4,1)--(6.3,1);
\draw [dashed] (7.7,1)--(8.6,1);
\draw [dashed] (9.4,1)--(10.5,1);

\draw [dashed] (2.3,2)--(3.2,2);
\draw [dashed] (4.8,2)--(5.8,2);
\draw [dashed] (6.3,2)--(7.8,2);
\draw [dashed] (8.3,2)--(9.8,2);

\draw [dashed] (3.6,3)--(4.4,3);
\draw [dashed] (5.6,3)--(6.6,3);
\draw [dashed] (7.6,3)--(8.6,3);

\draw [->] (0.2,0.2)--(0.8,0.8);
\draw [->] (1.2,1.2)--(1.8,1.8);
\draw [dotted] (2.2,2.2)--(2.8,2.8);
\draw [->] (3.2,3.2)--(3.8,3.8);
\draw [->] (-0.8,-0.8)--(-0.2,-0.2);

\draw [<-] (0.8,-0.8)--(0.2,-0.2);
\draw [<-] (6.8,-0.8)--(6.2,-0.2);
\draw [<-] (8.8,-0.8)--(8.2,-0.2);
\draw [<-] (10.8,-0.8)--(10.2,-0.2);
\draw [<-] (12.8,-0.8)--(12.2,-0.2);
\draw [->] (1.2,-0.8)--(1.8,-0.2);
\draw [->] (5.2,-0.8)--(5.8,-0.2);
\draw [->] (7.2,-0.8)--(7.8,-0.2);
\draw [->] (9.2,-0.8)--(9.8,-0.2);
\draw [->] (11.2,-0.8)--(11.8,-0.2);
\draw [->](4.2,3.8)--(4.8,3.2);
\draw[dotted](5.2,2.8)--(5.8,2.2);
\draw [->](6.2,1.8)--(6.8,1.2);
\draw [->](7.2,0.8)--(7.8,0.2);

\draw[very thick, dotted](3.7,2)--(4.3,2);
\draw[very thick, dotted](3.7,1)--(4.3,1);
\draw[very thick, dotted](3.7,0)--(4.3,0);
\draw[very thick, dotted](3.7,-1)--(4.3,-1);

\draw [->] (1.2,0.8)--(1.8,0.2);
\draw [->] (2.2,0.2)--(2.8,0.8);
\draw [->] (2.2,1.8)--(2.8,1.2);
\draw [->] (5.2,1.2)--(5.8,1.8);
\draw [->] (5.2,0.8)--(5.8,0.2);
\draw [->] (6.2,0.2)--(6.8,0.8);

\draw (6,4)  node{$\scriptstyle (1,0)$};
\draw (7,3)  node{$\scriptstyle (2,0)$};
\draw (8,4)  node{$\scriptstyle (2,1)$};
\draw (9,3)  node{$\scriptstyle (3,1)$};
\draw (10,2)  node{$\scriptstyle \bullet$};
\draw (8,2)  node{$\scriptstyle \bullet$};
\draw (9,1)  node{$\scriptstyle (n,0)$};
\draw (10,0)  node{$\scriptstyle (n+1,0)$};
\draw (11,1)  node{$\scriptstyle (n+1,1)$};
\draw (12,0)  node{$\scriptstyle (n+2,1)$};
\draw (13,-1)  node{$\scriptstyle (0,1)$};

\draw [->] (5.2,3.2)--(5.8,3.8);
\draw [->] (6.2,3.8)--(6.8,3.2);
\draw [->] (7.2,3.2)--(7.8,3.8);
\draw [dotted] (7.2,2.8)--(7.8,2.2);
\draw[->] (8.2,1.8)--(8.8,1.2);
\draw[->] (9.2,0.8)--(9.8,0.2);
\draw[->] (7.2,1.2)--(7.8,1.8);
\draw[->] (8.2,0.2)--(8.8,0.8);

\draw[->] (10.2,0.2)--(10.8,0.8);

\draw[->](8.2,3.8)--(8.8,3.2);
\draw[dotted](9.2,2.8)--(9.8,2.2);
\draw[->](10.2,1.8)--(10.8,1.2);
\draw[->](9.2,1.2)--(9.8,1.8);
\draw[->] (11.2,0.8)--(11.8,0.2);

\end{tikzpicture}
\caption{Auslander-Reiten quiver of the (Frobenius) cluster category $C_{\mathcal{F}}$ of type $A$.}
    \label{fig:ARquiver}
\end{figure}
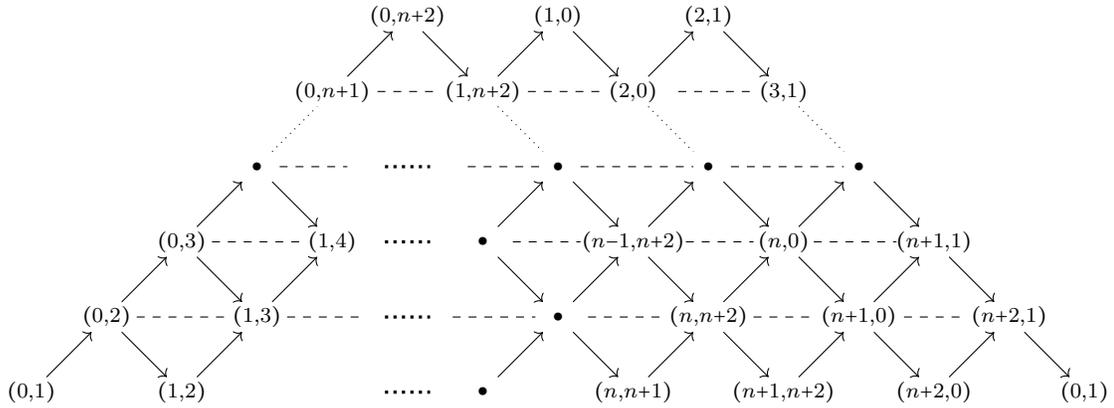

The \emph{Auslander-Reiten} functor $\tau$ is given by $\tau X_{i,j} = X_{i-1,j-1}$, reciprocally $\tau^{-1} X_{i,j} = X_{i+1,j+1}$, whenever $X_{i,j}$ is not projective-injective.

A path $X_1\to \dots \to X_t$ in the Auslander-Reiten quiver is \emph{sectional} if $\tau X_{i+1}\not=X_{i-1}$ for all $i=2, \dots, t-1$. For our category of interest all sectional paths are straight diagonal paths.

\subsubsection{The module category $\mathrm{mod} \, A_T$}\label{subsection_modA} 
Given a triangulation $T$ of the polygon, we can define a Jacobian algebra $A_T$. The Auslander-Reiten quiver for the module category $\mathrm{mod} \, A_T$ is obtained by simply removing the objects $X_{i,j}$ associated to the diagonals in $T$ and the edge objects $X_{i,i+1}$. There is a correspondence 
\[ \mathrm{diagonals \ not \ in} \ T \leftrightarrow \   \mathrm{indecomposable} \ A_T- \mathrm{modules}.\] 
This correspondence follows from a classical result in cluster theory \cite[Theorem 2.2]{BMR} and the above mentioned geometric cluster category \cite{CCS}. For a given diagonal $\gamma \notin T$, the dimension vector $\underline{\mathrm{dim}} (M_\gamma)$ is given by the crossings between $\gamma$ and $T$. In particular, the indecomposable projective module $P_a$ associated to the diagonal $a=(i+1,j+1)$ appears as $\tau^{-1} X_{i,j} = X_{i+1,j+1}$. In this case the object $X_{i,j}$, associated to the diagonal $(i,j)\in T$, is called the \emph{shifted projective} $P_a [1]$. 


\section{Tensoring with the algebra of dual numbers}
\label{Sec:tensoring}

Let $\La$ be a finite dimensional algebra over a field $K=\overline{K}$. Also, let $K[\epsilon]/(\epsilon^2)$ denote the 2-dimensional local algebra called the algebra of dual numbers, obtained by taking the quotient of the polynomial ring $K[\epsilon]$. We consider the tensor product of $\La$ with the dual numbers
\[\Ltil \colon = \La \otimes_K K[\epsilon]/(\epsilon^2).\]
Ringel and Zhang studied homological properties of $\Ltil$ in \cite{ringel17}. Later, a generalization of this algebra appeared in \cite{geiss17}, where instead of the dual numbers the authors considered $K[\epsilon]/(\epsilon^m)$ and proposed a way to associate quiver representations to Cartan matrices of Dynkin type with nontrivial symmetrizers.  

Note that $\Ltil$ is isomorphic to $\La\oplus \La$ as a $\La$-module.  Moreover, it is easy to see that if $\La=KQ/I$ then the quiver of $\Ltil$ is obtained from $Q$ by adding a loop $\epsilon_i$ for every vertex $i$ and the relations of $\Ltil$ are the same as in $\La$ together with $\epsilon_i^2=0$ and $\epsilon_i\alpha=\alpha \epsilon_j$ for every arrow $\alpha$ in $Q$ starting at $i$ and ending in $j$. See Example \ref{example_FN}. Thus, we obtain the following result. 

\begin{lemma}\label{lemma_bilinear_equal}
    The antisymmetrized bilinear form applied to simple modules over the algebras $\La$ and $\Ltil$ coincide, that is $\langle S_i, S_j\rangle_{\La}=\langle S_i, S_j\rangle_{\Ltil}$ for all $i,j$.
\end{lemma}

\begin{proof}
By definition of the bilinear form the statement clearly holds if $i=j$.  Now, suppose that $i\not=j$.  In this case, $\text{Hom}_{\La} (S_i, S_j)=\text{Hom}_{\Ltil} (S_i, S_j)=0$, and it suffices to show that $\text{dim Ext}^1_{\La}(S_i,S_j)=\text{dim Ext}^1_{\Ltil}(S_i,S_j)$.  The dimension of $\text{Ext}^1(S_i,S_j)$ equals the number of arrows from $i$ to $j$. By the discussion above, the quiver of $\Ltil$ is obtained from that of $\La$ by adding a loop at each vertex, hence the two quivers have the same number of arrows between any pair of distinct vertices.  This shows the desired claim that the two extension spaces are isomorphic. 
\end{proof}

Since $\La\cong \La\otimes_K 1$ is a subalgebra of $\Ltil=\La\otimes_K K[\epsilon]/(\epsilon^2)$, and the two algebras share the same identity, there is a general construction of induction and restriction functors between their module categories.  The \emph{induction functor} is defined as follows.
\[ -\otimes_{\La} \Ltil\colon \,\, \textup{mod}\,\La \to \textup{mod}\,\Ltil\]
Since $\Ltil$ considered as a $\La$-module is projective, then the induction functor is exact.  Moreover, it takes an indecomposable projective (resp. injective) $\La$-module at vertex $i$, to an indecomposable projective (resp. injective) $\Ltil$-module at that same vertex.  If a $\Ltil$-module is in the image of the induction functor, i.e. it is of the form $M\otimes_{\La} \Ltil$, then we say it is an \emph{induced module} and we denote it by $\widetilde{M}$. In the language of \cite{geiss17}, the induced modules are also known as \emph{locally free} modules.

The restriction functor $\text{mod}\,\Ltil\to \text{mod}\,\La$ is defined by taking a $\Ltil$-module $M_{\Ltil}$ and making it a $\La$-module $M_{\La}$ by restricting the scalars from $\Ltil$ to $\La$.  It is easy to see that the restriction of the induced module $\widetilde{M_{\La}}$ is isomorphic to $M_{\La}\oplus M_{\La}$.  The induction and restriction functors form an adjoint pair. 

The following lemma says that the index behaves well under induction. 

\begin{lemma}\label{lem:index}
The module $M\in\textup{mod}\,\La$ and its induced module have the same index, that is $\mathrm{ind}_{\La} (M)=\mathrm{ind}_{\Ltil} (\widetilde{M})$.
\end{lemma}

\begin{proof}
Let $M\in\textup{mod}\,{\La}$ and consider $0\to M \to I^0\to I^1$ the minimal injective presentation of $M$ in $\text{mod}\,\La$. Applying the induction functor to this exact sequence yields  
\[0\to M\otimes_{\La} \Ltil \to I^0\otimes_{\La}\Ltil \to I^1\otimes_{\La} \Ltil\] 
in $\text{mod}\,\Lambda$. This is a minimal injective presentation of  $\widetilde{M} = M\otimes_{\La}\Ltil$, since the induction functor is exact and it maps indecomposable injective $\La$-modules to indecomposable injective $\Ltil$-modules at the same vertex.  Thus, by definition of the index we conclude that $\mathrm{ind}_{\La} (M)=\mathrm{ind}_{\Ltil} (\widetilde{M})$.
\end{proof}

 The $\tau$-tilting theory introduced and studied in \cite{AIR} is a generalization of the classical tilting theory of Brenner and Butler. In the case of Jacobian algebras, $\tau$-tilting theory captures the structure of the associated cluster algebra. We review some of the relevant definitions below.

A module $M\in \text{mod}\, \La$ is called \emph{$\tau$-rigid} if $\text{Hom}_{\La}(M, \tau M)=0$. A pair of objects $M\oplus P[1]$, where $M, P\in \text{mod}\,\La$ and $P$ is projective, is called \emph{support $\tau$-tilting} if $M$ is $\tau$-rigid, $\text{Hom}_{\La}(P,M)=0$, and $\left| M\oplus P\right|$ equals the rank of $\La$. The term $P[1]$ called the shifted projective and $M$ is called a support $\tau$-tilting module. In general, the support $\tau$-tilting pair is uniquely determined by $M$.

Next, we observe that the two algebras $\La$ and $\Ltil$ have the same $\tau$-tilting structure.  First, we recall the main result of \cite{EJR}.

\begin{theorem}\cite[Theorem 1]{EJR}\label{thm:ejr}
For an ideal $I$ which is generated by central elements and contained in the Jacobson radical of $\La$, the $g$-vectors of indecomposable $\tau$-rigid (respectively support $\tau$-tilting) modules over $\La$ coincide with the ones for $\La / I$.
\end{theorem}

 As an application of the above theorem we obtain the following. 

\begin{corollary}
There is  a bijection between support $\tau$-tilting $\La$-modules and support $\tau$-tilting $\Ltil$-modules given by the induction functor. 
\end{corollary}

\begin{proof}
First, we claim that the induction functor maps $\tau$-rigid $\La$-modules to $\tau$-rigid $\Ltil$-modules.  Suppose that $M\in \text{mod}\,\La$ is $\tau$-rigid, and let $f_{\Ltil} \colon M\otimes_{\La} \Ltil \to \tau_{\Ltil} (M\otimes_{\La}\Ltil)$ be some morphism in $\text{mod}\,\Ltil$. Note that since the induction functor is exact, maps projectives to projectives, and injectives to injectives, we see that $\tau$ commutes with induction and so  $\tau_{\Ltil} (M\otimes_{\La}\Ltil)\cong (\tau_{\La} M)\otimes_{\La}\Ltil$.
Applying the restriction functor to $f_{\Ltil}$, we obtain $f_{\La}: M\oplus M \to \tau_{\La} M \oplus \tau_{\La} M$.  

Since $M$ is $\tau$-rigid, we conclude that $f_{\La}=0$.  Hence, $f_{\Ltil}=0$ since it is the same map as $f_{\La}$ on the level of vector spaces.  This shows the desired claim that the induced module $M\otimes_{\La}\Ltil$ is $\tau$-rigid.

This implies that the induction functor gives an inclusion of the support $\tau$-tilting modules from $\text{mod}\,\Lambda$ to $\text{mod}\,\Ltil$. Now, observe that $1\otimes \epsilon$ is a central element of $\Ltil$ and is contained in the Jacobson radical of $\La$.  Moreover, $\Ltil/\langle 1\otimes \epsilon \rangle \cong \La$.
By Theorem~\ref{thm:ejr} we conclude that the induction functor is also surjective onto $\tau$-rigid $\Ltil$-modules.  Note that by a similar reasoning as in the proof of Lemma~\ref{lem:index}, $M$ and $M\otimes_{\Lambda}\Ltil$ have the same $g$-vectors in $\text{mod}\,\Lambda$ and $\text{mod}\,\Ltil$ respectively. 
\end{proof}

Next, we construct a certain $\Lambda$-module $F(N)$ associated to a submodule $N$ of an induced module. Later $F(N)$ will correspond to the dimers depicted by single blue lines in a certain double dimer cover associated to $N$.

Let $N\in\text{mod}\,\Ltil$, then we define $N\epsilon$ to be the product of $N$ with the ideal of $\Ltil$ generated by $1\otimes \epsilon$.  In particular, $N\epsilon = \{n(1\otimes \epsilon)\mid n\in N\}$ is a submodule of $N$, and since $1\otimes \epsilon$ acts trivially on $N\epsilon$ then $N\epsilon$ is actually an $\Lambda$-module.  In particular, we have the following short exact sequence in $\text{mod}\,\Ltil$
\[0\to N\epsilon \to N \to N/N\epsilon\to 0.\]
Note that if $N$ is an induced module then the restriction $N_{\Lambda}\cong N\epsilon \oplus N/N\epsilon$, and moreover $N\epsilon\cong N/N\epsilon$.

Now, suppose that $N\in \text{mod}\,\Ltil$ is a submodule of an induced module $M\otimes_{\Lambda}\Ltil$ for some $M\in\text{mod}\,\Lambda$.

Observe that 
\[M\otimes_{\Lambda}\Ltil = M\otimes_{\Lambda} \Lambda\otimes_K K[\epsilon]/(\epsilon^2)\cong M\otimes_K K[\epsilon]/(\epsilon^2).\]

Then we can represent the elements of $M\otimes_{\Lambda}\Ltil$ by $\{m\otimes e\mid m\in M \text{ and } e\in K[\epsilon]/(\epsilon^2)\}$ where the action of $\Ltil$ is given as follows $(m\otimes e)\cdot (a\otimes e')=(ma\otimes ee')$.  Since $N$ is a submodule of $M\otimes_{\Lambda}\Ltil$, then a direct computation implies that there exist $N_1\leq N_2$ submodules of $M$ such that
\[N=\{n_1\otimes 1 + n_2\otimes \epsilon \mid n_1\in N_1, n_2\in N_2\}.\] 
Then we see that $N\epsilon = \{n_1\otimes \epsilon\mid n_1\in N_1\}$ which is isomorphic to $N_1$ as a $\Lambda$-module.  By the description of $N$ above we see that the induced module $N\epsilon \otimes_{\Lambda} \Ltil = \widetilde{N \epsilon}$ is a submodule of $N$, and moreover it is the largest induced submodule of $N$.  This allows us to define the desired quotient 
\[F(N)\colon = N/(\widetilde{N \epsilon}) \in \textup{mod}\,\Lambda.\]

Observe that if $N=N_{\Lambda}$ is already a $\Lambda$-module, then $N\epsilon=0$ and $F(N)=N$.  On the other hand, if $N$ is an induced module then $\widetilde{N \epsilon}=N$ and $F(N)=0$.

\begin{example}\label{example_FN}
Let $\Lambda$ be the path algebra $1\xrightarrow{\alpha}2$, then $\Ltil$ is given by the following quiver 
\begin{center}
 \begin{tikzcd}[arrow style=tikz,>=stealth,row sep=4em]
1 \arrow[out=60,in=120, distance=1.5em,loop,"\epsilon_1", swap]\arrow[rr,"\alpha"]
  && 2 \arrow[out=60,in=120, distance=1.5em,loop,"\epsilon_2", swap]
\end{tikzcd}   
\end{center}
with relations $\epsilon_1^2=\epsilon_2^2=0$ and $\epsilon_1\alpha=\alpha\epsilon_2$. Observe that the $\Ltil$-module  $M=\begin{smallmatrix}&1\\1&&2\\&2\end{smallmatrix}$ is induced from the $\Lambda$-module $\begin{smallmatrix}1\\2\end{smallmatrix}$.  Consider a submodule $N=\begin{smallmatrix}1&&2\\&2\end{smallmatrix}$ of $M$, and note that its restriction is $N_{\Lambda}=\begin{smallmatrix}1\\2\end{smallmatrix}\oplus 2$.  We observe that $N\epsilon=2$ and its induction $\widetilde{N \epsilon}=\begin{smallmatrix}2\\2\end{smallmatrix}$, so $F(N)=1$.
\end{example}

\subsection{Associating $\mu$-invariants to $N$.}
Let $T$ be a triangulation of a polygon, $\Lambda=A_T$ the corresponding Jacobian algebra and $\mathcal{C}_\mathcal{F}$ the corresponding Frobenius cluster category. We use $F(N)$ to associate a product of $\theta$'s to a submodule $N$ of an induced module.  

We recall that a path $M_1\to \dots \to M_t$ in the Auslander-Reiten quiver is sectional if $\tau M_{i+1}\not=M_{i-1}$ for all $i=2, \dots, t-1$. 
We say that three distinct objects $M_1,M_2,M_3\in \text{ind}\,\mathcal{C}_\mathcal{F}$ form a triangle if there exist sectional paths $\rho_i: M_i \to \dots \to M_{i+1}$ for all $i = 1,2,3$, 
where we consider indices modulo 3, such that the composition $\rho_i\rho_{i+1}$ is not sectional. 
Moreover, let $\Delta(M_1,M_2,M_3)$ denote the set of all objects that lie on one of the sectional paths $\rho_i$ for $i = 1,2,3$.  It is easy to see that three objects form a triangle if and only if  their corresponding diagonals form a triangle in the polygon in the geometric model. 
Thus, we also call $\Delta(M_1,M_2,M_3)$ a \emph{triangle}. 
Given a triangulation $T=\{\gamma_1, \gamma_2,\dots,\gamma_n\}$, let $\Delta(T)$ denote the set of all triangles determined by shifted projectives for $A_T$ and projective-injectives in $\mathcal{C}_{\mathcal{F}}$. Such triangles are of the form $\Delta(X,Y,Z)$, if $T$ has no internal triangles, among the three $X,Y,Z \in \mathrm{ind} \, \mathcal{C}_{\mathcal{F}}$, there are either one or two shifted projectives and, in correspondence, either two or one projective-injective objects.

With the notation above, we have the following statement. 

\begin{proposition}\label{prop:triangle}
Every indecomposable object $M_{\gamma}$ in $\mathcal{C}_\mathcal{F}$ that is not projective-injective belongs to exactly two triangles of $\Delta(T)$. If $\gamma$ belongs to $T$, these are the ones corresponding to the two triangles bounded by $\gamma$. Otherwise, they are the ones corresponding to the first and the last triangle crossed by $\gamma$ in the triangulated polygon.
\end{proposition}

\begin{proof}
Let $M\in \text{ind} \, \mathcal{C}_\mathcal{F}$ be a non projective-injective object.  Suppose that $M=M_\gamma$ corresponds to an arc $\gamma$ in the triangulated polygon. If $\gamma\in T$, then it belongs to exactly two triangles of $T$, which means that it belongs to exactly two triangles of $\Delta(T)$,  so the statement follows.  Now, suppose that $\gamma\not\in T$.
By definition, $M_\gamma$ lies on a sectional path $\rho: P_i[1]\to \dots \to P_j[1]$ if and only if the arcs $\gamma, \gamma_i, \gamma_j$ share a common vertex $x$  and one can pass from $\gamma_i$ to $\gamma$ and then to $\gamma_j$ in the polygon by moving the other endpoint of $\gamma_i$ counterclockwise around the boundary of the polygon without passing through $x$.  This implies that $M_\gamma$ belongs to exactly two triangles of $\Delta(T)$, which correspond to the first and last triangle that $\gamma$ passes through. 
\end{proof}

See Example \ref{eg_A_2_sectional} and Figure \ref{fig:AR_eg} for an example in type $A_2$.

We now give an ordering of the $\mu$-invariants from a representation theory point of view, that is using the triangles in $\Delta(T)$ and describe a way to associate $\mu$-invariants to each submodule $N$ of an induced module $\widetilde{M}$. Recall that each of these triangles $\Delta$ corresponds to a triangle in the triangulated polygon and hence to a $\mu$-invariant which we denote $\theta_\Delta$.

Assume now that $T$ is a fixed triangulation of a disk with $n+3$ marked points with no internal triangles, so that there exists a longest arc $\gamma=(s,t)$ crossing all the arcs of the triangulation. We fix an orientation of $\gamma$ and let $\gamma_1$ denote the first arc in $T$ crossed by $\gamma$, that is this arc corresponds to the shifted projective $P_1[1]$. Note that $T$ has exactly two \textit{ears}, that is triangles with two boundary edges, one of which contains the source $s$ of $\gamma$ and the other its target $t$, while all other triangles have exactly one boundary edge. Number the marked points $0,\,1,\dots, n+2$ in the counter-clockwise direction. Then the Auslander-Reiten quiver of $\mathcal{C}_{\mathcal{F}}$ is shown in Figure \ref{fig:ARquiver}.

Say that $s$ lies in the ear delimited by marked points $x,\, x+1, \, x+2$, where $s=x+1$  and sums are taken modulo $n+3$. This corresponds to the triangle $\Delta_1=\Delta(( x+1, x+2),(x+1, x),(x+2, x))$ in $\Delta(T)$:
\begin{align*}
    ( x+1, x+2)\rightarrow \dots \rightarrow (x+1, x)\rightarrow (x+2, x)\rightarrow ( x+1, x+2),
\end{align*}
where $( x+1, x+2),(x+1, x)$ are boundary edges and $(x+2, x)$ corresponds to the shifted projective $P_1[1]$. Note that the sectional path
\begin{align*}
    \rho_1: ( x+1, x+2)\rightarrow \dots \rightarrow (x+1, x)
\end{align*}
is longest, in the sense that it consist of a full ascending diagonal $d$ in the Auslander-Reiten quiver in Figure \ref{fig:ARquiver}. Excluding its endpoints, this diagonal $d$ contains $n$ objects, each of which belongs to $\Delta_1$ and exactly another triangle in $\Delta (T)$ by Proposition \ref{prop:triangle}. Label these triangles $\Delta_2,\, \dots,\, \Delta_{n+1}$ from the one containing $(x+1, x+3)$ to the one containing $(x+1, x-1)$, see Figure \ref{fig:delta_ordering}.

\begin{figure}
  \centering
    \begin{tikzpicture}[scale=1]


\draw (-1,-3)  node{$\scriptstyle (x+1,x+2)$};
\draw (-1.5,-3.5)  node[blue]{$\scriptstyle \Delta_1$};
\draw (0,-2)  node{$\scriptstyle (x+1,x+3)$};
\draw (-1,-1.7)  node[blue]{$\scriptstyle \Delta_2$};
\draw (1,-1)  node{$\scriptstyle (x+1,x+4)$};
\draw (0,-0.7)  node[blue]{$\scriptstyle \Delta_3$};
\draw (2,0)  node{$\udots$};
\draw (3,1)  node{$\udots$};
\draw (4,2)  node{$\scriptstyle (x+1,x-2)$};
\draw (3,2.3)  node[blue]{$\scriptstyle \Delta_{n}$};
\draw (5,3)  node{$\scriptstyle (x+1,x-1)$};
\draw (4,3.3)  node[blue]{$\scriptstyle \Delta_{n+1}$};
\draw (6,4)  node{$\scriptstyle (x+1,x)$};

\draw [->] (-0.8,-2.8)--(-0.2,-2.2);
\draw [->] (0.2,-1.8)--(0.8,-1.2);
\draw [->] (1.2,-0.8)--(1.8,-0.2);
\draw [dashed] (2.2,0.2)--(2.8,0.8);
\draw [->] (3.2,1.2)--(3.8,1.8);
\draw [->] (4.2,2.2)--(4.8,2.8);
\draw [->] (5.2,3.2)--(5.8,3.8);

  \end{tikzpicture}
 \caption{The diagonal $d$, belonging to the triangle $\Delta_1$. Excluding its first and last element, the rest of the elements in $d$ belong to a second triangle in $\Delta(T)$, indicated in blue next to it.}
    \label{fig:delta_ordering}
\end{figure}
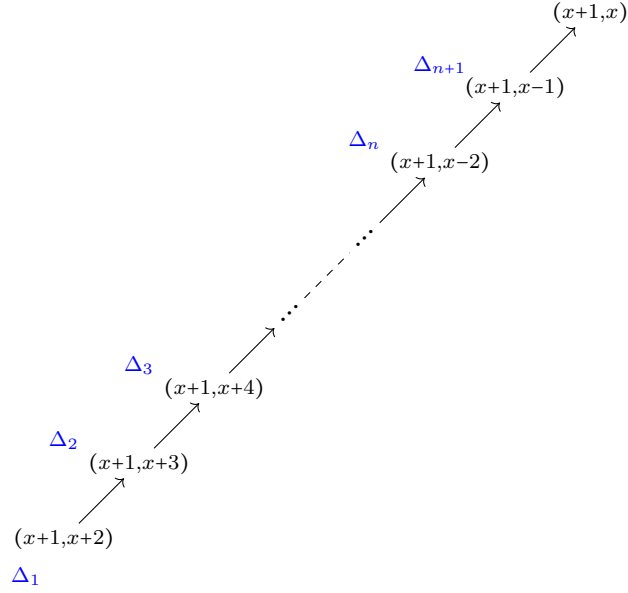

\begin{definition}
Following the above construction, and letting $\theta_i$ denote the $\mu$-invariant corresponding to triangle $\Delta_i$, we define the \textit{$\Delta$-positive ordering of the $\mu$-invariants} to be
\begin{itemize}
    \item $\theta_1> \theta_2 >\dots >\theta_n>\theta_{n+1}$ if the indecomposable projective $P_1$ in $\text{mod}\,A_T$ is simple,
    \item $\theta_2 >\dots >\theta_n>\theta_{n+1}>\theta_1$ if $P_1$ is not simple.
\end{itemize}
\end{definition}

Thanks to Proposition~\ref{prop:triangle} we can make the following definition.

\begin{definition}\label{defn_mu_module}
Let $M$ be an indecomposable module in $\text{mod}\,A_T$, then  $M\in \Delta\cap \Delta'$ for some triangles $\Delta, \Delta' \in \Delta(T)$.  Define $\mu(M)=\theta_{\Delta}\theta_{\Delta'}$ where we order the $\theta$'s according to the $\Delta$-positive ordering of the $\mu$-invariants.
Moreover, we extend this definition to arbitrary modules as follows. If $M = \bigoplus_i M_i$ is a finite direct sum of indecomposable $A_T$-modules $M_i$, then we define $\mu(M)=\Pi_i \mu(M_i)$. 
\end{definition}

\begin{remark}\label{remark_summands_ordering}
Note that even though the $\mu$-invariants anti-commute, the definition of $\mu(M)$ is independent of the ordering of the summands $M_i$.  Indeed, since $\mu(M_i)$ is a product of two $\theta$'s we see that $\mu(M_i)\mu(M_j)=\mu(M_j)\mu(M_i)$.
\end{remark}

We associate a $\mu$-invariant to any submodule $N$ of an induced module by passing to $F(N)$. 

\begin{definition}
Let $\Lambda=A_T$ and $\Ltil$ be the tensor algebra of $\Lambda$ with the dual numbers.  If $N\in \text{mod}\,\Ltil$ is a submodule of an induced module, then define $\mu(N)=\mu(F(N))=\mu(N/\widetilde{N\epsilon})$.
Recall that if $N$ is an induced module, then $F(N)=0$ and so $\mu(F(N))=1.$
\end{definition}

\subsection{Comparing the two orderings.}
We conclude this section by showing that the $\Delta$-positive ordering of the $\mu$-invariants coincides with the positive ordering by Musiker, Ovenhouse and Zhang we recalled in Definition~\ref{defn:positive_ordering}. First, we give an alternative description of this.

\begin{remark}\label{remark_positive_ordering_thetas}
Like in the above construction, say that the source of $\gamma$ lies in the ear delimited by marked points $x,\, x+1, \, x+2$, where $x+1$ is the source of $\gamma$ and $\theta_1$ is the $\mu$-invariant corresponding to this ear. The other ear of the triangulation is then delimited by the marked points $x+j,\, x+j+1, \, x+j+2$ for some $2\leq j\leq n+3$, where sums are taken modulo $n+3$. See Figure \ref{fig:ordering_circle} for an illustration of the following.
\begin{itemize}
    \item For $1\leq i\leq j$, let $\theta_i'$ denote the $\mu$-invariant corresponding to the triangle containing the boundary edge $(x+i, x+i+1)$.
    \item  For $j+1\leq i\leq n+1$, let $\theta_i'$ denote the $\mu$-invariant corresponding to the triangle containing the boundary edge $(x+i+1, x+i+2)$.
\end{itemize}
The positive ordering of $T$ from Definition \ref{defn:positive_ordering} is then:
\begin{itemize}
    \item $\theta_1> \theta_2' >\dots >\theta_n'>\theta_{n+1}'$ if $\theta_2'$ corresponds to the second triangle bounded by $(x, x+2)$,
    \item $\theta_2' >\dots >\theta_n'>\theta_{n+1}'>\theta_1$ otherwise.
\end{itemize}
\end{remark}

     \begin{figure}
  \centering
    \begin{subfigure}
    \centering
    \begin{tikzpicture}[scale=2]
      \draw (0,0) circle (1cm); 
      
     \draw (330:0.97cm) -- (330:1.03cm);
     \draw (300:0.97cm) -- (300:1.03cm);
     \draw (260:0.97cm) -- (260:1.03cm);
     \draw (220:0.97cm) -- (220:1.03cm);
     \draw (160:0.97cm) -- (160:1.03cm);
     \draw (130:0.97cm) -- (130:1.03cm);
     \draw (90:0.97cm) -- (90:1.03cm);
     \draw (50:0.97cm) -- (50:1.03cm);
     \draw (20:0.97cm) -- (20:1.03cm);

     \draw (20:1.25cm) node{$\scriptstyle x+n+2$};
      \draw (50:1.18cm) node{$\scriptstyle x$};
      \draw (90:1.18cm) node{$\scriptstyle x+1$};
      \draw (130:1.18cm) node{$\scriptstyle x+2$};
      \draw (160:1.18cm) node{$\scriptstyle x+3$};
      \draw (220:1.18cm) node{$\scriptstyle x+j$};
      \draw (260:1.18cm) node{$\scriptstyle x+j+1$};
      \draw (300:1.18cm) node{$\scriptstyle x+j+2$};
      \draw (330:1.22cm) node{$\scriptstyle x+j+3$};
      
      \draw (90:0.85cm) node[blue]{$\scriptstyle \theta_1$};
      \draw (145:0.9cm) node[blue]{$\scriptstyle \theta_2'$};
      \draw (260:0.9cm) node[blue]{$\scriptstyle \theta_j'$};
      \draw (315:0.85cm) node[blue]{$\scriptstyle \theta_{j+1}'$};
      \draw (40:0.85cm) node[blue]{$\scriptstyle \theta_{n+1}'$};
      \draw (50:1cm) -- (130:1cm);
      \draw (220:1cm) -- (300:1cm);
      \draw (50:1cm) -- (160:1cm);
      \draw[very thick, dashed, white] ([shift=(170:1cm)]0,0) arc (170:210:1cm);
        \draw[very thick, dashed, white] ([shift=(340:1cm)]0,0) arc (340:370:1cm);
    \draw (-90:1.5cm) node{$\scriptstyle \theta_1>\theta_2'>\dots>\theta_j'>\theta_{j+1}'>\dots >\theta_{n+1}'$};
    \end{tikzpicture}
    \end{subfigure}
    \hspace{3em}
    \begin{subfigure}
    \centering
    \begin{tikzpicture}[scale=2]
      \draw (0,0) circle (1cm); 
      
     \draw (330:0.97cm) -- (330:1.03cm);
     \draw (300:0.97cm) -- (300:1.03cm);
     \draw (260:0.97cm) -- (260:1.03cm);
     \draw (220:0.97cm) -- (220:1.03cm);
     \draw (160:0.97cm) -- (160:1.03cm);
     \draw (130:0.97cm) -- (130:1.03cm);
     \draw (90:0.97cm) -- (90:1.03cm);
     \draw (50:0.97cm) -- (50:1.03cm);
     \draw (20:0.97cm) -- (20:1.03cm);

     \draw (20:1.25cm) node{$\scriptstyle x+n+2$};
      \draw (50:1.18cm) node{$\scriptstyle x$};
      \draw (90:1.18cm) node{$\scriptstyle x+1$};
      \draw (130:1.18cm) node{$\scriptstyle x+2$};
      \draw (160:1.18cm) node{$\scriptstyle x+3$};
      \draw (220:1.18cm) node{$\scriptstyle x+j$};
      \draw (260:1.18cm) node{$\scriptstyle x+j+1$};
      \draw (300:1.18cm) node{$\scriptstyle x+j+2$};
      \draw (330:1.22cm) node{$\scriptstyle x+j+3$};
      
      \draw (90:0.85cm) node[blue]{$\scriptstyle \theta_1$};
      \draw (145:0.9cm) node[blue]{$\scriptstyle \theta_2'$};
      \draw (260:0.9cm) node[blue]{$\scriptstyle \theta_j'$};
      \draw (315:0.85cm) node[blue]{$\scriptstyle \theta_{j+1}'$};
      \draw (40:0.85cm) node[blue]{$\scriptstyle \theta_{n+1}'$};
      \draw (50:1cm) -- (130:1cm);
      \draw (220:1cm) -- (300:1cm);
      \draw (130:1cm) -- (20:1cm);
      \draw[very thick, dashed, white] ([shift=(170:1cm)]0,0) arc (170:210:1cm);
      \draw[very thick, dashed, white] ([shift=(340:1cm)]0,0) arc (340:370:1cm);
      \draw (-90:1.5cm) node{$\scriptstyle \theta_2'>\dots>\theta_j'>\theta_{j+1}'>\dots >\theta_{n+1}'>\theta_1$};
    \end{tikzpicture}
    \end{subfigure}
    \caption{The $\mu$-invariants associated to each triangle according to the notation from Remark \ref{remark_positive_ordering_thetas}. Note that $\theta_1$ and $\theta_j'$ correspond to the two ears, while the remaining $\mu$-invariants are indicated to the single boundary edge in the corresponding triangle. The two different cases are illustrated in the pictures.}
    \label{fig:ordering_circle}
\end{figure}
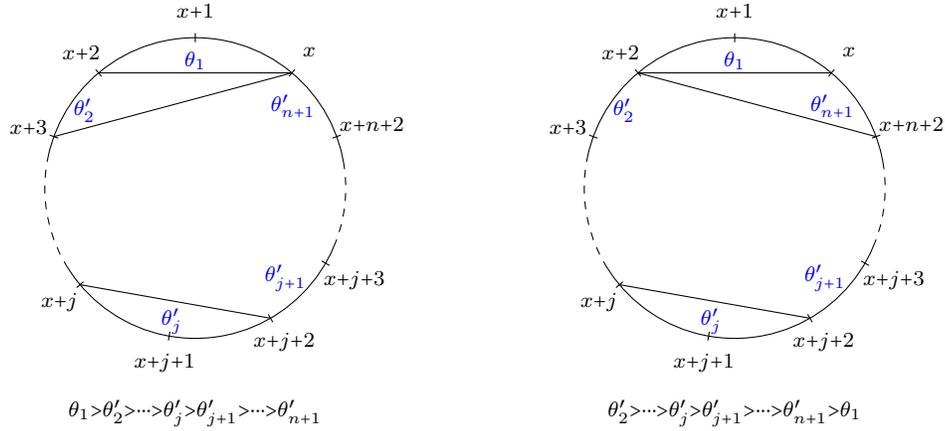

\begin{proposition}\label{prop_orderings_equal}
The $\Delta$-positive ordering coincides with the positive ordering.
\end{proposition}
\begin{proof}
Throughout this proof, we use the notation fixed above.
Consider first the triangle corresponding to $\theta_1$, that is the one delimited by the marked points $x,\, x+1,\,x+2$, where $x+1$ is the source of the longest arc $\gamma$. The third vertex of the other triangle in $T$ delimited by the arc $( x,x+2 )$ can either be $x+3$ or $x-1$. Note that the first case holds exactly when $\theta_2'$ corresponds to the second triangle bounded by $(x, x+2)$. By Remark \ref{remark_positive_ordering_thetas}, in the positive ordering we then have $\theta_1$ bigger than all other $\mu$-invariants in the first case, and $\theta_1$ smaller than all other $\mu$-invariants in the second case. In the quiver $Q_T$, constructed as described in Definition \ref{def:initial seed-cluster algebra}, these two cases correspond respectively to
\begin{center}
 \begin{tikzcd}[arrow style=tikz,>=stealth,row sep=4em]
\cdots \, \,  2 \arrow[r]
  & 1 & \text{ and }& \cdots \, \, 2 &1 \arrow[l],
\end{tikzcd}   
\end{center}
where $1$ corresponds to the arc $( x,x+2 )$ by construction and the rest of the quiver sits to the left of $2$. The first case corresponds to the projective $P_1$ at vertex $1$ being simple while the second to $P_1$ being not simple.

It is now enough to prove that $\theta_i=\theta_i'$ for $2\leq i\leq n+1$.
Consider first $2\leq i\leq j$ and recall that by construction the triangle corresponding to $\theta_i'$ contains the arc $( x+i, x+i+1)$ and its third vertex is $x+k$ for some $j+2\leq k\leq n+3$. The $i$th descending diagonal in the Auslander Reiten quiver crossing $d$ crosses $d$ at $( x+1, x+i+1 )$ and since $i<j+2\leq k$, the part of this diagonal between $( x+k, x+i+1)$ and $( x+i, x+i+1)$ contains $( x+1, x+i+1 )$, and belongs hence to $\Delta_i$, see Figure \ref{fig:proof_AR1}. Hence $\theta_i=\theta_i'$ for $2\leq i\leq j$.
\begin{figure}
  \centering
    \begin{tikzpicture}[scale=1]
\draw (-1,1)  node{$\ddots$};
\draw (0,0)  node{$\scriptstyle (x,x+2)$};

\draw (0,6)  node{$\scriptstyle (x+i+2,x+i+1)$};
\draw (1,5)  node{$\ddots$};
\draw (2,4)  node[red]{$\scriptstyle (x+k,x+i+1)$};
\draw (3,3)  node[red]{$\ddots$};
\draw (5,1)  node[red]{$\ddots$};
\draw (6,0)  node[red]{$\scriptstyle (x+i-1,x+i+1)$};
\draw (7,-1)  node[red]{$\scriptstyle (x+i,x+i+1)$};

\draw (1,-1)  node{$\scriptstyle (x+1,x+2)$};
\draw (2,0)  node{$\scriptstyle (x+1,x+3)$};
\draw (3,1)  node{$\udots$};
\draw (4,2)  node[red]{$\scriptstyle (x+1,x+i+1)$};
\draw (5,3)  node{$\udots$};

\draw [->](-0.8,0.8)--(-0.2,0.2);
\draw [->](0.2,-0.2)--(0.8,-0.8);

\draw [->](0.2,5.8)--(0.8,5.2);
\draw [->](1.2,4.8)--(1.8,4.2);
\draw [->, red](2.2,3.8)--(2.8,3.2);
\draw [->,red](3.2,2.8)--(3.8,2.2);
\draw [->,red](4.2,1.8)--(4.8,1.2);
\draw [->,red](5.2,0.8)--(5.8,0.2);
\draw [->,red](6.2,-0.2)--(6.8,-0.8);

\draw [->] (1.2,-0.8)--(1.8,-0.2);
\draw [->] (2.2,0.2)--(2.8,0.8);
\draw [->] (3.2,1.2)--(3.8,1.8);
\draw [->] (4.2,2.2)--(4.8,2.8);

  \end{tikzpicture}
 \caption{For $2\leq i\leq j$, the elements marked in red in the $i$th diagonal crossing $d$ belong to $\Delta_i$.}
    \label{fig:proof_AR1}
\end{figure}
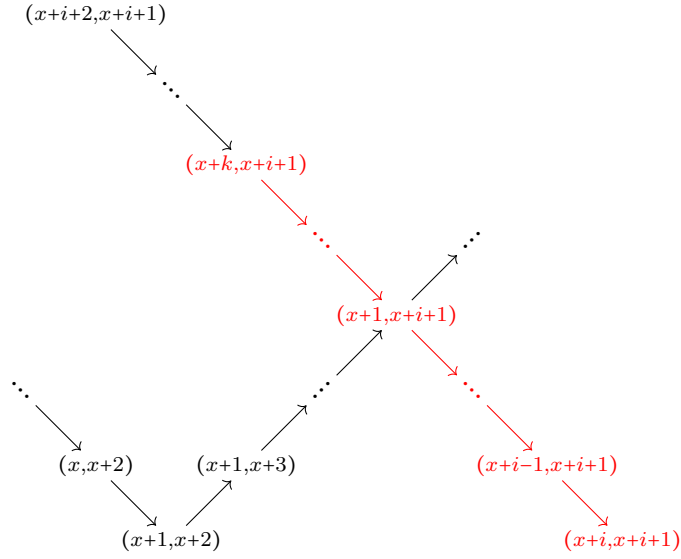
Let now $j+1\leq i\leq n+1$ and recall that by construction the triangle corresponding to $\theta_i'$ contains the arc $( x+i+1, x+i+2)$ and its third vertex is $x+l$ for some $2\leq l\leq j$. The $i$th descending diagonal crossing $d$ crosses $d$ at $( x+1, x+i+1 )$ and since $l<j+1\leq i $, the part of this diagonal between $( x+i+2, x+i+1)$ and $( x+l, x+i+1)$, contains $( x+1, x+i+1 )$, and belongs hence to $\Delta_i$. Hence $\theta_i=\theta_i'$ for all $i$.\end{proof}

\begin{example}\label{eg_A_2_sectional}
Consider the triangulation of the pentagon from Figure \ref{fig:triang_polygon_eg}. The Auslander-Reiten quiver of the corresponding $\mathcal{C}_\mathcal{F}$ is shown in Figure \ref{fig:AR_eg}, where the arc $(0,2)$ corresponds to $P_1$, $(0,3)$ to $I_1$, $(1,3)$  to $P_2[1]$, $(1,4)$  to $P_1[1]$ and $(2,4)$  to $P_2$. In this case, the source of the longest arc $\gamma$ lies in the ear delimited by the vertices $x=4, x+1=0, x+2=1$. Then the triangle $\Delta_1=\Delta((0,1), (0,4), (1,4))$ is highlighted in red in Figure \ref{fig:AR_eg} and the triangles $\Delta_2=\Delta((1,3), (1,4), (3,4))$ and $\Delta_3=\Delta((1,2), (1,3), (2,3))$ are highlighted in blue and green respectively. Then, noting that $P_1$ is not simple, we have that the $\Delta$-positive ordering is $\theta_3>\theta_2>\theta_1$, agreeing with the classic positive ordering from Definition \ref{defn:positive_ordering}.

Moreover, in Example \ref{example_FN}, we computed that $F(\begin{smallmatrix}
    1&&2\\&2
\end{smallmatrix})=I_1$. From above, we have that $I_1$, corresponding to $(0,3)$, belongs to $\Delta_1$ and $\Delta_2$. Hence, $\mu(\begin{smallmatrix}
    1&&2\\&2
\end{smallmatrix})=\mu(I_1)=\theta_2\theta_1$.
\end{example}
\begin{figure}
  \centering
    \begin{tikzpicture}[scale=1]
\fill[red!50!white] (-0.3,0) -- (2.7,3)--(3.3,3) --(0.3,0)--(-0.3,0);
\fill[red!50!white] (2.7,3) -- (3.3,3) --(4.3,2) --(3.7,2)--(2.7,3);
\fill[red!50!white] (3.7,2) -- (4.3,2) --(5.3,3) --(4.7,3)--(3.7,2);
\fill[red!50!white] (4.7,3) -- (5.3,3) --(8.3,0) --(7.7,0)--(4.7,3);
\fill[green!50!white] (0.7,1) -- (1.3,1) --(2.3,0) --(1.7,0)--(0.7,1);
\fill[green!50!white] (1.7,0) -- (2.3,0) --(3.3,1) --(2.7,1)--(1.7,0);
\fill[green!50!white] (2.7,1) -- (3.3,1) --(4.3,0) --(3.7,0)--(2.7,1);
\fill[green!50!white] (3.7,0) -- (4.3,0) --(6.3,2) --(5.7,2)--(3.7,0);
\fill[blue!50!white] (1.7,2) -- (2.3,2) --(3.3,1) --(2.7,1)--(1.7,2);
\fill[blue!50!white] (2.7,1) -- (3.3,1) --(4.3,2) --(3.7,2)--(2.7,1);
\fill[blue!50!white] (3.7,2) -- (4.3,2) --(6.3,0) --(5.7,0)--(3.7,2);
\fill[blue!50!white] (5.7,0) -- (6.3,0) --(7.3,1) --(6.7,1)--(5.7,0);

\draw (0,0)  node{$\scriptstyle (0,1)$};
\draw (2,0)  node{$\scriptstyle (1,2)$};
\draw (4,0)  node{$\scriptstyle (2,3)$};
\draw (6,0)  node{$\scriptstyle (3,4)$};
\draw (8,0)  node{$\scriptstyle (4,0)$};

\draw (1,1)  node{$\scriptstyle (0,2)$};
\draw (3,1)  node{$\scriptstyle (1,3)$};
\draw (5,1)  node{$\scriptstyle (2,4)$};
\draw (7,1)  node{$\scriptstyle (3,0)$};

\draw (2,2)  node{$\scriptstyle (0,3)$};
\draw (4,2)  node{$\scriptstyle (1,4)$};
\draw (6,2)  node{$\scriptstyle (2,0)$};

\draw (3,3)  node{$\scriptstyle (0,4)$};
\draw (5,3)  node{$\scriptstyle (1,0)$};

\draw [dashed] (1.5,1)--(2.5,1);
\draw [dashed] (3.5,1)--(4.5,1);
\draw [dashed] (5.5,1)--(6.5,1);

\draw [dashed] (2.5,2)--(3.5,2);
\draw [dashed] (4.5,2)--(5.5,2);

\draw [->] (0.2,0.2)--(0.8,0.8);
\draw [->] (1.2,1.2)--(1.8,1.8);
\draw [->] (2.2,2.2)--(2.8,2.8);
\draw [->] (2.2,0.2)--(2.8,0.8);
\draw [->] (3.2,1.2)--(3.8,1.8);
\draw [->] (4.2,2.2)--(4.8,2.8);
\draw [->] (4.2,0.2)--(4.8,0.8);
\draw [->] (5.2,1.2)--(5.8,1.8);
\draw [->] (6.2,0.2)--(6.8,0.8);

\draw [->](1.2,0.8)--(1.8,0.2);
\draw [->](2.2,1.8)--(2.8,1.2);
\draw [->](3.2,2.8)--(3.8,2.2);
\draw [->](3.2,0.8)--(3.8,0.2);
\draw [->](4.2,1.8)--(4.8,1.2);
\draw [->](5.2,0.8)--(5.8,0.2);
\draw [->](5.2,2.8)--(5.8,2.2);
\draw [->](6.2,1.8)--(6.8,1.2);
\draw [->](7.2,0.8)--(7.8,0.2);

\end{tikzpicture}
\caption{The Auslander-Reiten quiver of $\mathcal{C}_\mathcal{F}$ of type $A_2$.}
\label{fig:AR_eg}
\end{figure}
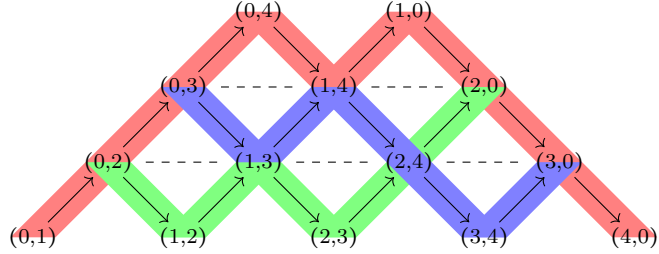


\section{Lattice bijections}\label{Sec:latticebij}
Throughout this section, we assume $\Lambda=KQ/I$, where $Q$ is a finite quiver and $I$ is an admissible ideal in $KQ$, and $M$ is a string module in $\module\Lambda$. We will denote by $\widetilde{M}=M\otimes_K K[\epsilon]/(\epsilon^k)$  the induced module in $\module \widetilde{\Lambda}= \module (\Lambda\otimes_K K[\epsilon]/(\epsilon^k))$, where $k\in\mathbb{N}$.

 Given a snake graph $\mathcal{G}=(G_1,G_2,\dots,G_d)$, recall that its minimal dimer cover is the one containing the West of $G_1$ and only boundary edges and, its minimal double dimer cover, similarly, is the one containing two copies of the West of $G_1$ and only boundary edges. The maximal dimer cover and the maximal double dimer cover are defined complementary to the minimal one. We will denote by $P_\textup{min}, P_\textup{max}, D_\textup{min}$ and $D_\textup{max}$ the minimal and maximal dimer covers (perfect matchings) and the minimal and maximal double dimer covers of $\mathcal{G}$, respectively.

\subsection{Dimer covers versus submodules of string modules}

In \cite{CS21}, a bijection between \emph{abstract snake graphs} and \emph{abstract strings} is introduced. An abstract snake graph is a snake graph with no face or edge weights and an abstract string is either the empty word or a finite word in the alphabet $\{\longrightarrow, \longleftarrow,\bullet \}$ that starts with a vertex, ends with a vertex and such that there is a vertex between any consecutive arrows. For a snake graph $\mathcal{G}=(G_1,G_2,\dots,G_d)$, an abstract string $w_\mathcal{G}$ (or simply $w$) is given by associating a vertex to each tile and direct and inverse arrows between vertices. More explicitly, the arrows in $w_\mathcal{G}$ are (uniquely) determined by the position of the second tile $G_2$, if this exists: if $G_2$ is to the right of $G_1$, we associate a direct arrow $ \bullet \longrightarrow \bullet$, otherwise we associate an inverse arrow $\bullet\longleftarrow \bullet$ and we iteratively derive direct and inverse arrows to adjacent tiles $G_i$ and $G_{i+1}$ depending on the form of the tiles $G_{i-1}, G_i, G_{i+1}$; if it is a zig-zag then the arrow induced by $G_i,G_{i+1}$ is same as the one induced by $G_{i-1},G_i$ and if it is straight, then it is opposite to the one induced by $G_{i-1},G_i$. For an abstract string $w$, an abstract snake graph $\mathcal{G}_w$ is given by associating a tile for each vertex in $w$ and glueing $G_2$ to the East edge of $G_1$ if the first arrow is direct and to the North edge if the first arrow is inverse and iteratively gluing tiles $G_{i+1}$ on the North or the East of $G_{i}$ depending whether the $i^\textup{th}$ arrow in $w$ agrees with the $i-1^\textup{st}$; if they agree we glue $G_{i+1}$ on the North or the East edge such that $G_{i-1},G_i,G_{i+1}$ is a zigzag and when the arrows disagree we consider straight pieces. 

\begin{example}
    For the abstract string $w=\bullet \longrightarrow \bullet \longleftarrow \bullet \longleftarrow \bullet$ the corresponsing astract snake graph is given by $G_w =$ \begin{tikzpicture}[scale=0.5]
  \draw (0,0) \rectanglepath;
  \draw (1,0) \rectanglepath;
    \draw (2,0) \rectanglepath;
      \draw (2,1) \rectanglepath;
  \end{tikzpicture}.
\end{example}

Building on the bijection between abstract snake graphs and abstract strings, in \cite{CS21}, an explicit bijection is introduced between the  dimer cover lattice $\mathcal{L}(\mathcal{G}_w)$ of $\mathcal{G}_w$ and the canonical submodule lattice $\mathcal{L}(M_w)$ of $M_w$, where $M_w$ is any string module whose underlinging string is $w$. In this correspondence, the maximal dimer cover $P_\textup{max}$ corresponds to the representation $M_w$ and the minimal  dimer cover $P_\textup{min}$ to the zero module. The adjacent vertices in the Hasse diagram of the submodule lattice $\mathcal{L}(M_w)$ are obtained by adding or removing a top from a submodule which in the Hasse diagram of the dimer cover lattice $\mathcal{L}(\mathcal{G}_w)$ corresponds to twisting a tile in a dimer cover. 

\subsection{Double dimer covers versus submodules of induced modules}
 
 The main result of this section establishes a lattice bijection between the 
double dimer covers of $\mathcal{G}_w$ and the canonical submodules of $\widetilde M_w \in \textup{mod}\, (\La \otimes_K K[\epsilon]/ \epsilon^2) $ where $\widetilde M_w$ is an induced module of a string module $M_w$ whose underlying string is $w$ and where $\mathcal{G}_w$ is the abstract snake graph associated to the abstract word $w$. Here we will focus on the case where $k=2$  but analogous constructions can be given when we consider $k$-tuple dimer covers (i.e. each vertex is incident with precisely $k$ edges) and induced $\La \otimes_K K[\epsilon]/\epsilon^k $-modules.

From now on, assume  $\widetilde M_w=M_w\otimes_K K[\epsilon]/\epsilon^2$ is an induced module in $\module\widetilde{\Lambda}$ where $M_w$ is a string module in $\module\Lambda$.
Analogous to the classic setting, we may associate a representation $\widetilde M_w$ over $\textup{mod }\widetilde\Lambda$ to the snake graph $\mathcal{G}_w$ such that the maximal double dimer cover $D_\textup{max}$ corresponds to $\widetilde M_w$ and the minimal double dimer cover $D_\textup{min}$ to the zero submodule, see Figure~\ref{fig:sg-indmod}. We will show that removing / adding a top to a submodule $N$ of $\widetilde M_w$ may be associated with twisting a tile in the double dimer cover $\mathcal{DD}(\mathcal{G}_w)$. In Figure~\ref{fig:SG_SM}, compare the Hasse diagram of the double dimer cover lattice of the snake graph $\mathcal{G}_w={\begin{tikzpicture}[scale=.5]
        \draw (0,0) \rectanglepath;
        \draw (1,0) \rectanglepath;
  \draw (0.5,0.5) node{\tiny $1$};
  \draw (1.5,0.5) node{\tiny $2$};
    \end{tikzpicture}}$ and the Hasse diagram of the submodule lattice of the induced module $\widetilde{M}_w=
    \begin{smallmatrix}
     &1&\\1&&2\\&2
    \end{smallmatrix}$.
   
   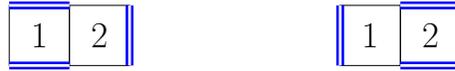
\begin{figure}[h!]\scalebox{0.8}{
    \centering
 \begin{tikzpicture}
 \begin{scope}
   \draw (0,0) \rectanglepath;
  \draw (1,0) \rectanglepath;
  \draw[very thick, double, double distance=1.3pt, blue] (2,1)--(2,0);
  \draw[very thick, double, double distance=1.3pt, blue] (1,1)--(0,1);
  \draw[very thick, double, double distance=1.3pt, blue] (0,0)--(1,0);
 \draw (0.5,0.5) node{\Large $1$};
  \draw (1.5,0.5) node{\Large $2$};
  \end{scope}
 \begin{scope}[xshift=5.5cm]
 \draw (0,0) \rectanglepath;
  \draw (1,0) \rectanglepath;
 \draw[very thick, double, double distance=1.3pt, blue] (0,0)--(0,1);
  \draw[very thick, double, double distance=1.3pt, blue] (1,0)--(2,0);
  \draw[very thick, double, double distance=1.3pt, blue] (1,1)--(2,1);
 \draw (0.5,0.5) node{\Large $1$};
  \draw (1.5,0.5) node{\Large $2$};
\end{scope}
\end{tikzpicture}
}
    \caption{The maximal (left) and minimal (right) double dimer covers and the corresponding indecomposable modules are $\widetilde{M} = \begin{smallmatrix}
    &1&\\1&&2\\&2
\end{smallmatrix}$ and the zero module in $\module\widetilde\Lambda$, respectively.}
     \label{fig:sg-indmod}
\end{figure}

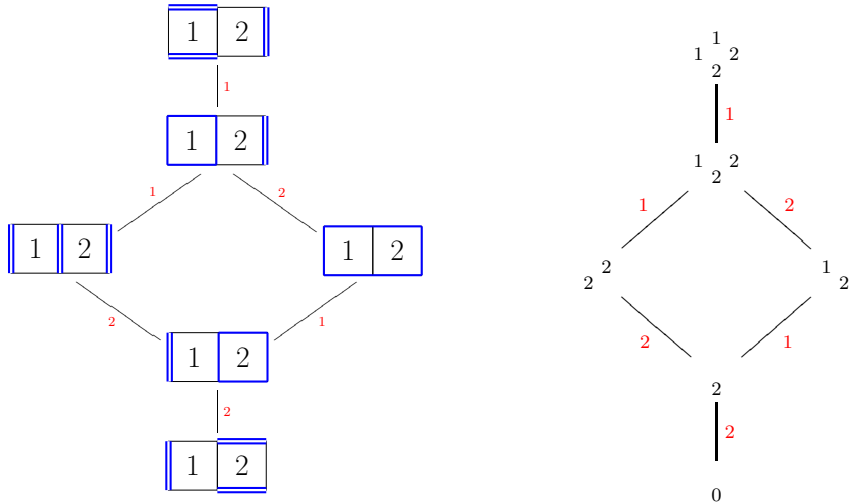
\begin{figure}[h!]
    \centering
\scalebox{0.65}{\xymatrix{
&{\begin{tikzpicture}
  \draw (0,0) \rectanglepath;
  \draw (1,0) \rectanglepath;
  \draw[very thick, double, double distance=1.3pt, blue] (2,1)--(2,0);
  \draw[very thick, double, double distance=1.3pt, blue] (1,1)--(0,1);
  \draw[very thick, double, double distance=1.3pt, blue] (0,0)--(1,0);
 \draw (0.5,0.5) node{\Large $1$};
  \draw (1.5,0.5) node{\Large $2$};
\end{tikzpicture}}\ar@[red]@{-}[d]^-{\color{red}1}
\\
&{\begin{tikzpicture}
  \draw (0,0) \rectanglepath;
  \draw (1,0) \rectanglepath;
  \draw[very thick, double, double distance=1.3pt, blue] (2,1)--(2,0);
  \draw[very thick, blue]  (0,1)--(1,1);
 \draw[very thick, blue]   (0,0)--(1,0);
  \draw[very thick, blue]  (1,0)--(1,1);
  \draw[very thick, blue]  (0,0)--(0,1);
 \draw (0.5,0.5) node{\Large $1$};
  \draw (1.5,0.5) node{\Large $2$};
\end{tikzpicture}}
\ar@[red]@{-}[ld]_-{\color{red}1}\ar@[red]@{-}[rd]^-{\color{red}2}
\\
{\begin{tikzpicture}
  \draw (0,0) \rectanglepath;
  \draw (1,0) \rectanglepath;
  \draw[very thick, double, double distance=1.3pt, blue] (2,1)--(2,0);
  \draw[very thick, double, double distance=1.3pt, blue] (0,0)--(0,1);
  \draw[very thick, double, double distance=1.3pt, blue] (1,0)--(1,1);
 \draw (0.5,0.5) node{\Large $1$};
  \draw (1.5,0.5) node{\Large $2$};
\end{tikzpicture}}
\ar@[red]@{-}[rd]_-{\color{red}2}
&&
{\begin{tikzpicture}
  \draw (0,0) \rectanglepath;
  \draw (1,0) \rectanglepath;
  \draw[very thick, blue]   (0,1)--(1,1);
 \draw[very thick, blue]  (0,0)--(1,0);
  \draw[very thick, blue]  (0,0)--(0,1);
    \draw[very thick, blue]  (1,0)--(2,0);
  \draw[very thick, blue]  (2,0)--(2,1);
    \draw[very thick, blue]  (2,1)--(1,1); 
 \draw (0.5,0.5) node{\Large $1$};
  \draw (1.5,0.5) node{\Large $2$};
\end{tikzpicture}}
\ar@[red]@{-}[ld]^-{\color{red}1}
\\
&{\begin{tikzpicture}
  \draw (0,0) \rectanglepath;
  \draw (1,0) \rectanglepath;
  \draw[very thick, double, double distance=1.3pt, blue] (0,0)--(0,1);
   \draw[very thick, blue]  (1,0)--(1,1);
    \draw[very thick, blue]  (1,0)--(2,0);
  \draw[very thick, blue]  (2,0)--(2,1);
    \draw[very thick, blue] (2,1)--(1,1);
 \draw (0.5,0.5) node{\Large $1$};
  \draw (1.5,0.5) node{\Large $2$};
\end{tikzpicture}}
\ar@[red]@{-}[d]^-{\color{red}2}
\\
&{\begin{tikzpicture}
  \draw (0,0) \rectanglepath;
  \draw (1,0) \rectanglepath;
 \draw[very thick, double, double distance=1.3pt, blue] (0,0)--(0,1);
  \draw[very thick, double, double distance=1.3pt, blue] (1,0)--(2,0);
  \draw[very thick, double, double distance=1.3pt, blue] (1,1)--(2,1);
 \draw (0.5,0.5) node{\Large $1$};
  \draw (1.5,0.5) node{\Large $2$};
\end{tikzpicture}}
}
}
\qquad \qquad 
\scalebox{0.9}{\xymatrix{
&{\begin{smallmatrix}
    &1&\\1&&2\\&2
\end{smallmatrix}}\ar@[red]@{-}[d]^-{\color{red}1}
\\
&{\begin{smallmatrix}
    &&\\1&&2\\&2
\end{smallmatrix}}
\ar@[red]@{-}[ld]_-{\color{red}1}\ar@[red]@{-}[rd]^-{\color{red}2}
\\
{\begin{smallmatrix}
    &&\\&&2\\&2
\end{smallmatrix}}
\ar@[red]@{-}[rd]_-{\color{red}2}
&&
{\begin{smallmatrix}
    &&\\1&&\\&2
\end{smallmatrix}}
\ar@[red]@{-}[ld]^-{\color{red}1}
\\
&{\begin{smallmatrix}
    &&\\&&\\&2
\end{smallmatrix}}
\ar@[red]@{-}[d]^-{\color{red}2}
\\
&{\begin{smallmatrix}
    &&\\&&\\&0
\end{smallmatrix}}
}
}
\caption{The Hasse diagrams for the lattice of double dimer covers of $\mathcal{G}_w$ and the submodule lattice of $\widetilde{M}_w$ corresponding to the word $w=1\to 2$.}\label{fig:SG_SM}
\end{figure}

\begin{definition}
\begin{enumerate}
    \item We define the \emph{symmetric difference} $D\ominus D'$  of  double dimer covers $D$ and $D'$ of a snake graph $\mathcal{G}$ to be the  multigraph consisting of the set of edges given by $(D\cup D') \backslash (D \cap D')$ where for any edge $e$ in $\mathcal{G}$, the union is given by the maximum multiplicities of $e$ in $D$ and $D'$ and the intersection as the minumum multiplicities of $e$.
     \item A graph is called a \emph{snake multigraph} if the graph obtained by replacing multiple edges with single edges is a snake graph. 
     \item The \emph{completion} of a sub-multigraph $\mathcal{G}'$ of a snake multigraph, denoted by $(\mathcal{G}')^c$, is the multigraph obtained by adding a single copy of all missing edges of a tile if at least one of its boundary edges is in $\mathcal{G}'$. 
     \item A \emph{snake sub-multigraph} is a sub-multigraph of a snake multigraph such that each of its connected components is itself a snake multigraph. 
\end{enumerate}
\end{definition}

 \begin{lemma}\label{lem:enclosedtiles} Given a snake graph $\mathcal{G}$ and a double dimer cover $D$, the symmetric difference  
    $D\ominus D_\textup{min}$ gives rise to enclosed tiles of $\mathcal{G}$ where $D\ominus D_\textup{min}$ may contain either double or single edges. Moreover, the completion $(D\ominus D_\textup{min})^c$ is a snake sub-multigraph of doubled $\mathcal{G}$, that is a graph obtained by superimposing two copies of $\mathcal{G}$.
 \end{lemma}
 
\begin{proof}
    Consider first the case when $\mathcal{G}$ consists of a single tile. Then $D$ is either the minimal double dimer cover, the maximal one, or it consists of one copy of each of the four edges. In the first case, $D\ominus D_\textup{min}$ is the empty set, in the second it consists of double copies of all edges and in the third case of single copies of all edges. That is, in all cases $D\ominus D_\textup{min}$ gives rise to enclosed tiles of $\mathcal{G}$.

    Suppose now that for any snake graph with at most $k-1$ tiles, $D\ominus D_\textup{min}$ gives rise to enclosed tiles of $\mathcal{G}$ containing either single or double edges. Consider a snake graph $\mathcal{G}$ with $k$ tiles, labelled in order $G_1,G_2,\dots,G_{k-1},G_k$ and let $\mathcal{G}'$ be the snake graph consisting of tiles $G_1,G_2,\dots,G_{k-1}$, that is obtained by removing the last tile from $\mathcal{G}$. Given a double dimer cover $D$ of $\mathcal{G}$, we define a double dimer cover $D'$ of $\mathcal{G}'$ that agrees with $D$ apart from possibly at the edge $e$ between $G_{k-1}$ and $G_k$ in $\mathcal{G}$ in the following way: 
    \begin{enumerate}
        \item if $D$ contains a single copy of the three edges of $G_k$ different from $e$ (and possibly also a single copy of $e$), then  $D'$ contains a single copy of $e$,
        \item If $D$ contains a single copy of $e$ and a double copy of the edge in $G_k$ parallel to $e$, then  $D'$ contains a single copy of $e$,
        \item if $D$ contains double copies of the edges in $G_k$ incident to $e$ or a double copy of $e$, then  $D'$ contains a double copy of $e$,
        \item if $D$ contains a double copy of the edge in $G_k$ parallel to $e$ and no other edge of $G_k$, then $D'$ contains no copy of $e$. 
    \end{enumerate}
    Letting $D'_{\textup{min}}$ be the minimal double dimer cover of $\mathcal{G}'$, by inductive hypothesis we know that $D'\ominus D'_\textup{min}$ gives rise to enclosed tiles of $\mathcal{G}'$. Moreover, note that $D'_{\textup{min}}$ agrees with the restriction of $D_{\textup{min}}$ to $\mathcal{G}'$ apart from possibly at edge $e$. In case (1), the single edges of $G_k$ contained in $D$ are also contained in $D\ominus D_\textup{min}$. Then if $D$ contains a single copy of $e$, the symmetric difference encloses tile $G_k$ on its own, otherwise it encloses tile $G_k$ together with the last enclosed tiles in $D'\ominus D'_\textup{min}$. In the remaining cases, $D$ has some double copies of edges in $G_k$. If the double copies of the boundary edges agree with $D_{\textup{min}}$, then it is easy to see $D\ominus D_\textup{min}$ gives rise to the same enclosed tiles as $D'\ominus D'_\textup{min}$. On the other hand, if they do not agree with $D_{\textup{min}}$, then $D\ominus D_\textup{min}$ contains double copies of the three boundary edges of $G_k$ and it is again easy to check that then either $D\ominus D_\textup{min}$ encloses tile $G_k$ on its own or together with the last enclosed tiles in $D'\ominus D'_\textup{min}$.

    This proves the claim. Then, each maximal set of enclosed tiles gives a connected component in $(D\ominus D_\textup{min})$. When completed, each of these components is a snake multigraph and so  $(D\ominus D_\textup{min})^c$ is a snake sub-multigraph of doubled $\mathcal{G}$.
\end{proof} 

We denote each connected component of $(D\ominus D_\textup{min})^c$ in Lemma~\ref{lem:enclosedtiles} by $\mathcal{H}_i$, note this is a snake multigraph,  and write $\bigcup \mathcal{H}_i$ for $(D\ominus D_\textup{min})^c$.

\subsubsection{Abstract loopy strings} 

Let $\mathcal{B}=  \{ \longrightarrow, \longleftarrow, \Loop, \bullet \}$ be a set of four letters where we refer to the letters as a direct arrow, an inverse arrow, a loop and a vertex, respectively. An abstract \emph{loopy string} is either the empty word denoted by $\emptyset$ or it is a finite word in the alphabet $\mathcal{B}$ which starts with a vertex, ends with a vertex and such that a loop is followed by a vertex (and considered as starting and ending at the same vertex) and such that there is a vertex between any consecutive arrows and there is at most one loop at each vertex.

For simplicity, we will often omit the vertices and simply write, for instance,  $\longrightarrow \ou{\Loop}{} \longrightarrow$ for the loopy string 
$\bullet\longrightarrow\ou{\Loop}{\bullet}\longrightarrow\bullet$.

\subsubsection{Snake multigraphs versus abstract loopy strings}

Let $w$ be an abstract loopy string. We construct a snake multigraph associated to $w$ as follows. If $w=\emptyset$, the corresponding snake multigraph is the empty graph. If $w\neq\emptyset$, we consider its underlying abstract string $\overline{w}$ obtained by removing all loops and  the abstract snake graph $\mathcal{G}_{\overline{w}}$ associated to $\overline{w}$.  The snake multigraph $\mathcal{G}_w$ associated to $w$ is obtained by considering the boundary edges with multiplicity two for all the tiles corresponding to the loops in $w$.

\begin{example}
    For the loopy string $\longrightarrow \ou{\Loop}{} \longrightarrow$, the corresponding abstract snake multigraph is given by $
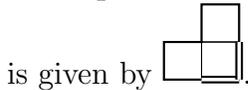
\begin{tikzpicture}[scale=.5]
        \draw (0,0) \rectanglepath;
        \draw (1,0) \rectanglepath;
        \draw (1,1) \rectanglepath;
  \draw[thick] (1,0)--(0,0)--(0,1)--(1,1)--(1,2)--(2,2)--(2,1);
  \draw[thick, double, double distance=1.3pt] (1,0)--(2,0);
  \draw[thick, double, double distance=1.3pt] (2,0)--(2,1);
    \end{tikzpicture}$.
\end{example}
    
We will consider snake multigraphs such that for each tile either all boundary edges are single or double and refer to them as \emph{good snake multigraphs}. For a good snake multigraph $\mathcal{G}$, we will associate an abstract loopy string by considering the abstract string corresponding to the induced snake graph of $\mathcal{G}$ obtained by replacing all double edges with a single edge and adding loops at the vertices corresponding to the tiles in $\mathcal{G}$ for which the boundary edges have multiplicity two.

\begin{theorem}
    With the notation above, there is a bijection between the set of abstract loopy  strings and the set of abstract good snake multigraphs.
\end{theorem}
\begin{proof}
    This is straightforward since finite loopy strings give rise to good snake multigraphs by construction. 
\end{proof}

We may also generalise this theorem to a labeled version where snake multigraphs are considered with face weights which induces weights on the vertices of the corresponding loopy string and vice versa. For instance, we may consider the labeled loopy string $1\longrightarrow\ou{\Loop}{2}\longrightarrow 3$ in correspondence with the labeled snake multigraph $
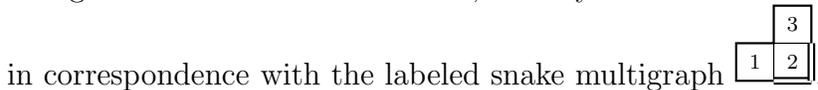
\begin{tikzpicture}[scale=.5]
        \draw (0,0) \rectanglepath;
        \draw (1,0) \rectanglepath;
        \draw (1,1) \rectanglepath;
  \draw (0.5,0.5) node{\tiny $1$};
  \draw (1.5,0.5) node{\tiny $2$};
  \draw (1.5,1.5) node{\tiny $3$};
  \draw[thick] (1,0)--(0,0)--(0,1)--(1,1)--(1,2)--(2,2)--(2,1);
  \draw[thick, double, double distance=1.3pt] (1,0)--(2,0);
  \draw[thick, double, double distance=1.3pt] (2,0)--(2,1);
    \end{tikzpicture}$.

\begin{definition} 
    We call a good snake sub-multigraph $\mathcal{G}'$ of $\mathcal{G}$ \emph{optimal} if whenever two copies of the boundary edges of a tile $G_k$ appears in $\mathcal{G}'$ then 
    $\mathcal{G}'$ contains two copies of the boundary edges of all the tiles succeeding $G_k$ in a maximal zig-zag that induces a direct string and $\mathcal{G}'$ contains two copies of the boundary edges of all the tiles preceding $G_k$ in a maximal zig-zag that induces an inverse string.
\end{definition}

The previous definition can be rephrased as follows. If $G_1,\dots,G_k,G_{k+1}, \dots, G_j$ is a maximal zig-zag snake sub-multigraph that induces a direct abstract string, then the tiles $G_k,G_{k+1}, \dots, G_j$ all have two copies of the boundary edges in $\mathcal{G}'$ and if the tiles $G_1,\dots,G_k,$ $G_{k+1}, \dots, G_j$ form a maximal zig-zag sub-snake multigraph that induces an inverse abstract string, then the tiles $G_1,, \dots, G_{k-1}, G_k$ all have two copies of the boundary edges in $\mathcal{G}'$.

The following remark will be useful in the proof of Lemma \ref{loop-mazimal-zigzag}.
\begin{remark}\label{remark_headofarrow}
 Note that our convention for minimal dimer cover implies that the upper part (target) of each arrow in the corresponding abstract string is parallel to the edges belonging to the minimal dimer cover. For example

    
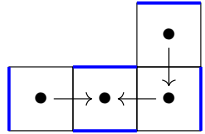
\begin{center}
    \begin{tikzpicture}[scale=0.85]
  \draw (0,0) \rectanglepath;
  \draw (0.5,0.5)  node{$\bullet$};
  \draw (1.5,0.5)  node{$\bullet$};
  \draw (2.5,0.5)  node{$\bullet$};
  \draw (2.5,1.5)  node{$\bullet$};
  \draw (1,0) \rectanglepath;
    \draw (2,0) \rectanglepath;
      \draw (2,1) \rectanglepath;
      \draw [->](0.7,0.5)--(1.3,0.5);
      \draw [->](2.3,0.5)--(1.7,0.5);
      \draw [->](2.5,1.3)--(2.5,0.7);
 \draw[very thick, blue]  (0,0)--(0,1);
 \draw[very thick, blue]  (1,0)--(2,0);
   \draw[very thick, blue]  (1,1)--(2,1);
     \draw[very thick, blue]  (3,0)--(3,1);
       \draw[very thick, blue]  (2,2)--(3,2);
  \end{tikzpicture}
  \end{center}
  Moreover, the same convention and rule applies for double dimer covers.
\end{remark}

\begin{lemma}\label{loop-mazimal-zigzag}
    Let $\mathcal{G}$ be a snake graph and $D$ be a double dimer cover of $\mathcal{G}$.  
The completion $(D\ominus D_\textup{min})^c$ is an optimal  snake sub-multigraph.
\end{lemma}

\begin{proof}
    We first prove that $(D\ominus D_\textup{min})^c$ is a good snake sub-multigraph. It is enough to consider one connected component $\mathcal{H}_i$ of $(D\ominus D_\textup{min})^c$. If $\mathcal{H}_i$ consists only of one tile, then $(D\ominus D_\textup{min})^c$ contains all of its edges either all single or all double, hence $\mathcal{H}_i$ is a good snake multigraph.
    Suppose now that $\mathcal{H}_i$ is a good snake multigraph whenever it has at most $k$ tiles and assume that $\mathcal{H}_i$ has $k+1$ tiles. Consider the $(k+1)$th tile $G$ of $\mathcal{H}_i$. 

Single edges covering vertices of $G$ in $D$ have to be in the tile $G$, otherwise we would get a contradiction to $\mathcal{H}_i$ ending at $G$. Hence if $D$ contains one single edge of $G$, then it has to contain the three single edges not contained in the previous tile of $\mathcal{H}_i$. Then, in the symmetric difference with $D_\textup{min}$, these three edges are single. On the other hand, if one (or two) of these three edges is double in $D$, then $D_\textup{min}$ has to contain a double copy of the other two (or other one) because $\mathcal{H}_i$ encloses $G$. By induction $\mathcal{H}_i$ is a good snake multigraph.

We now prove that $(D\ominus D_\textup{min})^c$ is an optimal  snake sub-multigraph. Again, it is enough to consider one connected component $\mathcal{H}_i=\{G_1,\dots, G_m\}$. Suppose there is a zig-zag $Z$ in $\mathcal{H}_i$ that induces a direct string and that there is a tile $G_k$ in $Z$ such that $\mathcal{H}_i$ contains two copies of the boundary edges of $G_k$. Assume for a contradiction that there exists a subsequent tile in $Z$ such that $\mathcal{H}_i$ contains only one copy of its boundary edges and that $G_{k+j}$ is the first such tile after $G_k$. Assume $G_{k+j}$ is North of the previous tile (the case when it is East works analogously). Then $\mathcal{H}_i$ contains at least the blue edges in the following drawing. Moreover, by Remark \ref{remark_headofarrow}, $D_\textup{min}$ contains the red double edges indicated in the drawing and hence $D$ must contain the green edges.

\begin{center}
$\begin{tikzpicture}[scale=.85]
\begin{scope}
        \draw (1,0) \rectanglepath;
        \draw (1,1) \rectanglepath;
  \draw (1.5,1.5) node{\tiny $k+j$};
  \draw (1,0)--(2,0);
  \draw (2,0)--(2,1);
  \draw[very thick, double, double distance=1.3pt, blue] (1,0)--(2,0);
  \draw[very thick, double, double distance=1.3pt, blue] (2,0)--(2,1);
  \draw[very thick, blue] (1,1)--(1,2);
  \end{scope}
  \begin{scope}[xshift=2.8cm, yshift=1cm]
      \draw (0,0) node {=};
   \end{scope}
  \begin{scope}[xshift=2.5cm]
        \draw (1,0) \rectanglepath;
        \draw (1,1) \rectanglepath;
  \draw (1.5,1.5) node{\tiny $k+j$};
  \draw (1,0)--(2,0);
  \draw (2,0)--(2,1);
  \draw[very thick, double, double distance=1.3pt, green] (2,0)--(2,1);
  \draw[very thick, green] (1,1)--(1,2);
  \end{scope}
  \begin{scope}[xshift=5.3cm, yshift=1cm]
      \draw (0,0) node {\Large $\ominus$};
   \end{scope}
    \begin{scope}[xshift=2.8cm, yshift=1cm]
      \draw (0,0) node {=};
   \end{scope}
  \begin{scope}[xshift=5.3cm]
        \draw (1,0) \rectanglepath;
        \draw (1,1) \rectanglepath;
  \draw (1.5,1.5) node{\tiny $k+j$};
  \draw (1,0)--(2,0);
  \draw (2,0)--(2,1);
  \draw [->](0.7,0.5)--(1.3,0.5);
  \draw [->](1.5,0.7)--(1.5,1.3);
  \draw[very thick, double, double distance=1.3pt, red] (1,0)--(2,0);
  \draw[very thick,double, double distance=1.3pt, red] (1,1)--(1,2);
  \end{scope}
    \end{tikzpicture}$
    \end{center}
    
 Note that since $D$ contains a single edge, this has to be part of a cycle. However there is no way to draw such a cycle without passing through the South-East vertex of $G_{k+j}$, which is already covered twice by $D$. Hence we have a contradiction.
Analogously, one can prove the corresponding statement when the zig-zag is induced by an inverse string.
\end{proof}

\begin{definition}\label{defn:face-function}
Let $\mathcal{G}$ be a snake multigraph with face weights $\{v_1,v_2,\dots, v_n\}\subset \mathbb{N}$. Define a function $h: \mathcal{G} \to \mathbb{N}^k$ given by $h(\mathcal{G})=(n_1,\dots, n_k)$ where $n_i=s_i+d_i$ is such that $s_i$ equals the number of tiles whose boundary edges are all single edges with face weight $v_i$ and $d_i$ equals twice the number of tiles with at least one double boundary edge with face weight $v_i$ in $\mathcal{G}$. We call $h$ the \emph{face function} of $\mathcal{G}$.
\end{definition}

\begin{example}
Consider the following snake multigraphs with face weights:

    \begin{center}
    \begin{tikzpicture}[scale=0.85]
    \begin{scope}
  \draw (-0.8,0.5) node{$\mathcal{G}_1:$};
  \draw (0,0) \rectanglepath;
  \draw (1,0) \rectanglepath;
  \draw (2,0) \rectanglepath;
  \draw[very thick, double, double distance=1.3pt] (0,0)--(1,0);
  \draw[very thick, double, double distance=1.3pt] (1,1)--(0,1);
  \draw[very thick, double, double distance=1.3pt] (0,0)--(0,1);
  \draw (0.5,0.5) node{\tiny $1$};
  \draw (1.5,0.5) node{\tiny $2$};
  \draw (2.5,0.5) node{\tiny $1$};
  \end{scope}
  \begin{scope}[xshift=5.5cm]
  \draw (-0.8,0.5) node{$\mathcal{G}_2:$};
  \draw (0,0) \rectanglepath;
  \draw (1,0) \rectanglepath;
  \draw (2,0) \rectanglepath;
  \draw[very thick, double, double distance=1.3pt] (0,0)--(1,0);
  \draw[very thick, double, double distance=1.3pt] (1,1)--(0,1);
  \draw[very thick, double, double distance=1.3pt] (0,0)--(0,1);
  \draw (0.5,0.5) node{\tiny $1$};
  \draw (1.5,0.5) node{\tiny $2$};
  \draw (2.5,0.5) node{\tiny $3$};
  \end{scope}
  \end{tikzpicture}
  \end{center}
  Then $h(\mathcal{G}_1)=(3,1)$ and $h(\mathcal{G}_2)=(2,1,1)$.
\end{example}

We will define a representation $N_w=N_{\mathcal{G}_w}$ for an \emph{optimal} loopy string $w$; i.e. a loopy string which gives rise to an optimal snake multigraph. The representation $N_w$ associated to an optimal snake multigraph $\mathcal{G}_w$ is obtained as follows. Each vertex followed by a loop is replaced by two copies of the field, other vertices with a single copy of the field and the action of an arrow on $N_w$ is the identity morphism if both the tail and head contain a single or double copy of the field, given by $\begin{bsmallmatrix} 0 \\ 1 \end{bsmallmatrix}$ if the arrow is from $K$ to $K^2$, by $\begin{bsmallmatrix} 1 & 0 \end{bsmallmatrix}$ if the arrow is from $K^2$ to $K$, and for the loops, by $\begin{bsmallmatrix} 0 & 0 \\  1 &  0\end{bsmallmatrix}$ when the vector space is $K^2$ and $0$ otherwise. 
If there are repeated labels, for each vertex with the same label take the direct sum of the vector spaces and the maps.  Note that by construction $h(\mathcal{G}_w)=\underline{\dim} (N_w)$.

\begin{example} The loopy string $ w=\ou{\Loop}1\longrightarrow\ou{\Loop}{2}\longleftarrow 1$ 
 gives rise to the optimal snake multigraph 
$ \begin{tikzpicture}[scale=0.85]
      \draw (-0.8,0.5) node{$\mathcal{G}_w:$};
  \draw (0,0) \rectanglepath;
  \draw (1,0) \rectanglepath;
  \draw (2,0) \rectanglepath;
  \draw[very thick, double, double distance=1.3pt] (0,0)--(2,0);
  \draw[very thick, double, double distance=1.3pt] (2,1)--(0,1);
  \draw[very thick, double, double distance=1.3pt] (0,0)--(0,1);
  \draw (0.5,0.5) node{\tiny $1$};
  \draw (1.5,0.5) node{\tiny $2$};
  \draw (2.5,0.5) node{\tiny $1$};
  \end{tikzpicture}$. 
Then, 
\begin{align*}
N_w=  \begin{tikzcd}[ampersand replacement=\&, arrow style=tikz,>=stealth,row sep=4em,column sep=4em]
K^3\arrow[r,shift right=0.6ex, "{\begin{bsmallmatrix}0&0&0\\0&0&1
\end{bsmallmatrix}}"']
\arrow[r, shift left=0.6ex, "{\begin{bsmallmatrix}1&0&0\\0&1&0
\end{bsmallmatrix}}"] \arrow[out=60,in=120,distance=1.5em,loop, "{\begin{bsmallmatrix}0&0&0\\1&0&0\\0&0&0
\end{bsmallmatrix}}"'] \&
\arrow[out=60,in=120,distance=1.5em,loop, "{\begin{bsmallmatrix}0&0\\1&0
\end{bsmallmatrix}}"'] K^2
\end{tikzcd}
\end{align*}
a (indecomposable) representation for  $\widetilde{\Lambda} = K (\ou{\Loop}{1} \rightrightarrows \ou{\Loop}{2} ) / I $.  
\end{example} 

\begin{remark}
    Note that it is possible to associate representations for good snake multigraphs as well but they are not necessarily  sub-representations of an indecomposable induced module. For instance, if we consider $\begin{tikzpicture}[baseline={(0, 0)},scale=0.5]
  \draw (-0.8,0.5) node{$\mathcal{G}:$};
  \draw (0,0) \rectanglepath;
  \draw (1,0) \rectanglepath;
  \draw (2,0) \rectanglepath;
  \draw[very thick, double, double distance=1.3pt] (0,0)--(1,0);
  \draw[very thick, double, double distance=1.3pt] (1,1)--(0,1);
  \draw[very thick, double, double distance=1.3pt] (0,0)--(0,1);
  \draw (0.5,0.5) node{\tiny $1$};
  \draw (1.5,0.5) node{\tiny $2$};
  \draw (2.5,0.5) node{\tiny $1$};
  \end{tikzpicture}$, 
  then the corresponding representation would have dimension vector $(3,1)$ and there is no indecomposable induced module in $\widetilde\Lambda = K (\ou{\Loop}1 \rightrightarrows \ou{\Loop} 2 )/I$ which has a sub-representation of this dimension vector.
  However, $\begin{smallmatrix}
      1&&1&\\
      &2&&1
  \end{smallmatrix}$ is a representation which is not a sub-representation of an induced module in $\module\Lambda$.
\end{remark}

 \begin{proposition}\label{prop:dimer--submod}
     With the notation above, let $D$ be a  double dimer cover
     of a snake graph $\mathcal{G}$.  Then the optimal  snake multigraph $(D \ominus D_{\textup{min}})^c$ gives rise to an embedding $N_D \hookrightarrow \widetilde M_\mathcal{G}$ such that $h((D\ominus D_\textup{min})^c)$ is the dimension vector of the submodule $N_D$ of $\widetilde M_\mathcal{G}$.
 \end{proposition}

 \begin{proof}
     Let $(D \ominus D_{\textup{min}})^c =\bigcup \mathcal{H}_i$. Note that each $\mathcal{H}_i$ is an optimal snake multigraph by Lemma~\ref{loop-mazimal-zigzag}. 
      Let $\varphi_i: {\mathcal{H}}_i \hookrightarrow \mathcal{G}$ be the canonical embedding of snake multigraphs for each $i$ and denote by $w_i$ the induced
      optimal loopy string for each $\mathcal{H}_i$. Then $\varphi_i: N_{{w}_i} \hookrightarrow M_{\mathcal{G}}$ is a canonical embedding such that $\underline{\dim} (N_{{w}_i}) = h({\mathcal{H}}_i)$ by the construction of $N_w$ and
      the definition of face function in~\ref{defn:face-function}. Note also that $N_{w_i}$ is a submodule of $M_\mathcal{G}$ as $N_{w_i}$ is obtained from $M_\mathcal{G}$ by removing a sequence of tops. \end{proof}

\begin{proposition}\label{lemma_decomposition} 
   We may decompose $D\ominus D_\textup{min}$ as a union of enclosed graphs in the following way. Decompose each (enclosed) connected component $C_i$ of $D\ominus D_\textup{min}$ into an enclosed graph $C_{i,1}$ consisting of one copy of each boundary edge of $C_i$, the remaining (single) edges then form enclosed graphs $C_{i,2},\dots C_{i,m}$. The completions of $C_{i,1},\dots, C_{i,m}$ are snake graphs. 
\end{proposition}

\begin{proof}
    Consider a connected component $C_i$ of $D\ominus D_\textup{min}$. If $C_i$ consists only of double edges, then in this portion of $\mathcal{G}$, we have that $D$ and $D_\textup{min}$ consist only of complementary double edges covering the boundary edges of $C_i$. Then $C_i$ can be decomposed into two copies of $C_{i,1}$, constructed as described in the statement.
    Similarly, if there is a portion of $C_i$ containing boundary (in the sense of the boundary of $C_i$) double edges, there cannot be internal edges in this portion, apart from possibly at the end.
    Suppose now there is at least one single edge in $C_i$. Since $D_\textup{min}$ only consists of double edges, the single edges in $C_i$ correspond to single edges in $D$. By construction of double dimer covers, all single edges in $D$ appear in some cycle and moreover there are no edges in $D$ adjacent and not contained in such a cycle. Hence if there are edges adjacent to single edges in $C_i$, these need to be double edges belonging to $D_\textup{min}$.
    Putting all of the above together, selecting one copy of each boundary edge of $C_i$ to form $C_{i,1}$ we are left with either an empty set (if $C_i$ only consisted of single edges) or some enclosed graphs obtained by putting together the last tile of a single cycle and the remaining copies of double edges.
\end{proof}

We illustrate Proposition~\ref{lemma_decomposition} in an example.

\begin{example}
Consider the following snake graph, where $D_\textup{min}$ is indicated in red, a choice of a double dimer cover $D$ in blue and we have computed $D\ominus D_{\textup{min}}$.
    \begin{figure}[h!]\scalebox{0.8}{
    \centering
 \begin{tikzpicture}
 \begin{scope}
  \draw (0,0) \rectanglepath;
  \draw (1,0) \rectanglepath;
  \draw (2,0) \rectanglepath;
  \draw (3,0) \rectanglepath;
  \draw (4,0) \rectanglepath;
  \draw[very thick, blue]  (0,1)--(1,1);
 \draw[very thick, blue]   (0,0)--(1,0);
  \draw[very thick, blue]  (1,0)--(1,1);
  \draw[very thick, blue]  (0,0)--(0,1);
  \draw[very thick, blue]  (2,1)--(3,1);
 \draw[very thick, blue]   (2,0)--(3,0);
  \draw[very thick, blue]  (3,0)--(3,1);
  \draw[very thick, blue]  (2,0)--(2,1);
   \draw[very thick, blue]  (4,1)--(5,1);
 \draw[very thick, blue]   (4,0)--(5,0);
  \draw[very thick, blue]  (5,0)--(5,1);
  \draw[very thick, blue]  (4,0)--(4,1);
 \draw (0.5,0.5) node{\Large $1$};
  \draw (1.5,0.5) node{\Large $2$};
  \draw (2.5,0.5) node{\Large $3$};
  \draw (3.5,0.5) node{\Large $4$};
  \draw (4.5,0.5) node{\Large $5$};
  \end{scope}
  \begin{scope}[xshift=5.5cm, yshift=0.5cm]
      \draw (0,0) node {\Large $\ominus$};
   \end{scope}
\begin{scope}[xshift=6cm]
  \draw (0,0) \rectanglepath;
  \draw (1,0) \rectanglepath;
  \draw (2,0) \rectanglepath;
  \draw (3,0) \rectanglepath;
  \draw (4,0) \rectanglepath;
 \draw[very thick, double, double distance=1.3pt, red] (0,0)--(0,1);
 \draw[very thick, double, double distance=1.3pt, red] (1,1)--(2,1);
 \draw[very thick, double, double distance=1.3pt, red] (1,0)--(2,0);
 \draw[very thick, double, double distance=1.3pt, red] (3,1)--(4,1);
 \draw[very thick, double, double distance=1.3pt, red] (3,0)--(4,0);
 \draw[very thick, double, double distance=1.3pt, red] (5,0)--(5,1);
 \draw (0.5,0.5) node{\Large $1$};
  \draw (1.5,0.5) node{\Large $2$};
  \draw (2.5,0.5) node{\Large $3$};
  \draw (3.5,0.5) node{\Large $4$};
  \draw (4.5,0.5) node{\Large $5$};
  \end{scope}
  \begin{scope}[xshift=11.5cm, yshift=0.5cm]
      \draw (0,0) node {\Large $=$};
   \end{scope}
\begin{scope}[xshift=12cm]
  \draw (0,0) \rectanglepath;
  \draw (1,0) \rectanglepath;
  \draw (2,0) \rectanglepath;
  \draw (3,0) \rectanglepath;
  \draw (4,0) \rectanglepath;
 \draw[very thick, double, double distance=1.3pt, red] (1,1)--(2,1);
 \draw[very thick, double, double distance=1.3pt, red] (1,0)--(2,0);
 \draw[very thick, double, double distance=1.3pt, red] (3,1)--(4,1);
 \draw[very thick, double, double distance=1.3pt, red] (3,0)--(4,0);
  \draw[very thick, blue]  (0,1)--(1,1);
 \draw[very thick, blue]   (0,0)--(1,0);
  \draw[very thick, blue]  (1,0)--(1,1);
  \draw[very thick, red]  (0,0)--(0,1);
  \draw[very thick, blue]  (2,1)--(3,1);
 \draw[very thick, blue]   (2,0)--(3,0);
  \draw[very thick, blue]  (3,0)--(3,1);
  \draw[very thick, blue]  (2,0)--(2,1);
   \draw[very thick, blue]  (4,1)--(5,1);
 \draw[very thick, blue]   (4,0)--(5,0);
  \draw[very thick, red]  (5,0)--(5,1);
  \draw[very thick, blue]  (4,0)--(4,1);
 \draw (0.5,0.5) node{\Large $1$};
  \draw (1.5,0.5) node{\Large $2$};
  \draw (2.5,0.5) node{\Large $3$};
  \draw (3.5,0.5) node{\Large $4$};
  \draw (4.5,0.5) node{\Large $5$};
  \end{scope}
   \end{tikzpicture}}
   \end{figure}
Note that $D\ominus D_{\textup{min}}$ consists of a single connected component and there are multiple ways of decomposing it into a union of enclosed graphs whose completion is a union of snake graphs. One of these ways is the one described in Proposition~\ref{lemma_decomposition}, that is:
\begin{figure}[h!]\scalebox{0.8}{
    \centering
    \begin{tikzpicture}
       \begin{scope}
  \draw (0,0) \rectanglepath;
  \draw (1,0) \rectanglepath;
  \draw (2,0) \rectanglepath;
  \draw (3,0) \rectanglepath;
  \draw (4,0) \rectanglepath;
 \draw[very thick, red] (1,1)--(2,1);
 \draw[very thick, red] (1,0)--(2,0);
 \draw[very thick, red] (3,1)--(4,1);
 \draw[very thick, red] (3,0)--(4,0);
  \draw[very thick, blue]  (0,1)--(1,1);
 \draw[very thick, blue]   (0,0)--(1,0);
  \draw[very thick, red]  (0,0)--(0,1);
  \draw[very thick, blue]  (2,1)--(3,1);
 \draw[very thick, blue]   (2,0)--(3,0);
   \draw[very thick, blue]  (4,1)--(5,1);
 \draw[very thick, blue]   (4,0)--(5,0);
 \draw[very thick, red]  (5,0)--(5,1);
 \draw (0.5,0.5) node{\Large $1$};
  \draw (1.5,0.5) node{\Large $2$};
  \draw (2.5,0.5) node{\Large $3$};
  \draw (3.5,0.5) node{\Large $4$};
  \draw (4.5,0.5) node{\Large $5$};
  \end{scope}
  \begin{scope}[xshift=5.5cm, yshift=0.5cm]
      \draw (0,0) node {\Large $\bigsqcup$};
   \end{scope}
   \begin{scope}[xshift=6cm]
      \draw (0,0) \rectanglepath;
       \draw[very thick,  red] (0,0)--(1,0) (0,1)--(1,1);
       \draw[very thick,  blue] (0,0)--(0,1) (1,0)--(1,1);
      \draw (0.5,0.5) node{\Large $2$};
   \end{scope}
   \begin{scope}[xshift=7.5cm, yshift=0.5cm]
      \draw (0,0) node {\Large $\bigsqcup$};
   \end{scope}
   \begin{scope}[xshift=8cm]
      \draw (0,0) \rectanglepath;
       \draw[very thick,  red] (0,0)--(1,0) (0,1)--(1,1);
       \draw[very thick,  blue] (0,0)--(0,1) (1,0)--(1,1);
      \draw (0.5,0.5) node{\Large $4$};
   \end{scope}
    \end{tikzpicture}}
\end{figure}

The module over $\widetilde{\Lambda}$ corresponding to $D$ is then $N:=N_D=\begin{smallmatrix}
    1&2&3&4&5\\
    &2&&4
\end{smallmatrix}$. Using the notation from Section~\ref{Sec:tensoring}, we have that the union of the two cycles enclosing tile $2$ and tile $4$ corresponds to the module $N\epsilon=\begin{smallmatrix}2\end{smallmatrix}\oplus \begin{smallmatrix}4\end{smallmatrix}$, while the cycle enclosing the five tiles corresponds to the module $N/N\epsilon=\begin{smallmatrix}
    1&&3&&5\\
    &2&&4
\end{smallmatrix}$, both viewed as modules over $\Lambda$. Moreover, note that this is true more generally. In fact, given a double dimer $D$, with associated module $N$ in $\module \widetilde{\Lambda}$, using the notation from Proposition~\ref{lemma_decomposition}, we have that $\bigcup_{i}\bigcup_j C_{i,j}$, for $j\geq 2$,  corresponds to $N\epsilon$, while $\bigcup_i C_{i,1}$ corresponds to $N/N\epsilon$ as objects in $\module\Lambda$.
\end{example}

Conversely, given an induced module $\widetilde M$ and a submodule $N$, we wish to associate a double dimer cover. Suppose $\mathcal{G}$ is the snake graph associated with a single copy of the restriction of $\widetilde{M}$ to $\Lambda$, that is the snake graph of $M$. Let $N=N_1 \oplus \dots \oplus N_k$,  with canonical embedding $\varphi$ and $w_j$ be the abstract optimal loopy word associated with $N_j$, for $j=1,\dots,k$, and $\mathcal{H}_i$ be the optimal  snake sub-multigraph of $\mathcal{G}$ associated with $w_j$ corresponding to the embedding $\varphi$.  
    Let  $E(\mathcal{H}_i)$ denote the collection of all double edges in $\mathcal{H}_i$ that are  not in $D_{\text{min}}$  of $\mathcal{G}$ plus all the single boundary edges in $\mathcal{H}_i$. Moreover, in $E(\mathcal{H}_i)$, complete any adjacent collection of single edges to a cycle (by adding either the missing bottom or left edge of the first tile and the missing top or right edge of the last tile).
Set

\[
D_\varphi=(\bigcup\limits_{i=1}^k E(\mathcal{H}_i))\cup D_{\textup{min}}\mid_{\mathcal{G}\backslash 
(\bigcup\limits_{i=1}^k\mathcal{H}_i)}.
\]

 \begin{example}
     Consider the submodule $N=\begin{smallmatrix}
        &&3\\
        1&2&3\\
        &2
    \end{smallmatrix}\oplus \begin{smallmatrix}6\\6\end{smallmatrix}$
    of the induced module $\widetilde{M}$ in $\module\widetilde\Lambda$ for $M=\begin{smallmatrix}
        &&&&5&\\
        &&&4&&6\\
        1&&3&&&\\
        &2
    \end{smallmatrix}$ and
    
    \begin{align*}\widetilde\Lambda = K \Bigg(\begin{tikzcd}[ampersand replacement=\&, arrow style=tikz,>=stealth,row sep=2em,column sep=2em]
1\arrow[r] \arrow[out=60,in=120, distance=1.5em,loop]\& 2
\arrow[out=60,in=120,distance=1.5em,loop]\& 3\arrow[out=60,in=120,distance=1.5em,loop]\arrow[l]\& 4\arrow[out=60,in=120,distance=1.5em,loop]\arrow[l]\& 5\arrow[out=60,in=120,distance=1.5em,loop]\arrow[l]\arrow[r]\& 6\arrow[out=60,in=120,distance=1.5em,loop]
\end{tikzcd}\Bigg)\Bigg/I.
\end{align*}
    Then 
    \begin{align*}
w_1=\begin{tikzcd}[ampersand replacement=\&, arrow style=tikz,>=stealth,row sep=2em,column sep=2em]
1\arrow[r] \& 2
\arrow[out=60,in=120,distance=1.5em,loop]\& 3\arrow[out=60,in=120,distance=1.5em,loop]\arrow[l]
\end{tikzcd}, \text{ and } w_2=\begin{tikzcd}[ampersand replacement=\&, arrow style=tikz,>=stealth,row sep=2em,column sep=2em]6.\arrow[out=60,in=120,distance=1.5em,loop]\end{tikzcd}
\end{align*}
The optimal snake sub-multigraphs of $\begin{tikzpicture}[scale=0.5]
  \draw (-0.8,0.5) node{$\mathcal{G}=$};
  \draw (0,0) \rectanglepath;
  \draw (1,0) \rectanglepath;
  \draw (2,0) \rectanglepath;
  \draw (2,1) \rectanglepath;
  \draw (3,1) \rectanglepath;
  \draw (4,1) \rectanglepath;
  \draw (0.5,0.5) node{\tiny $1$};
  \draw (1.5,0.5) node{\tiny $2$};
  \draw (2.5,0.5) node{\tiny $3$};
  \draw (2.5,1.5) node{\tiny $4$};
  \draw (3.5,1.5) node{\tiny $5$};
  \draw (4.5,1.5) node{\tiny $6$};
  \end{tikzpicture}$ associated to $w_1$ and $w_2$ are respectively
$ \begin{tikzpicture}[scale=0.5]
  \draw (-0.8,0.5) node{$\mathcal{H}_1:\,\,$};
  \draw (0,0) \rectanglepath;
  \draw (1,0) \rectanglepath;
  \draw (2,0) \rectanglepath;
  \draw[very thick, double, double distance=1.3pt] (1,0)--(2,0);
  \draw[very thick, double, double distance=1.3pt] (1,1)--(2,1);
  \draw[very thick, double, double distance=1.3pt] (2,0)--(3,0);
  \draw[very thick, double, double distance=1.3pt] (2,1)--(3,1);
  \draw[very thick, double, double distance=1.3pt] (3,0)--(3,1);
  \draw (0.5,0.5) node{\tiny $1$};
  \draw (1.5,0.5) node{\tiny $2$};
  \draw (2.5,0.5) node{\tiny $3$};
  \end{tikzpicture}$
   and $ \begin{tikzpicture}[scale=0.5]
  \draw (-0.8,0.5) node{$\mathcal{H}_2:\,\,$};
  \draw (0,0) \rectanglepath;
  \draw[very thick, double, double distance=1.3pt] (0,0)--(1,0);
  \draw[very thick, double, double distance=1.3pt] (1,0)--(1,1);
  \draw[very thick, double, double distance=1.3pt] (1,1)--(0,1);
  \draw[very thick, double, double distance=1.3pt] (0,1)--(0,0);
  \draw (0.5,0.5) node{\tiny $6$};
  \end{tikzpicture}$.
Then
\begin{align*}
D_\varphi={\color{blue} E(\mathcal{H}_1)}\cup {\color{magenta} E(\mathcal{H}_2)}\cup  {\color{cyan} D_{\textup{min}}\mid_{\mathcal{G}\backslash (\mathcal{H}_1\cup \mathcal{H}_2)}}=
   \begin{tikzpicture}[scale=0.7]
  \draw (1,0) \rectanglepath;
  \draw[very thick, blue] (1,0)--(0,0)--(0,1)--(1,1);
  \draw[very thick, dashed, blue] (1,0)--(1,1);
  \draw (2,0) \rectanglepath;
  \draw (2,1) \rectanglepath;
  \draw (3,1) \rectanglepath;
  \draw (4,1) \rectanglepath;
  \draw[very thick, blue, double, double distance=1.3pt] (2,0)--(3,0);
  \draw[very thick, blue, double, double distance=1.3pt] (2,1)--(3,1);
  \draw[very thick, magenta, double, double distance=1.3pt] (4,1)--(4,2);
  \draw[very thick, magenta, double, double distance=1.3pt] (5,1)--(5,2);
  \draw[very thick, cyan, double, double distance=1.3pt] (2,2)--(3,2);
  \draw (0.5,0.5) node{\tiny $1$};
  \draw (1.5,0.5) node{\tiny $2$};
  \draw (2.5,0.5) node{\tiny $3$};
  \draw (2.5,1.5) node{\tiny $4$};
  \draw (3.5,1.5) node{\tiny $5$};
  \draw (4.5,1.5) node{\tiny $6$};
  \end{tikzpicture} 
\end{align*}
 where the dashed edge on the first tile is obtained by completion to a cycle.
 Moreover, observe that  we have 
  $(D_\textup{min}\ominus D_\varphi)^c = \mathcal{H}_1\cup \mathcal{H}_2$.
\end{example}

 \begin{proposition}\label{prop:submod--dimer}
     With the notation above, the set of edges $D_\varphi$ of $\mathcal{G}$ is a double dimer cover of $\mathcal{G}$ and the submodule $N_{D_\varphi}$ of $\widetilde M$ associated with $D_\varphi$ agrees with the canonical submodule $N$ with the embedding $\varphi$.  
 \end{proposition}

 \begin{proof}

 We will verify first that each $E(\mathcal{H}_i)$ is a double dimer cover of $\mathcal{H}_i$. If $\mathcal{H}_i$ consists of only single edges, then $E(\mathcal{H}_i)$ contains the boundary of $\mathcal{H}_i$ and thus is a double dimer cover. 
 
Suppose now $C$ is a cycle in $E(\mathcal{H}_i)$ on the tiles $(G_s,\dots,G_t)$. Note that if either there exists a tile $G_{s-1}$ preceding $C$ or a tile $G_{t+1}$ succeeding $C$ in $\mathcal{H}_i$, then these tiles must have double boundary edges in $\mathcal{H}_i$, because otherwise we could have extended $C$ to a larger cycle. Moreover, all the (double) boundary edges of $G_{s-1}$ or $G_{t+1}$ must be in $\mathcal{H}_i$ as $\mathcal{H}_i$ is an optimal (and therefore a good) snake multigraph.

 We claim that the double boundary edge of $G_{s-1}$ or $G_{t+1}$ that is adjacent to the tile $G_s$ or $G_t$, respectively, must be in the minimal matching. We will only argue for the existence of $G_{s-1}$ as the other case is similar.

 Notice that the tiles $(G_{s-1},G_s)$ give rise to an \emph{inverse} arrow because otherwise since $\mathcal{H}_i$ is an optimal snake multigraph and $G_{s-1}$ contains two copies of its boundary edges, $G_s$ would also contain two copies of its boundary edges. 

Consider the maximal inverse string associated to $\mathcal{G}$ that contains the vertices $s-1 \leftarrow s$ and suppose the socle vertex of this maximal string is $w$. Note that the tile $G_w$ corresponding to the vertex $w$ will have two copies of the boundary edges in $D_{\text{min}}(\mathcal{G})|_{\mathcal{H}_i}$  by Remark~\ref{remark_headofarrow} since $w$ is a socle. Without loss of generality, we may assume the tile $G_{s-1}$ is to the left of $G_s$ so the minimal double dimer cover would have the (blue) double edges indicated in the figure below. This implies that two copies of the boundary edge of $G_{s-1}$ that is not adjacent to $G_s$ is in $E(\mathcal{H}_i)$. Combining these edges with the cycles that is obtained in the construction of $E(\mathcal{H}_i)$ will give rise to a double double cover on $\mathcal{H}_i$, thus each $E(\mathcal{H}_i)$ is a double dimer cover on $\mathcal{H}_i$.

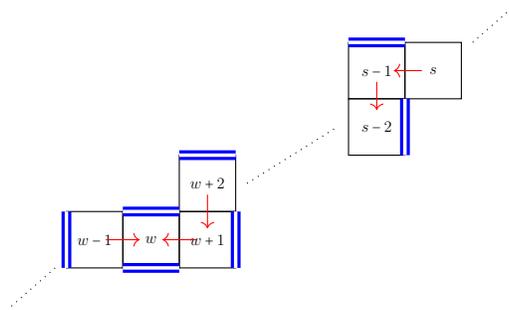
\begin{center}
    \begin{tikzpicture}[scale=0.75]
  \draw (0,0) \rectanglepath;
  \draw (0.5,0.5)  node[scale=.5]{$w-1$};
  \draw (1.5,0.5)  node[scale=.5]{$w$};
  \draw (2.5,0.5)  node[scale=.5]{$w+1$};
  \draw (2.5,1.5)  node[scale=.5]{$w+2$};
  \draw (1,0) \rectanglepath;
    \draw (2,0) \rectanglepath;
      \draw (2,1) \rectanglepath;
      \draw [->, red](0.7,0.5)--(1.3,0.5);
      \draw [->, red](2.3,0.5)--(1.7,0.5);
      \draw [->, red](2.5,1.3)--(2.5,0.7);
      \draw [dotted](3.2,1.5)--(4.8,2.5);
      \draw (5,2) \rectanglepath;
      \draw (5,3) \rectanglepath;
    \draw (6,3) \rectanglepath;
    \draw (5.5,2.5)  node[scale=.5]{$s-2$};
  \draw (5.5,3.5)  node[scale=.5]{$s-1$};
  \draw (6.5,3.5)  node[scale=.5]{$s$};
  \draw [->, red](5.5,3.3)--(5.5,2.8);
  \draw [<-, red](5.8,3.5)--(6.3,3.5);
  \draw [dotted](7.2,4)--(8,4.7);
  \draw [dotted](-.2,0)--(-1,-.7);
 \draw[very thick, blue, double, double distance=1.3pt]  (0,0)--(0,1);
 \draw[very thick, blue, double, double distance=1.3pt]  (1,0)--(2,0);
   \draw[very thick, blue, double, double distance=1.3pt]  (1,1)--(2,1);
     \draw[very thick, blue, double, double distance=1.3pt]  (3,0)--(3,1);
       \draw[very thick, blue, double, double distance=1.3pt]  (2,2)--(3,2);
       \draw[very thick, blue, double, double distance=1.3pt]  (6,2)--(6,3);
       \draw[very thick, blue, double, double distance=1.3pt]  (5,4)--(6,4);
  \end{tikzpicture}
  \end{center}

Consider the minimal dimer cover on $\mathcal{G}$ and replace each $D_{\text{min}}|_{\mathcal{H}_i}$ by $E(\mathcal{H}_i)$. 
This is a double dimer cover on $\mathcal{G}$ and coincides with $D_\varphi$. In view of Proposition~\ref{prop:dimer--submod}, the canonical submodule associated to $D_\varphi$ agrees with $N$, as required.
 \end{proof}

\begin{theorem}\label{thm_lattice_bijection_double}
         With the notation above, the lattice of the double dimer covers of $\mathcal{G}$ is in bijection with the submodule lattice of $\widetilde M_\mathcal{G}$.
 \end{theorem}

 \begin{proof}
By Proposition~\ref{prop:dimer--submod} and Propoposition~\ref{prop:submod--dimer}, we have an equality of the two lattices as sets. Let $H(\mathcal{G})$ and $H(\widetilde{M}_\mathcal{G})$ be the Hasse diagrams of the double dimer cover lattice $\mathcal{L}(\mathcal{G})$ and the canonical submodule lattice $\mathcal{L}(\widetilde{M}_\mathcal{G})$, respectively. We will verify the existing of an edge in one lattice if and only if there is an edge between their images in the other lattice.

Let $D$ and $D'$ be double dimer covers such that there is an edge between those vertices in $H(\mathcal{G})$. Suppose $D\ominus D_\textup{min}=\bigcup\limits_i\mathcal{H}_i$ and $D'\ominus D_\textup{min}=\bigcup\limits_i\mathcal{H}_i'$.
Suppose $N_D,N_{D'}$ are the corresponding submodules in Proposition~\ref{prop:dimer--submod} with canonical embeddings $\varphi: \bigcup\limits_i\mathcal{H}_i \hookrightarrow \mathcal{G}$ and $\varphi': \bigcup\limits_i\mathcal{H}_i' \hookrightarrow \mathcal{G}$, respectively. Note that the dimer covers $D$ and $D'$ only differ on a single tile $G$ where the symmetric difference on those tiles consists of the boundary edges of $G$. This could occur in two different configurations, either when one of them has double edges (on opposite sides in $G$) and the other consists of alternating single edges in $G$ or when both of them consist of single edges complementary to each other (i.e. one consisting of two horizontal and the other of two vertical edges). In the latter, the vertex $v$ corresponding to $G$ is either in the loopy word $w_D$ or in $w_{D'}$ corresponding to $N_D$ and $N_{D'}$, respectively. This implies that $N_D$ and $N_{D'}$ agree everywhere except for the neighbouring of the vertex $v$. In this setting, the argument that $v$ corresponds to a single top is similar to that of \cite{CS21}. For the former configuration, if the double edges agree with $D_\textup{min}$, then we again have the same argument as above. If not, then the double edges are in $D_\textup{max}$ so the corresponding loopy word contains the vertex $v$ with a loop whereas the dimer cover consisting of single edges corresponds to only the vertex $v$ without a loop. Hence the corresponding representations only differ at the vertex $v$, in which one has $K^2$ but the other $K$. These two representations only differ by a top at vertex $v$ and thus there is an edge between them in $H(\widetilde{M}_\mathcal{G})$. Moreover, without loss of generality, assume $\bigcup\limits_i\mathcal{H}_i  \backslash \bigcup\limits_i\mathcal{H}_i'=G$, then $\varphi\mid_{\bigcup\limits_i\mathcal{H}_i'}=\varphi '$ implying that $\varphi\mid_{N_{D'}}=\varphi '$.

Conversely, suppose that $(N_{w_i},\varphi_i)$ and $(N_{w_j},\varphi_j)$ are two canonically embedded submodules of $\widetilde{M}_\mathcal{G}$ connected by an edge in $H(\widetilde{M}_{\mathcal{G}})$. Therefore, without loss of generality, there is exactly one vertex or a loop contained in $w_i$ but not in $w_j$ and all other vertices, arrows, inverse arrows and loops in $w_i$ and $w_j$ are the same. Let $D_{\varphi_i}$ and $D_{\varphi_j}$ be the double dimer covers associated to  $(N_{w_i},\varphi_i)$ and $(N_{w_j},\varphi_j)$, respectively. Consider $D_{\varphi_i}$ and $D_{\varphi_j}$ and observe that  $D_{\varphi_i} \ominus D_{\varphi_j}$ is a single tile $G$ of $\mathcal{G}$ consisting only of single edges as their induced loopy words differ either by a loop or a vertex. Since two dimer covers have an edge in a double dimer cover lattice if any only if their symmetric difference is a single tile, we obtain the desired result. 
\end{proof}

\begin{remark}
    Note that Theorem~\ref{thm_lattice_bijection_double} generalises straightforwardly if we consider $d$-dimer covers of a snake graph and the induced module in $\module \widetilde{\Lambda}= \module (\Lambda\otimes_K K[\epsilon]/(\epsilon^d))$, where $d\in\mathbb{N}$. In a similar fashion, we obtain a bijection between the lattice of $d$-dimer covers of a snake graph $\mathcal{G}$ and the submodule lattice of $\widetilde{M_{\mathcal{G}}}=M \otimes_\Lambda \widetilde{\Lambda}$.  
\end{remark}

For example, Figure~\ref{fig:4dimer} illustrates the bijection between the lattice of 4-dimer covers of $\mathcal{G}={\begin{tikzpicture}[scale=.5]
        \draw (0,0) \rectanglepath;
        \draw (1,0) \rectanglepath;
  \draw (0.5,0.5) node{\tiny $1$};
  \draw (1.5,0.5) node{\tiny $2$};
    \end{tikzpicture}}$ 
    and the submodule lattice of $\widetilde{M_\mathcal{G}}=
    \begin{smallmatrix}
    &&&1&\\&&1&&2\\&1&&2\\1&&2&&\\&2&&& \end{smallmatrix}$.

\tikzset{
    triple/.style args={[#1] in [#2] in [#3]}{
        #1,preaction={preaction={draw,#3},draw,#2}
    }
}
\tikzset{
    quadruple/.style args={[#1] in [#2] in [#3] in [#4]}{
        #1,preaction={preaction={preaction={draw, #4},draw,#3},draw,#2}}}
        
    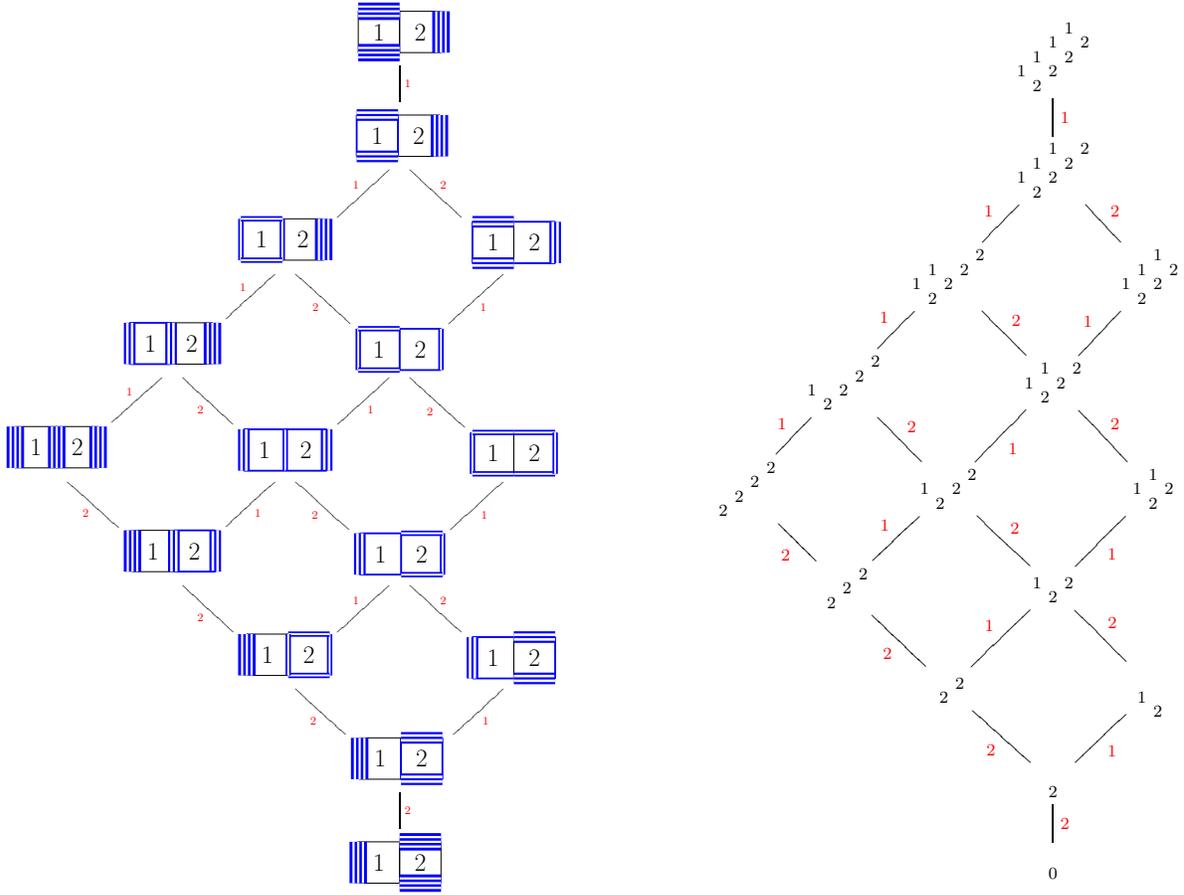
\begin{figure}[h!]
    \centering
\scalebox{0.55}{\xymatrix@C=0.4em{
&&&
{\begin{tikzpicture}
  \draw (0,0) \rectanglepath;
  \draw (1,0) \rectanglepath;
 \draw[quadruple={[line width=1.5pt, white] in
  [line width=5pt,blue] in
      [line width=8pt,white] in
     [line width=12pt,blue]}]
     (0,0)--(1,0);
  \draw[quadruple={[line width=1.5pt, white] in
  [line width=5pt,blue] in
      [line width=8pt,white] in
     [line width=12pt,blue]}] (2,1) to (2,0);
  \draw[quadruple={[line width=1.5pt, white] in
  [line width=5pt,blue] in
      [line width=8pt,white] in
     [line width=12pt,blue]}] (0,1) to (1,1);
 \draw (0.5,0.5) node{\Large $1$};
  \draw (1.5,0.5) node{\Large $2$};
\end{tikzpicture}}
\ar@[red]@{-}[d]^-{\color{red}1}
\\
&&&
{\begin{tikzpicture}
  \draw (0,0) \rectanglepath;
  \draw (1,0) \rectanglepath;
  \draw[very thick, blue] (0,0)--(0,1);
  \draw[very thick, blue] (1,1)--(1,0);
  \draw[quadruple={[line width=1.5pt, white] in
  [line width=5pt,blue] in
      [line width=8pt,white] in
     [line width=12pt,blue]}] (2,1)--(2,0);
  \draw[triple={[line width=1.5pt,blue] in
     [line width=5pt,white] in
    [line width=8pt,blue]}] (1,1)--(0,1);
  \draw[triple={[line width=1.5pt,blue] in
     [line width=5pt,white] in
    [line width=8pt,blue]}] (0,0)--(1,0);
 \draw (0.5,0.5) node{\Large $1$};
  \draw (1.5,0.5) node{\Large $2$};
\end{tikzpicture}}
\ar@[red]@{-}[ld]_-{\color{red}1}\ar@[red]@{-}[rd]^-{\color{red}2}
\\
&&
{\begin{tikzpicture}
  \draw (0,0) \rectanglepath;
  \draw (1,0) \rectanglepath;
  \draw[very thick, double, double distance=1.3pt, blue] (0,0)--(0,1);
  \draw[very thick, double, double distance=1.3pt, blue] (1,0)--(1,1);
  \draw[quadruple={[line width=1.5pt, white] in
  [line width=5pt,blue] in
      [line width=8pt,white] in
     [line width=12pt,blue]}] (2,1)--(2,0);
  \draw[very thick, double, double distance=1.3pt, blue] (1,1)--(0,1);
  \draw[very thick, double, double distance=1.3pt, blue] (0,0)--(1,0);
 \draw (0.5,0.5) node{\Large $1$};
  \draw (1.5,0.5) node{\Large $2$};
\end{tikzpicture}}
\ar@[red]@{-}[rd]_-{\color{red}2} \ar@[red]@{-}[ld]_-{\color{red}1}
&&
{\begin{tikzpicture}
  \draw (0,0) \rectanglepath;
  \draw (1,0) \rectanglepath;
  \draw[very thick, blue] (0,0)--(0,1);
  \draw[very thick, blue] (1,1)--(2,1);
  \draw[very thick, blue] (1,0)--(2,0);
  \draw[triple={[line width=1.5pt,blue] in
     [line width=5pt,white] in
    [line width=8pt,blue]}] (2,1)--(2,0);
  \draw[triple={[line width=1.5pt,blue] in
     [line width=5pt,white] in
    [line width=8pt,blue]}] (1,1)--(0,1);
  \draw[triple={[line width=1.5pt,blue] in
     [line width=5pt,white] in
    [line width=8pt,blue]}] (0,0)--(1,0);
 \draw (0.5,0.5) node{\Large $1$};
  \draw (1.5,0.5) node{\Large $2$};
\end{tikzpicture}}
\ar@[red]@{-}[ld]^-{\color{red}1}
\\
&{\begin{tikzpicture}
  \draw (0,0) \rectanglepath;
  \draw (1,0) \rectanglepath;
  \draw[very thick, blue]  (0,0)--(1,0);
  \draw[very thick, blue]  (0,1)--(1,1);
  \draw[quadruple={[line width=1.5pt, white] in
  [line width=5pt,blue] in
      [line width=8pt,white] in
     [line width=12pt,blue]}] (2,1)--(2,0);
  \draw[triple={[line width=1.5pt,blue] in
     [line width=5pt,white] in
    [line width=8pt,blue]}] (0,0)--(0,1);
  \draw[triple={[line width=1.5pt,blue] in
     [line width=5pt,white] in
    [line width=8pt,blue]}] (1,0)--(1,1);
 \draw (0.5,0.5) node{\Large $1$};
  \draw (1.5,0.5) node{\Large $2$};
\end{tikzpicture}}
\ar@[red]@{-}[rd]_-{\color{red}2} \ar@[red]@{-}[ld]_-{\color{red}1}
&&
{\begin{tikzpicture}
  \draw (0,0) \rectanglepath;
  \draw (1,0) \rectanglepath;
  \draw[very thick, double, double distance=1.3pt, blue] (0,0)--(0,1);
  \draw[very thick, blue] (1,1)--(2,1);
  \draw[very thick, blue] (1,0)--(2,0);
  \draw[very thick, blue] (1,0)--(1,1);
  \draw[very thick, double, double distance=1.3pt, blue] (2,1)--(2,0);
  \draw[very thick, double, double distance=1.3pt, blue] (1,1)--(0,1);
  \draw[very thick, double, double distance=1.3pt, blue] (0,0)--(1,0);
 \draw (0.5,0.5) node{\Large $1$};
  \draw (1.5,0.5) node{\Large $2$};
\end{tikzpicture}}
\ar@[red]@{-}[rd]_-{\color{red}2} \ar@[red]@{-}[ld]^-{\color{red}1}
\\
{\begin{tikzpicture}
  \draw (0,0) \rectanglepath;
  \draw (1,0) \rectanglepath;
  \draw[quadruple={[line width=1.5pt, white] in
  [line width=5pt,blue] in
      [line width=8pt,white] in
     [line width=12pt,blue]}] (2,1)--(2,0);
  \draw[quadruple={[line width=1.5pt, white] in
  [line width=5pt,blue] in
      [line width=8pt,white] in
     [line width=12pt,blue]}] (0,0)--(0,1);
  \draw[quadruple={[line width=1.5pt, white] in
  [line width=5pt,blue] in
      [line width=8pt,white] in
     [line width=12pt,blue]}] (1,0)--(1,1);
 \draw (0.5,0.5) node{\Large $1$};
  \draw (1.5,0.5) node{\Large $2$};
\end{tikzpicture}}
\ar@[red]@{-}[rd]_-{\color{red}2}
&&{\begin{tikzpicture}
  \draw (0,0) \rectanglepath;
  \draw (1,0) \rectanglepath;
  \draw[triple={[line width=1.5pt,blue] in
     [line width=5pt,white] in
    [line width=8pt,blue]}] (2,1)--(2,0);
  \draw[very thick, blue]  (0,1)--(1,1);
 \draw[very thick, blue]   (0,0)--(1,0);
  \draw[very thick, double, double distance=1.3pt, blue]  (1,0)--(1,1);
  \draw[very thick, blue] (1,0)--(2,0);
  \draw[very thick, blue] (1,1)--(2,1);
  \draw[triple={[line width=1.5pt,blue] in
     [line width=5pt,white] in
    [line width=8pt,blue]}]  (0,0)--(0,1);
 \draw (0.5,0.5) node{\Large $1$};
  \draw (1.5,0.5) node{\Large $2$};
\end{tikzpicture}}\ar@[red]@{-}[rd]_-{\color{red}2}\ar@[red]@{-}[ld]^-{\color{red}1}
&&{\begin{tikzpicture}
  \draw (0,0) \rectanglepath;
  \draw (1,0) \rectanglepath;
  \draw[very thick, double, double distance=1.3pt, blue] (0,0)--(0,1);
  \draw[very thick, double, double distance=1.3pt, blue] (1,1)--(2,1);
  \draw[very thick, double, double distance=1.3pt, blue] (1,0)--(2,0);
  \draw[very thick, double, double distance=1.3pt, blue] (2,1)--(2,0);
  \draw[very thick, double, double distance=1.3pt, blue] (1,1)--(0,1);
  \draw[very thick, double, double distance=1.3pt, blue] (0,0)--(1,0);
 \draw (0.5,0.5) node{\Large $1$};
  \draw (1.5,0.5) node{\Large $2$};
\end{tikzpicture}}\ar@[red]@{-}[ld]^-{\color{red}1}
\\
&{\begin{tikzpicture}
  \draw (0,0) \rectanglepath;
  \draw (1,0) \rectanglepath;
  \draw[triple={[line width=1.5pt,blue] in
     [line width=5pt,white] in
    [line width=8pt,blue]}] (2,1)--(2,0);
  \draw[very thick, blue] (1,0)--(2,0);
  \draw[very thick, blue] (1,1)--(2,1);
  \draw[quadruple={[line width=1.5pt, white] in
  [line width=5pt,blue] in
      [line width=8pt,white] in
     [line width=12pt,blue]}] (0,0)--(0,1);
  \draw[triple={[line width=1.5pt,blue] in
     [line width=5pt,white] in
    [line width=8pt,blue]}] (1,0)--(1,1);
 \draw (0.5,0.5) node{\Large $1$};
  \draw (1.5,0.5) node{\Large $2$};
\end{tikzpicture}}
\ar@[red]@{-}[rd]_-{\color{red}2}
&&{\begin{tikzpicture}
  \draw (0,0) \rectanglepath;
  \draw (1,0) \rectanglepath;
  \draw[very thick, double, double distance=1.3pt, blue] (2,1)--(2,0);
  \draw[very thick, blue]  (0,1)--(1,1);
 \draw[very thick, blue]   (0,0)--(1,0);
  \draw[very thick, blue]  (1,0)--(1,1);
  \draw[very thick, double, double distance=1.3pt, blue] (1,0)--(2,0);
  \draw[very thick, double, double distance=1.3pt, blue] (1,1)--(2,1);
  \draw[triple={[line width=1.5pt,blue] in
     [line width=5pt,white] in
    [line width=8pt,blue]}]  (0,0)--(0,1);
 \draw (0.5,0.5) node{\Large $1$};
  \draw (1.5,0.5) node{\Large $2$};
\end{tikzpicture}}
\ar@[red]@{-}[ld]_-{\color{red}1}\ar@[red]@{-}[rd]^-{\color{red}2}
\\
&&{\begin{tikzpicture}
  \draw (0,0) \rectanglepath;
  \draw (1,0) \rectanglepath;
  \draw[very thick, double, double distance=1.3pt, blue] (2,1)--(2,0);
  \draw[very thick, double, double distance=1.3pt, blue] (1,0)--(2,0);
  \draw[very thick, double, double distance=1.3pt, blue] (1,1)--(2,1);
  \draw[quadruple={[line width=1.5pt, white] in
  [line width=5pt,blue] in
      [line width=8pt,white] in
     [line width=12pt,blue]}] (0,0)--(0,1);
  \draw[very thick, double, double distance=1.3pt, blue] (1,0)--(1,1);
 \draw (0.5,0.5) node{\Large $1$};
  \draw (1.5,0.5) node{\Large $2$};
\end{tikzpicture}}
\ar@[red]@{-}[rd]_-{\color{red}2}
&&
{\begin{tikzpicture}
  \draw (0,0) \rectanglepath;
  \draw (1,0) \rectanglepath;
  \draw[very thick, blue]   (0,1)--(1,1);
 \draw[very thick, blue]  (0,0)--(1,0);
  \draw[triple={[line width=1.5pt,blue] in
     [line width=5pt,white] in
    [line width=8pt,blue]}]  (0,0)--(0,1);
    \draw[triple={[line width=1.5pt,blue] in
     [line width=5pt,white] in
    [line width=8pt,blue]}]  (1,0)--(2,0);
  \draw[very thick, blue]  (2,0)--(2,1);
    \draw[triple={[line width=1.5pt,blue] in
     [line width=5pt,white] in
    [line width=8pt,blue]}]  (2,1)--(1,1); 
 \draw (0.5,0.5) node{\Large $1$};
  \draw (1.5,0.5) node{\Large $2$};
\end{tikzpicture}}
\ar@[red]@{-}[ld]^-{\color{red}1}
\\
&&&{\begin{tikzpicture}
  \draw (0,0) \rectanglepath;
  \draw (1,0) \rectanglepath;
  \draw[quadruple={[line width=1.5pt, white] in
  [line width=5pt,blue] in
      [line width=8pt,white] in
     [line width=12pt,blue]}] (0,0)--(0,1);
   \draw[very thick, blue]  (1,0)--(1,1);
    \draw[very thick, blue]  (2,0)--(2,1);
  \draw[triple={[line width=1.5pt,blue] in
     [line width=5pt,white] in
    [line width=8pt,blue]}]  (2,0)--(1,0);
    \draw[triple={[line width=1.5pt,blue] in
     [line width=5pt,white] in
    [line width=8pt,blue]}] (2,1)--(1,1);
 \draw (0.5,0.5) node{\Large $1$};
  \draw (1.5,0.5) node{\Large $2$};
\end{tikzpicture}}
\ar@[red]@{-}[d]^-{\color{red}2}
\\
&&&{\begin{tikzpicture}
  \draw (0,0) \rectanglepath;
  \draw (1,0) \rectanglepath;
 \draw[quadruple={[line width=1.5pt, white] in
  [line width=5pt,blue] in
      [line width=8pt,white] in
     [line width=12pt,blue]}]
     (0,0)--(0,1);
  \draw[quadruple={[line width=1.5pt, white] in
  [line width=5pt,blue] in
      [line width=8pt,white] in
     [line width=12pt,blue]}] (1,0) to (2,0);
  \draw[quadruple={[line width=1.5pt, white] in
  [line width=5pt,blue] in
      [line width=8pt,white] in
     [line width=12pt,blue]}] (1,1) to (2,1);
 \draw (0.5,0.5) node{\Large $1$};
  \draw (1.5,0.5) node{\Large $2$};
\end{tikzpicture}}
}
}
\qquad \qquad 
\scalebox{0.8}{\xymatrix@C=0.4em@R=1.5em{
&&&
{\begin{smallmatrix}
    &&&1\\&&1&&2\\&1&&2\\1&&2\\&2
\end{smallmatrix}}
\ar@[red]@{-}[d]^-{\color{red}1}
\\
&&&
{\begin{smallmatrix}
    &&1&&2\\&1&&2\\1&&2\\&2
\end{smallmatrix}}
\ar@[red]@{-}[rd]^-{\color{red}2}\ar@[red]@{-}[ld]_-{\color{red}1}
\\
&&
{\begin{smallmatrix}
    &&&&2\\&1&&2\\1&&2\\&2
\end{smallmatrix}}
\ar@[red]@{-}[rd]^-{\color{red}2}\ar@[red]@{-}[ld]_-{\color{red}1}
&&
{\begin{smallmatrix}
    &&1&&\\&1&&2\\1&&2\\&2
\end{smallmatrix}}
\ar@[red]@{-}[ld]_-{\color{red}1}
\\
&
{\begin{smallmatrix}
    &&&&2\\&&&2\\1&&2\\&2
\end{smallmatrix}}
\ar@[red]@{-}[rd]^-{\color{red}2}\ar@[red]@{-}[ld]_-{\color{red}1}
&&{\begin{smallmatrix}
    &1&&2\\1&&2\\&2
\end{smallmatrix}}\ar@[red]@{-}[ld]^-{\color{red}1}\ar@[red]@{-}[rd]^-{\color{red}2}
\\
{\begin{smallmatrix}
    &&&&2\\&&&2\\&&2\\&2
\end{smallmatrix}}
\ar@[red]@{-}[rd]_-{\color{red}2}
&&{\begin{smallmatrix}
    &&&2\\1&&2\\&2
\end{smallmatrix}}
\ar@[red]@{-}[rd]^-{\color{red}2}\ar@[red]@{-}[ld]_-{\color{red}1}
&&{\begin{smallmatrix}
    &1&\\1&&2\\&2
\end{smallmatrix}}\ar@[red]@{-}[ld]^-{\color{red}1}
\\
&{\begin{smallmatrix}
    &&&2\\&&2\\&2
\end{smallmatrix}}
\ar@[red]@{-}[rd]_-{\color{red}2}
&&{\begin{smallmatrix}
    &&\\1&&2\\&2
\end{smallmatrix}}
\ar@[red]@{-}[ld]_-{\color{red}1}\ar@[red]@{-}[rd]^-{\color{red}2}
\\
&&{\begin{smallmatrix}
    &&\\&&2\\&2
\end{smallmatrix}}
\ar@[red]@{-}[rd]_-{\color{red}2}
&&
{\begin{smallmatrix}
    \\\\&&\\1&&\\&2
\end{smallmatrix}}
\ar@[red]@{-}[ld]^-{\color{red}1}
\\
&&&{\begin{smallmatrix}
    &&\\&&\\&2
\end{smallmatrix}}
\ar@[red]@{-}[d]^-{\color{red}2}
\\
&&&{\begin{smallmatrix}
    &&\\&&\\&0
\end{smallmatrix}}
}
}
    \caption{An example of the lattice bijection for $d=4$.}
    \label{fig:4dimer}
\end{figure}

Our lattice bijection between $d$-dimer covers and submodules of the corresponding induced module induces a formula for the number of $d$-dimer covers in terms of the number of submodules of the corresponding induced module. Comparing this with the enumerative results from \cite{musiker2023}, we may relate the number of submodules of an induced module with matrix products or the higher continued fractions in the sense of \cite{musiker2023}. 

    \begin{corollary} Let $\mathcal{G}$ be snake graph and $\widetilde{M_\mathcal{G}}$ be the corresponding induced module as above.  
        \begin{enumerate}
            \item The number of $d$-dimer covers of $\mathcal{G}$ is equal to the number of submodules of $\widetilde{M_{\mathcal{G}}}$.
            \item The number of submodules of $\widetilde{M_{\mathcal{G}}}$ may be computed by the top left entry of the matrix product given in \cite[Theorem 1.1]{musiker2023}.
            \item The number of submodules of $\widetilde{M_{\mathcal{G}}}$ may be computed by the numerator of the continued fraction given in \cite[Theorem 1.3]{musiker2023}.
        \end{enumerate}
    \end{corollary}


\section{``Super" Caldero-Chapoton map for type $A$}\label{Sec:superCC}

Putting together the results proved in the previous sections, we now construct a super analogue of the Caldero-Chapoton map for type $A$ and prove it agrees with Musiker, Ovenhouse and Zhang's formula for computing super lambda lengths.
We first present some general results for $\widetilde{\Lambda}=\Lambda\otimes_K K[\epsilon]/(\epsilon^2)$, and then specialise to the case when $\Lambda$ is of type $A$ and prove the main result.

Let $\widetilde{M_\gamma}$ be an indecomposable induced module in mod$\,\widetilde{\Lambda}$ and $\mathcal{G}_\gamma$ its corresponding snake graph.  Denote the minimal double dimer cover by $D_{\textup{min}}$ and the minimal (single) dimer cover of $\mathcal{G}_\gamma$ by $P_{\textup{min}}$.

\begin{lemma}\label{lemma_index_double}
   With the notation above, we have that
   \begin{align*}
       \dfrac{\textup{wt}_2 (D_{\textup{min}})}{\textup{cross}(\gamma)}=X^{\textup{ind}_{\widetilde{\Lambda}}(\widetilde{M_\gamma})}.
   \end{align*}
\end{lemma}
\begin{proof}
    By Lemma \ref{lem:index}, we have that $\textup{ind}_{\widetilde{\Lambda}}(\widetilde{M_\gamma})=\textup{ind}_{\Lambda}(M_\gamma)$. Moreover, we have that
    \begin{align*}
        \textup{wt}_2(D_{\textup{min}})=\sqrt{\textup{wt}(P_{\textup{min}})}\sqrt{\textup{wt}(P_{\textup{min}})}= \textup{wt}(P_{\textup{min}}).
    \end{align*}
    The result then follows from Remark \ref{remark_BZ_hidden}.
\end{proof}

Recall that by Lemma \ref{lemma_bilinear_equal}, we have that  $\langle S_i , \oplus_j S_j^{m_j}\rangle_{\widetilde{\Lambda}}=\langle S_i , \oplus_j S_j^{m_j}\rangle_{\Lambda}$, hence we may omit the algebra over which the bilinear form is taken.
 Moreover, note that since $D_{\textup{min}}$ consists only of double edges, then it has no cycles and $\textup{wt}_2 (D_{\textup{min}})$ is given by the edge component. This is not true for a general double dimer cover, as it may contain cycles. In the following lemma, we only work with the edge component of the weight of $D_N$. We will study the cycle component in a subsequent result.

\begin{lemma}\label{lemma_weight_edge_component}
    With the notation above, let $N$ be a submodule of $\widetilde{M_\gamma}$  and $D_N$ be the double dimer cover associated to $N$. We have that
  \begin{align*}
        \dfrac{ \prod\limits_{a \in D_N} \sqrt{x_a}}{\textup{wt}_2(D_{\textup{min}})}= \prod\limits_{i=1}^n \sqrt{x_i}^{\langle S_i , \oplus_j S_j^{m_j}\rangle},
    \end{align*}
    where by $a\in D_N$ we mean that $a$ is an edge in the double dimer $D_N$ (if it is a double edge, it is taken twice) and $\underline{\textup{dim}}(N)=\underline{\textup{dim}}(\oplus_j S_j^{m_j})$.
\end{lemma}

\begin{proof}
    By Remark~\ref{remark_superimpose_single_dimers}, the double dimer cover $D_N$ can be obtained by superimposing two dimer covers of $\mathcal{G}$, say $P_{N_1}$ and $P_{N_2}$. Recall that $N_1$ and $N_2$ correspond to submodules $N_1$ and $N_2$ of the module $M_\gamma$ in mod$\,\Lambda$. Then, following the definitions, it is easy to see that
    \begin{align*}
         \prod\limits_{a \in D_N} \sqrt{x_a}=\sqrt{\textup{wt}(P_{N_1})} \sqrt{\textup{wt}(P_{N_2})}.
    \end{align*}
    Similarly, $D_{\textup{min}}$ can be obtained by superimposing two copies of the minimal (single) dimer cover $P_{\textup{min}}$ of $\mathcal{G}_\gamma$ and
    \begin{align*}
        \textup{wt}_2(D_{\textup{min}})=\sqrt{\textup{wt}(P_{\textup{min}})}\sqrt{\textup{wt}(P_{\textup{min}})}.
    \end{align*}
    Hence
    \begin{align*}
        \dfrac{ \prod\limits_{a \in D_N} \sqrt{x_a}}{\textup{wt}_2(D_{\textup{min}})}= \dfrac{\sqrt{\textup{wt}(P_{N_1})} \sqrt{\textup{wt}(P_{N_2})}}{\sqrt{\textup{wt}(P_{\textup{min}})}\sqrt{\textup{wt}(P_{\textup{min}})}}.
    \end{align*}

    Using the definition of symmetric difference and properties of union, intersection and difference of multisets, it is easy to see that $D_N\ominus D_{\textup{min}}$ is obtained superimposing $P_{N_1}\ominus P_{\textup{min}}$ and $P_{N_2}\ominus P_{\textup{min}}$. That is, the face weight of the first is equal to the sum of the face weights of the other two. Then, by Proposition~\ref{prop:dimer--submod} and the corresponding classic one, see \cite[Proposition 3.1]{CS21}, we conclude that $\underline{\textup{dim}}(N)=\underline{\textup{dim}}(N_1)+\underline{\textup{dim}}(N_2)$. The result then follows by Remark~\ref{remark_BZ_hidden}.
    \end{proof}

From now on, we focus on Jacobian algebras of type $A$, that is coming from triangulations of polygons, where we further assume triangulations do not have internal triangles.
We first prove a lemma about the Euler-Poincare characteristic. 

\begin{lemma} \label{lemma2:grassmannians computation equal 1} 
Let $\Ltil=\Lambda\otimes_K K[\epsilon]/(\epsilon^2)$ where $\Lambda$ is a Jacobian algebra coming from a triangulation (with no internal triangles) of a polygon.
Let $M$ be an indecomposable representation over $\Lambda$ and $M\otimes_\Lambda \Ltil = \widetilde{M}$ be the corresponding induced representation over $\Ltil$.    
Then for each nonempty submodule Grassmannian $\mathrm{Gr}_{\mathbf{e}} (\widetilde{M})$ the \emph{Euler-Poincar\'e characteristic} $\chi ( \mathrm{Gr}_{\mathbf{e}} (\widetilde{M}) )$ is 1. 
\end{lemma}

\begin{proof}
    Recall that the quiver of $\Ltil$ is obtained from the quiver $Q$ of $\Lambda$ by adding loops $\epsilon_i$ for each vertex together with the relations $\epsilon_i\alpha = \alpha\epsilon_i$ for every arrow $\alpha$.  Since $\Lambda$ is a Jacobian algebra coming from a triangulation of a polygon, then every indecomposable representation $M$ of $\Lambda$ is at most one dimensional at every vertex.  Thus, the induced representation $V:=M\otimes_\Lambda \Ltil$ satisfies the following conditions: 
\begin{enumerate}
    \item All $V_i$ are $K$ vector spaces of dimension 0 or 2.
    \item The matrices are $V_{\epsilon_i} = \begin{bsmallmatrix} 0 & 0 \\  1 &  0\end{bsmallmatrix}$
    for all $i\in Q_0$ and $V_\alpha$ is the natural inclusion for all arrows $\alpha$ in $Q$.
\end{enumerate}

Let $W = (W_i)_{i \in Q_0}$ be a subrepresentation of $V$. Then either $W_i$ is $K^2$, or it is $0$, or it is a one dimensional subspace $\overline{(x,y)}$ of $K^2$. Suppose $W_i$ is one dimensional, then the subrepresentation conditions require that for each generator of this subspace and any $(a,b)^t\in W_i$ we have: 
\[  \begin{bsmallmatrix} 0 & 0 \\ 1 & 0\end{bsmallmatrix} (a,b)^t = (0, a) \in \overline{(x,y)}. \] 
Hence $\overline{(x,y)} = \overline{(0,1)}$. So all possible submodule Grassmannians are given by (a direct product of) subspaces, such that each subspace is uniquely determined by its dimension and can be either $\{ 0 \}, \overline{(0,1)}$ or $K^2$. Then for all  dimension vectors $\mathbf{e}$ each nonempty submodule Grassmannian $\mathrm{Gr}_{\mathbf{e}}(V)$, considered as an algebraic variety, is a direct product of points.
Thus $\chi ( \mathrm{Gr}_{\mathbf{e}} (V) )$ is 1 (see \cite[Section 2.3]{pla18}).
\end{proof}

In the next lemma, we use some notation introduced earlier on. In particular, $\mu(N)$ is as in Definition~\ref{defn_mu_module}, while by $c(D_N)$ we mean the set of cycles formed by edges of the double dimer cover $D_N$ and the weight $\textup{wt}(C)$ for a cycle $C$ in $c(D_N)$ is as in Definition~\ref{defn_weight_MOZ}.
    \begin{lemma}\label{lemma_tethas_behave}
    Let $\Ltil=\Lambda\otimes_K K[\epsilon]/(\epsilon^2)$ where $\Lambda$ is a Jacobian algebra coming from a triangulation (with no internal triangles) of a polygon.
Let $M$ be an indecomposable representation in $ \module \Lambda$ and $M\otimes_\Lambda \Ltil = \widetilde{M}$ be the corresponding induced representation in $ \module \Ltil$. Then, for any submodule $N$ of $\widetilde{M}$ in $ \module \Ltil$, we have that
\begin{align*}
    \mu (N)=\prod_{C\in c(D_N)} \wt(C).
\end{align*}
    \end{lemma}
    \begin{proof}
        By Proposition \ref{prop:submod--dimer}, the module $N$ corresponds to a double dimer cover $D_N$ of the snake graph associated to $M$. Moreover, the face function $h((D_N\ominus D_{\textup{min}})^c)$ gives the dimension vector of $N$, by Proposition \ref{prop:dimer--submod}. Since we are in the type $A$ setup, entries in the face function can only have values $0,1$ or $2$. The entries with value $2$ are exactly those giving the dimension vector of the largest induced submodule of $N$, and this is $N\epsilon\otimes\widetilde{\Lambda}$ by Section~\ref{Sec:tensoring}. On the other hand, the entries of value $1$ are exactly those giving $F(N):=N/(N\epsilon\otimes\widetilde{\Lambda})$ and this is a module in $ \module \Lambda$ by Section~\ref{Sec:tensoring}.
        
        Note that, by definition, an entry is $1$ in the face function exactly when the corresponding tile has single boundary edges. Moreover, since $D_{\textup{min}}$ consists of double edges, the single edges in $D_N\ominus D_{\textup{min}}$ are exactly those in $D_N$ and the completion of this symmetric difference does not affect boundary edges. Hence, the entries of value $1$ correspond to the tiles enclosed in cycles in $D_N$ and it is easy to see that two adjacent entries in the face function are $1$ if and only if the corresponding tiles are enclosed by the same cycle.
        
        Then, each indecomposable summand $A$ of $F(N)$, corresponds both to an arc $\gamma_A$ in the triangulated polygon (since it is an indecomposable module over $\Lambda$) and to a cycle $C_A$ in $c(D_N)$. By Proposition \ref{prop:triangle}, $A$ belongs to exactly two triangles in $\Delta(T)$: the first and last triangle $\gamma_A$ passes through. Then, $\mu(A)$ is the product in the $\Delta$-positive ordering of the two thetas corresponding to these two triangles and $\mu(N)=\mu(F(N))=\mu(\oplus A_i)=\prod\mu(A_i)$ by Definition~\ref{defn_mu_module}.
        
        On the other hand, for $C_A\in c(D_N)$, its weight wt$(C_A)$ is defined to be the product, in the positive ordering, of the first and last thetas appearing in $C_A$. When looking at the corresponding arc $\gamma_A$ these are the $\mu$-invariants associated to the first and last triangle the arc $\gamma_A$ crosses.
        The result then follows by recalling that the two orderings coincide by Proposition \ref{prop_orderings_equal} and the order of multiplication of the wt$(C_{A_i})$'s and the $\mu(A_i)$'s does not matter by Remark \ref{remark_summands_ordering}.
    \end{proof}

    \begin{notation}\label{notation_mu}
    Let $\Ltil=\Lambda\otimes_K K[\epsilon]/(\epsilon^2)$ where $\Lambda$ is a Jacobian algebra coming from a triangulation (with no internal triangles) of a polygon. Let $M$ be an indecomposable representation in $\text{mod}\,\Lambda$ and $M\otimes_\Lambda \Ltil = \widetilde{M}$ be the corresponding induced representation in $ \module \Ltil$.
    By Lemma~\ref{lemma2:grassmannians computation equal 1}, for each dimension vector $\textbf{e}\in\mathbb{Z}^n$, there is at most one submodule $N$ of $\widetilde{M}$ with dimension vector $\textbf{e}$. Then define $\mu_{\textbf{e}}:=\mu(N)$ if such a submodule $N$ exists and $\mu_{\textbf{e}}:=0$ otherwise.
    \end{notation}

   We now present a super-analogue of the CC-map (we recalled the classic one in Section \ref{subsect: CC map}) and prove it is an alternative way to compute the super lambda lengths for type $A$. In other words, we prove this gives the same result as \cite[Theorem 6.2]{musiker21}, also recalled in Theorem~\ref{thm_MOZ_formula}, and hence the same result as recursively applying the super Ptolemy relations, recalled in Section~\ref{subsec-super-L-len}.

    \begin{theorem}\label{thm_superCC}
      Let $\Ltil=\Lambda\otimes_K K[\epsilon]/(\epsilon^2)$ where $\Lambda$ is a Jacobian algebra coming from a triangulation $T$ (with no internal triangles) of an $(n+3)$-gon. For an arc $\gamma$ in the polygon that is not in $T$, let $\widetilde{M_\gamma}$ be the corresponding indecomposable induced module in $\mathrm{mod} \,\Ltil$. Then, the corresponding super lambda length is
      \begin{align*}
         x_\gamma  
&= X^{\mathrm{ind}_{\widetilde{\Lambda}} (\widetilde{M}_\gamma)  } \sum\limits_{\mathbf{e}\, \in \mathbb{Z}^n }\chi (\mathrm{Gr}_{\mathbf{e}}( \widetilde{M}_\gamma ) )  \prod\limits_{i=1}^n \sqrt{x_i}^{\langle S_i , \oplus_j S_j^{m_j} \rangle_{\widetilde{\Lambda}}}\mu_{\mathbf{e}},
      \end{align*}
      
where $\mathbf{e}  = \underline{\mathrm{dim}} ( \bigoplus_j  S_j^{m_j})$ and $\langle - , - \rangle_{\Ltil}$ is the antisymmetrized bilinear form from Definition \ref{def:bilinear form}.  \end{theorem}
  
    \begin{proof}
    First, we simplify the stated expression, that is we prove that
    \begin{align*}
         X^{\mathrm{ind}_{\widetilde{\Lambda}} (\widetilde{M}_\gamma)  } \sum\limits_{\mathbf{e}\, \in \mathbb{Z}^n }\chi (\mathrm{Gr}_{\mathbf{e}}( \widetilde{M}_\gamma ) )  \prod\limits_{i=1}^n \sqrt{x_i}^{\langle S_i , \oplus_j S_j^{m_j} \rangle_{\widetilde{\Lambda}}}\mu_{\mathbf{e}}=X^{\mathrm{ind}_{\Lambda} (M_\gamma)  } \sum\limits_{\substack{N\subseteq \widetilde{M}_\gamma \\\textup{dim}(N)=\mathbf{e}}}  \prod\limits_{i=1}^n \sqrt{x_i}^{\langle S_i , \oplus_j S_j^{m_j} \rangle_{\Lambda}}\mu(N).
         \end{align*}
  In fact, $\textup{ind}_{\widetilde{\Lambda}} (\widetilde{M}_\gamma)=\textup{ind}_{\Lambda} (M_\gamma)$ by Lemma~\ref{lem:index}, $\chi (\mathrm{Gr}_{\mathbf{e}}( \widetilde{M}_\gamma ) ) =1$ if there exists exactly one submodule $N$ of $\widetilde{M}_\gamma$ with dimension vector $\mathbf{e}$ and $0$ otherwise by Lemma~\ref{lemma2:grassmannians computation equal 1}, the antisymmetrized bilinear forms coincide on simple modules by Lemma~\ref{lemma_bilinear_equal} and $\mu_{\mathbf{e}}=\mu(N)$ whenever a submodule $N$ with dimension vector $\textbf{e}$ exists by Notation~\ref{notation_mu}.

    We now prove that the simplified expression agrees with \cite[Theorem 6.2]{musiker21} and hence the formula computes the super lambda length associated to $\gamma$. Let $\mathcal{G}_\gamma$ be the snake graph associated to the arc $\gamma$ and recall that the submodules $N$ of $\widetilde{M}_\gamma$ are in bijection with the double dimer covers $D_N$ of $\mathcal{G}_\gamma$ by Theorem~\ref{thm_lattice_bijection_double}. Combining Lemma~\ref{lemma_index_double} and
     Lemma~\ref{lemma_weight_edge_component}, for a double dimer cover $D_N$ we have that 
    \begin{align*}
        \dfrac{ \prod\limits_{a \in D_N} \sqrt{x_a}}{\textup{cross}(\gamma)}=
        \dfrac{ \textup{wt}_2(D_{\textup{min}})}{\textup{cross}(\gamma)}\dfrac{ \prod\limits_{a \in D_N} \sqrt{x_a}}{\textup{wt}_2(D_{\textup{min}})}
        =X^{\textup{ind}_{\Lambda}(M_\gamma)} \prod\limits_{i=1}^n \sqrt{x_i}^{\langle S_i , \oplus_j S_j^{m_j}\rangle_\Lambda},
    \end{align*}
    where $\underline{\textup{dim}}\,(N)=\underline{\textup{dim}}(\oplus_j S_j^{m_j})$.
    Moreover, by Lemma~\ref{lemma_tethas_behave}, we have that
    \begin{align*}
    \mu (N)=\prod_{C\in c(D_N)} \wt(C).
    \end{align*}
    Hence, taking the sum over all the possible submodules $N$ of $\widetilde{M}_\gamma$ or over all possible double dimer covers $D_N$ of $\mathcal{G}_\gamma$ respectively in the two formulae, we see that the second expression in the statement coincides with \cite[Theorem 6.2]{musiker21}.
    \end{proof}

\begin{remark}\label{remark_simplified_form}
    As pointed out in the proof of Theorem~\ref{thm_superCC}, the expression for the super lambda length of an arc $\gamma$ in a triangulated polygon can be simplified to
      \begin{align*}
         x_\gamma  
&= X^{\mathrm{ind}_{\Lambda} (M_\gamma)  } \sum\limits_{\substack{N\subseteq \widetilde{M}_\gamma \\\textup{dim}(N)=\mathbf{e}}}  \prod\limits_{i=1}^n \sqrt{x_i}^{\langle S_i , \oplus_j S_j^{m_j} \rangle_{\Lambda}}\mu(N),
      \end{align*}
      where $\mathbf{e}  = \underline{\mathrm{dim}} ( \bigoplus_j  S_j^{m_j})$.
      Note that in this version both the index and the antisymmetrized bilinear form are computed over the algebra $\Lambda$.
\end{remark}

\begin{definition}\label{defn_superCC} Let $\Ltil=\Lambda\otimes_K K[\epsilon]/(\epsilon^2)$ where $\Lambda$ is a Jacobian algebra coming from a triangulation $T$ (with no internal triangles) of an $(n+3)$-gon and $\mathcal{A}$ is the super algebra associated to $T$. 
    The \emph{cluster character} of $E$, where $E$ is either an induced $\Ltil$-module or a shifted projective, 
    is defined as the map $CC$ with values in $\mathcal{A}$ as follows.
    \begin{itemize}
        \item If $E=\widetilde{M}_\gamma$ is an indecomposable induced module in $\module\Ltil$, then

        \[CC(E):=X^{\mathrm{ind}_{\widetilde{\Lambda}} (E)  } \sum\limits_{\mathbf{e}\, \in \mathbb{Z}^n }\chi (\mathrm{Gr}_{\mathbf{e}}( E ) )  \prod\limits_{i=1}^n \sqrt{x_i}^{\langle S_i , \oplus_j S_j^{m_j} \rangle_{\widetilde{\Lambda}}}\mu_{\mathbf{e}}, \]

        where $\mathbf{e}  = \underline{\mathrm{dim}} ( \bigoplus_j  S_j^{m_j})$. That is, $CC(E)=x_\gamma$ as computed in Theorem~\ref{thm_superCC}.
        \item If $E=P_i[1]$ is an indecomposable shifted projective as in Section~\ref{Sec:tensoring}, then $CC(E):=x_i$.

        \item If $E=\oplus_{i=1}^r E_i$, where each $E_i$ is either an indecomposable induced module in $\module\Ltil$ or an indecomposable shifted projective, then $CC(E):=\prod_{i=1}^r CC(E_i)$.
    \end{itemize}
\end{definition}

\begin{remark}
    It would be tempting to define the cluster character of a decomposable induced module using a similar formula to the one for an indecomposable one. However, as we show in the following example, this would not give a multiplicative $CC$-map.
    
    Considering Example~\ref{eg_A_2_sectional}, we have that

    \[CC(\begin{smallmatrix}
        1\\1
    \end{smallmatrix})=x_1^{-1}(1+\sqrt{x_1} \theta_2\theta_1+x_1),
    \]

    where the summands correspond to the submodules $0$, $1$ and $\begin{smallmatrix}
        1\\1
    \end{smallmatrix}$ respectively, and $\theta_2\theta_1=\mu(1)=\mu_{\mathbf{(1,0)}}$. Then, noting that $\mu(1)^2=0$, we have that 

    \[CC(\begin{smallmatrix}
        1\\1
    \end{smallmatrix})\cdot CC(\begin{smallmatrix}
        1\\1
    \end{smallmatrix})=
    x_1^{-2}(1+2\sqrt{x_1} \theta_2\theta_1+2 x_1+2x_1\sqrt{x_1}\theta_2\theta_1+x_1^2).
    \]

    On the other hand, applying a formula involving the Euler characteristic to the decomposable module $\begin{smallmatrix}
        1\\1
    \end{smallmatrix} \oplus \begin{smallmatrix}
        1\\1
    \end{smallmatrix}$ would give us a term with coefficient $3$ corresponding to $\mathbf{e}=(2,0)$, and hence $CC(\begin{smallmatrix}
        1\\1
    \end{smallmatrix}\oplus \begin{smallmatrix}
        1\\1
    \end{smallmatrix})\neq CC(\begin{smallmatrix}
        1\\1
    \end{smallmatrix})\cdot CC(\begin{smallmatrix}
        1\\1
    \end{smallmatrix})$. This problem is caused by the fact that $\mathbf{e}=(2,0)$ corresponds both to the submodule $\begin{smallmatrix}
        1\\1
    \end{smallmatrix}$ with $\mu$-term equal to 1 and the submodule $1\oplus 1$ with $\mu$-term equal to 0.

    Alternatively, we could define the cluster character of a decomposible induced module $E=\oplus_{i=1}^r \widetilde{M_i}$ as 

    \[CC(E)=X^{\mathrm{ind}_{\widetilde{\Lambda}} (E)  } \sum\limits_{\substack{(N,\iota)\\ \iota: N\hookrightarrow E}}  \prod\limits_{i=1}^n \sqrt{x_i}^{\langle S_i , \oplus_j S_j^{m_j} \rangle_{\widetilde{\Lambda}}}\mu(N), \]
    
    where $\underline{\mathrm{dim}}(N)  = \underline{\mathrm{dim}} ( \bigoplus_j  S_j^{m_j})$. Here the sum runs over submodules $N$ of $E$ together with a choice of an embedding $\iota: N\hookrightarrow E$ up to isomorphisms.  Submodules $N$ of $E$ are sums of submodules $N_i$'s of $\widetilde{M_i}$'s. Note that some of the $N_i$'s might be submodules of more than one summand of $E$ and that, as in the above example, the same dimension vector could correspond to different submodules of $E$. This alternative formula takes both of these facts into account. Moreover, it could easily be generalised to allow summands of $E$ to be shifted projectives.
\end{remark}

Our CC-map recovers the combinatorial formula for the super lambda lengths and since super lambda lengths respect super Ptolemy relations, we obtain the following result.

\begin{corollary}\label{corollary_super_CC_ptolemy}
 Let $\Ltil=\Lambda\otimes_K K[\epsilon]/(\epsilon^2)$ where $\Lambda$ is a Jacobian algebra coming from a triangulation $T$ (with no internal triangles) of an $(n+3)$-gon. 
If,  for $i=1,\dots, 6$, $\gamma_i$ are arcs as in Figure~\ref{fig:corollary_super_ptolemy}, let $\widetilde{M_i}$ be the corresponding indecomposable induced module in $\mathrm{mod} \,\Ltil$ or the corresponding shifted projective $P_i[1]$ if $\gamma_i$ is in $T$, then 
 \[CC(\widetilde{M_1})\cdot CC(\widetilde{M_2})= CC(E)+CC(E')+\sqrt{CC(E'')} \sigma \theta\] 
 where $E=\widetilde{M_3}\oplus \widetilde{M_5}$, $E'=\widetilde{M_4}\oplus \widetilde{M_6}$, $E''=E\oplus E'$, $\sigma$ and $\theta$ are the $\mu$ invariants associated to the triangles in the figure.

\begin{figure}[h]
    \centering
 \begin{tikzpicture}
 \begin{scope}
   \node[draw,minimum size=3cm,regular polygon,regular polygon sides=3] (a) at (0,0) {};
  \draw (a.side 3) node[right]{$\gamma_4$};
  \draw (a.side 2) node[below right]{$\gamma_1$};
  \draw (a.side 1) node[left]{$\gamma_3$};
  \draw (a.center) node[right]{$\theta$};
  \node[rotate=180, draw,minimum size=3cm,regular polygon,regular polygon sides=3,anchor= side 2] (b) at (a.side 2) {};
  \draw (b.side 3) node[left]{$\gamma_6$};
  \draw (b.side 1) node[right]{$\gamma_5$};
  \draw (b.center) node[right]{$\sigma$};
  \draw (a.corner 1)--node[above left] {$\gamma_2$}(b.corner 1);
  \draw (a.corner 3)--(a.corner 2) node[currarrow,
    pos=0.4, 
    xscale=1,
    sloped]{};
    \draw (b.corner 1)--(a.corner 1) ;
 \end{scope}
 \end{tikzpicture}
 \caption{The CC-map respects Ptolemy relations.}\label{fig:corollary_super_ptolemy}
\end{figure}
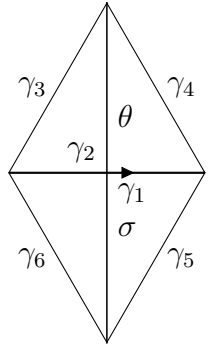

\end{corollary}

We conclude by illustrating our result in the running example.

\begin{example}
    Consider the triangulated pentagon from Figure~\ref{fig:triang_polygon_eg} with arc $\gamma$ indicated in red. The corresponding indecomposable $kA_2$-module is $M_\gamma=\begin{smallmatrix}1\\2\end{smallmatrix}$ with indecomposable induced $\widetilde{kA_2}$-module
$\widetilde{M_\gamma}=\begin{smallmatrix}&1&\\1&&2\\&2&\end{smallmatrix}$, see Example~\ref{example_FN}. The submodule lattice of $\widetilde{M_\gamma}$ is illustrated in Figure~\ref{fig:SG_SM}. We apply the simplified formula from Remark~\ref{remark_simplified_form} to compute the super lambda length $x_\gamma$.
First, note that $M_\gamma=\begin{smallmatrix}1\\2\end{smallmatrix}=I_2$ is an injective $kA_2$-module. Hence $\mathrm{ind}_{kA_2} (M_\gamma)=[0]-[I_2]$ and
\begin{align*}
   X^{\mathrm{ind}_{kA_2} (M_\gamma)}=\frac{1}{x_2}.
\end{align*}
The only nonzero Ext$^1$-group between simples is the $1$-dimensional group $\textup{Ext}^1_{kA_2}(S_1,S_2)$. Hence, applying the definition, we have
\begin{align*}
    \langle S_1,S_1\rangle_{kA_2} = \langle S_2,S_2\rangle_{kA_2} =0,\, \langle S_1,S_2\rangle_{kA_2} = -1, \, \langle S_2,S_1\rangle_{kA_2} =1.
\end{align*}
As shown in Example~\ref{eg_A_2_sectional}, $\mu(\begin{smallmatrix}1&&2\\&2&\end{smallmatrix})=\theta_2\theta_1$. Similarly, one can compute $\mu(\begin{smallmatrix}1\\2\end{smallmatrix})=\theta_3\theta_1$ and $\mu(\begin{smallmatrix} 2\end{smallmatrix})=\theta_3\theta_2$. Since the remaining three submodules $N$ of $\widetilde{M_\gamma}$ are induced modules, we have that $F(N)=0$ and $\mu(N)=1$. Applying the formula, we conclude that
\begin{align*}
    CC(\widetilde{M_\gamma})=x_\gamma =\frac{1}{x_2}\Big(\sqrt{x_1}^{\,-2}\sqrt{x_2}^2+\sqrt{x_1}^{\,-2}\sqrt{x_2}\theta_2\theta_1+\sqrt{x_1}^{\,-2}+\sqrt{x_1}^{\,-1}\sqrt{x_2}\theta_3\theta_1+\sqrt{x_1}^{\,-1}\theta_3\theta_2+1\Big).
\end{align*}
Simplifying and reordering the above expression, we can observe that, as expected, it coincides with the expression computed in Example~\ref{eg_double_dim} using Musiker, Ovenhouse and Zhang's formula. 
\end{example}

\bibliographystyle{alpha}
\bibliography{main.bib}

\end{document}